%% file: main_arxiv.tex
\documentclass[12pt,a4paper]{article}

\input{macros_arxiv}

\numberwithin{equation}{section}

\title{Convergence of the non-staggered Nessyahu-Tadmor scheme for coupled systems of one-dimensional nonlocal balance laws}
\author{%
Sanjibanee Sudha\footnotemark[1],
 Jan Friedrich\footnotemark[2], and Samala Rathan\footnotemark[1] 
 }

\renewcommand{\rv}[1]{\textcolor{black}{#1}}
\begin{document}

\footnotetext[2]{During the time of this work: RWTH Aachen University, Institute of Applied Mathematics, 52064 Aachen, Germany; now: Chair of Optimal Control, Department for Mathematics, School of Computation, Information and Technology, Technical University of Munich, Boltzmannstraße 3, 85748 Garching b. Munich, Germany, jan.friedrich@cit.tum.de}
\footnotetext[1]{Department of Humanities and Sciences, Indian Institute of Petroleum and Energy, Visakhapatnam, Andhra Pradesh, India-530003 (\{sudhamath21,rathans.math\}@iipe.ac.in)}
\maketitle

\begin{abstract}
\justify
We derive a second-order accurate, non-staggered central scheme based on the well-known Nessyahu-Tadmor scheme to approximate solutions of coupled systems of nonlocal balance laws.
We show that the approximate solutions stay bounded by an exponential $L^\infty$ bound in time.
Under linearity assumptions on the flux and source terms the approximate solutions converge weakly-$*$ to weak solutions of the nonlocal balance laws.
Assuming stronger regularity, in particular on the convolution kernel, we show strong convergence towards entropy weak solutions in the nonlinear case.
Numerical examples validate our \rv{results} and demonstrate its applicability to various systems of nonlocal problems.
\end{abstract}
\medskip

\section{Introduction}
Systems of balance laws are of great importance in the modeling and comprehension of complex physical phenomena, as they are capable of describing the interaction of multiple quantities that evolve over time.
In recent years, there has been a growing interest in nonlocal extensions of well-known hyperbolic balance laws, e.g. \cite{aggarwal2015nonlocal,aggarwal2024well,friedrich2020nonlocal,bayen2022multilane,BlandinGoatin2016,chiarello2018global,keimer2018multi,KeimerPflug2017,friedrich2018godunov,friedrich2020micromacro,betancourt2011nonlocal,amadori2012integro,gottlich2014modeling,bhatnagar2021well} and the references therein.
We also refer to \cite{KEIMER2023} for an overview on recent results.
Similar to \cite{aggarwal2015nonlocal}, a generalized form of a system of $N$ nonlocal balance laws in one spatial dimension can be expressed as follows
\begin{align}\label{eq:generalsystem}
\partial_t\brho(t,x)+\partial_xF(\brho,\omega*\brho)\rv{(t,x)}&=S(\brho,\omega*\brho)\rv{(t,x)},\\
\brho(t,0)&=\brho_0(x),\label{eq:initcond}
\end{align}
with $t\in\mathbb{R}_{+},\, x\in\mathbb{R},\, \brho(t,x)\in\mathbb{R}^{N},\, \omega: \R\to \R^{m\times N}, \,(\omega\ast \brho) (t,x) \in\mathbb{R}^{m},\,F:\R^N\times \R^m\to \mathbb{R}^{N},\, S:\R^N\times \R^m\to \R^N,$ suitable initial data $\brho_0:\R\to \R^N$.
\rv{
In this work, we assume that the system \eqref{eq:generalsystem} is only coupled through the nonlocal term or the source, i.e., for $k=1,\dots,N$ the equation \eqref{eq:generalsystem} can be expressed as
\begin{align}\label{eq:generalsystemsimple}
    \partial_t \rho^k(t,x)+\partial_x F_k(\rho^k,\omega \ast \brho)(t,x)=S_k(\brho,\omega \ast\brho)(t,x),
\end{align}
with $F_k:\R\times \R^m\to \R$ and $S_k: \R\times \R^m\to \R$.}
The nonlocal term $\omega\ast\brho,$ \rv{\eqref{eq:generalsystem} and \eqref{eq:generalsystemsimple}, respectively, is for $\ell=1,\dots,m$} given by, 
\begin{align}\label{eq:Rl}
     \rv{R^\ell(t,x)}\coloneq (\omega\ast\brho)_{\ell}(t,x)=\sum_{k=1}^{N}\underbrace{\int_{\mathbb{R}}\omega^{\ell,k}(\rv{y-x})\rho^{k}(t,y)dy}_{\eqcolon R^{\ell,k}(t,x)},\quad \rv{ (t,x)\in \mathbb{R}_{+}\times \mathbb{R}.}
\end{align}
Here, \rv{$\omega$ is a suitable kernel function.} We use boldface variables for vector expressions and to shorten the notation we define $\bm{R}(t,x)=(\omega \ast \brho)(t,x)$.  
It should be noted that \rv{the study of existence and uniqueness for equation \eqref{eq:generalsystem} is, as in the local case, a challenging task. Unlike the local case, this area has not yet been extensively explored.}
Nevertheless, for the case of a system of nonlocal conservation laws and a coupling of those similar to \eqref{eq:generalsystemsimple}, rather general results can be found in \cite{aggarwal2015nonlocal,aggarwal2024well} and results for specific flux functions in \cite{friedrich2020nonlocal,bayen2022multilane,bhatnagar2021well,friedrich2020micromacro}.
Similar, to the local setting the scalar case $N=1$ is well understood, see e.g. \cite{chiarello2018global,KeimerPflug2017,amorim2015numerical}.
Especially in the scalar case, many applications can be modeled by \eqref{eq:generalsystem}, such as sedimentation processes \cite{betancourt2011nonlocal}, granular material dynamics \cite{amadori2012integro}, supply chains \cite{goettlich2010supplychains,colombo2011control}, material flow on a conveyor belt (two spatial dimensions) \cite{gottlich2014modeling}, crowd dynamics (two spatial dimensions) \cite{colombo2012class, colombo2011control}, and particularly traffic flows \cite{BlandinGoatin2016, friedrich2018godunov, chiarello2018global}.
Examples of systems with a nonlocal term in the flux as in \eqref{eq:generalsystem} include traffic flow models with a second equation describing the momentum \cite{friedrich2020micromacro}, multiclass models \cite{chiarello2019multiclass}, or multilane models \cite{friedrich2020nonlocal,bayen2022multilane}, multi-population crowd dynamics \cite{colombo2012nonlocal}, a nonlocal extension of the Keyfitz-Kranzer model \cite{aggarwal2015nonlocal,aggarwal2024well}, or more recently, pressureless and nonlocal Euler equations with relaxation \cite{bhatnagar2021well}. \par 
Numerical schemes for nonlocal conservations are typically based on ideas for schemes used to approximate local hyperbolic balance laws.
Especially for the scalar case, first-order methods are reliable, robust and well understood, see e.g.  \cite{amorim2015numerical,chiarello2018global,friedrich2023numerical, huang2024asymptotic}.
Similar to the local schemes, finite-volume methods are considered, in particular Godunov-type schemes (upwind method) \cite{godunov1959}, and Lax-Friedrichs schemes (LxF) (central scheme) are widely used for nonlocal conservation laws.
Note that in contrast to the local setting no Riemann solvers are known, \rv{since due to the nonlocality the Riemann problem can generally not be solved explicitly. However,} for certain structures of the flux Godunov/upwind-type schemes can still be derived, e.g. \cite{friedrich2018godunov,friedrich2023numerical}. Central schemes avoid this by integrating over space-time control volumes where each Riemann fan is contained within its own volume.
This results in a more dissipative behavior, but their advantage is the simplicity and direct applicability to different flux functions, such as those given by the system \eqref{eq:generalsystem}.
In contrast, Godunov methods are quite difficult to derive for \eqref{eq:generalsystem}.
However, for special structures of the flux, e.g. similar to \eqref{eq:generalsystemsimple}, they can be derived in a straightforward way  \cite{aggarwal2024well,friedrich2020nonlocal,chiarello2019multiclass}.
\par
Apart from first-order schemes, higher-order schemes for nonlocal scalar conservation laws have been developed in \cite{ChalonsGoatinVillada2018, friedrich2019maximum, gd2023convergence}. \rv{In \cite{manoj2024positivity}, a second-order MUSCL-type scheme with Runge-Kutta time integration and LxF flux for systems of nonlocal conservation laws has been investigated.}
All these works consider a higher-order reconstruction within each cell and an appropriate flux function at the cell interface.
Furthermore, high-resolution schemes based on the ideas of central schemes have been considered in \cite{GoatinScialanga2016, belkadi2022non, belkadi2024unstaggered,kurganov2009non}.
These ideas go back to the central differencing high resolution scheme developed by Nessyahu and Tadmor in \cite{nessyahu1990non}.
To reduce the dissipative behavior of first-order central schemes they used MUSCL-type interpolants and maintained the simplicity of the Riemann-solver free approach by evolving the cell-averages over staggered cells. 
The resulting and so-called central Nessyahu-Tadmor (NT) scheme is a second-order accurate method.
The staggered grids can also be transformed into a non-staggered grid by following \cite{jiang1998high}.
The main property of non-staggered schemes is further simplicity since they avoid the need to alternate between two staggered grids, which is particularly challenging near the boundaries. 
\rv{For scalar nonlocal conservation laws the NT scheme was initially proposed in \cite{goatin2016speed}. This idea was later extended in \cite{belkadi2024unstaggered, belkadi2022non, belkadi2024class, belkadi2025unstaggered}, where the focus was primarily on computational investigations and numerical performance. Nevertheless, a rigorous theoretical framework for analyzing the NT scheme even in the context of scalar nonlocal conservation laws has not yet been fully established.}
\par\rv{Hence, the aim of this work is twofold: First, we develop the NT scheme using non-staggered grids for a general nonlocal system such as \eqref{eq:generalsystemsimple} and second, we establish theoretical convergence results for this scheme. 
From this we also obtain existence results for \eqref{eq:generalsystemsimple}.
Therefore, in Section \ref{sec:staggered central}, we present our main results and extend the NT schemes considered in \cite{GoatinScialanga2016, belkadi2022non} to systems of nonlocal balance laws. Thereby, we use the ideas of \cite{jiang1998high} to consider them on a non-staggered grid. 
Additionally, we prove an $L^\infty$ bound of the NT scheme under minimal assumption.
In Section \ref{weakconvergence}, we prove the weak-$*$ convergence of approximate solutions towards weak solutions for a linear setting.
Then in Section \ref{sec:BV}, we use stricter regularity assumptions to prove the convergence towards weak entropy solutions in the nonlinear case. Finally, Section \ref{sec:numerical} presents several numerical examples that demonstrate the second-order accuracy of the considered schemes and their applicability to various modeling equations.}
\section{Second-order central scheme} \label{sec:staggered central}
\rv{In this section, we present the main results along with the necessary assumptions ensuring the convergence of the NT scheme and the existence of solutions to \eqref{eq:generalsystemsimple}. Then, we describe the adapted NT numerical scheme developed for nonlocal balance laws in detail and derive the $L^\infty$ bound, which holds under rather general assumptions.}

\subsection{\rv{Main results}}
We study the system \eqref{eq:generalsystemsimple} which assumes that the system is only coupled through the nonlocal term or the source. 
Several relevant applications can already be modeled by \eqref{eq:generalsystemsimple}, see also Section \ref{sec:numerical}.
Nevertheless, we note that the derivation of the NT scheme in Section \ref{subsec1} can also be applied to the balance law \eqref{eq:generalsystem}, but it is more challenging to derive a CFL condition and further analytical results.
In addition, we are interested in the convergence of the NT scheme which will provide existence results for \eqref{eq:generalsystemsimple}, too. 
Studying the uniqueness of solutions is beyond the scope of this work.\\
The following set of assumptions will ensure an $L^\infty$ bound for sufficiently small time horizons, cf. Section \ref{sec:posinf}.
\begin{assumption}\label{set1}
    For $k\in\{1,\dots,N\}$, we consider the following assumptions on the input data
    \begin{itemize}
    \item \emph{Flux function}: there exist  $\rho_M>\rho_m\geq 0$ such that for all $\bm{R}\in\R^{m}$ it holds $F_k(\rho_m,\bm{R})=0$, $F_k\in \text{Lip}([\rho_m,\rho_M]^{m+1})$ with Lipschitz constant $L_F$,
    \item \emph{Source term}: $S_k(\brho_m,\bm{R})=0$ for all $\bm{R}\in \R^m$ with $\brho_m=(\rho_m,\dots,\rho_m)^T$, $S_k\in \text{Lip}([\rho_m,\rho_M]^{N+m})$  with Lipschitz constant $L_S$,
    \item \emph{Initial data}: $\brho_0 \in L^{\infty}(\R;[\rho_m,\rho_M)^N)$,
    \item \emph{Kernel function}: for $\ell=1,\dots,m$, $\omega^{\ell,k}\in L^1\cap \text{BV}(\mathbb{R};\mathbb{R}_+)$, $0\in \text{supp}(\omega^{\ell,k})$ and $\int_\R \omega^{\ell,k}(x)dx=1$.
    \end{itemize}
\end{assumption}
For simplicity, we assumed the existence of $\rho_m<\rho_M$, which are uniform over all $k\in \{1,\dots,N\}$ and a uniform integral of each kernel. We note that the analysis below can easily be extended to the case that $\rho_m$ and $\rho_M$ are different for each $k\in \{1,\dots,N\}$ or $L^1-$kernels which do not have a unit integral. Furthermore, the assumption that zero is included in the kernel's support is only made to simplify the numerical discretization.\\
Here, solutions need to be understood in a weak sense. In particular, weak solutions of \eqref{eq:generalsystemsimple} are defined similarly to \cite[Definition 1]{KEIMER2023} and \cite[Definition 5.2]{eymard2000finitevolume}.
\begin{definition}\label{def:weak}
A function $\brho\in
L^\infty((0,T)\times \R;\R^N)$ is called a weak solution to \eqref{eq:generalsystemsimple} with $\brho_0\in L^\infty(\R;\R^N)$, if for $k=1,\dots,N$
\begin{align*}\int_{\R}\rho^k_0(x)\phi(0,x)dx&+\int_{0}^{T}\int_{\R}\rho^k(t,x)\phi_t(t,x)dxdt\\
&+\int_0^{T}\int_{\R}F_k(\rho^{k},\omega\ast \brho)\phi_x(t,x)dxdt+\int_0^{T}\int_{\mathbb{R}}\phi(t,x)S_k(\brho,\omega\ast\brho)dxdt=0\end{align*}
for all $\phi\in C_0^1((-42,T)\times\R;\R)$.
\end{definition}
To prove a first convergence result, we focus on the weak-$*$ convergence of an $L^\infty$ bounded sequence in a linear setting. 
To apply this, we need additional assumptions on the flux function and source term:
\begin{assumption}~\label{set2}
Additionally to the Assumption \ref{set1}, we assume for $k\in\{1,\dots,N\}$
    \begin{itemize}
    \item \emph{Flux function}:  $F_k(\rho^k,\bm{R})=\rho^k V_k(\bm{R})$ with $V_k\in \text{Lip}([\rho_m,\rho_M]^{m})$,
    \item \emph{Source term}: $\partial_{\rho^j}S_k(\brho, \bm{R})=\text{const.}$ for $j=1,\dots,N$, i.e., $S_k$ is a linear function in each $\rho^j$.
\end{itemize}
\end{assumption}
This linearity in the source term and flux function allows us to obtain the following convergence result for the NT scheme constructed in Section \ref{subsec1}.
\begin{theorem}\label{thm:weakconv}
Let the Assumption \ref{set2} hold and let $\brho_\Delta $ be the sequence of approximate solutions generated by the NT scheme with the slopes \eqref{eq:slopesmodconv} satisfying the CFL conditions \eqref{eq:CFLmaxNT}.   
Then, as $\Delta x, \Delta t \to 0$, there exists a sufficiently small time $T^*\in (0,T]$ and a function $\brho \in L^{\infty}((0,T^*)\times\R)$ such that, up to a subsequence,
\[
\brho_\Delta \rightarrow \brho \quad \text{weakly-\(*\) in}\,\,L^{\infty}((0,T^*)\times\R\bigr).
\]
Moreover, $\rho$ is a weak solution of the nonlocal model \eqref{eq:generalsystemsimple} in the sense of Definition \ref{def:weak}.
\end{theorem}
Let us now comment on this result and the assumptions made in Assumption \ref{set2}.
On the one hand the linearity assumptions allow us to apply the weak-$*$ convergence. On the other hand the linearity in the flux function usually guarantees the uniqueness of weak solutions for nonlocal balance laws.
For example it is well-known, that in the scalar case weak solutions are unique, and no additional entropy condition is required as discussed in \cite{keimer2018multi,gd2023convergence}. Therefore, the second-order scheme converges directly to the unique weak solution without imposing additional admissibility criteria. We expect a similar result to hold in the system case. In \cite{bayen2022multilane} a multilane traffic flow model is studied which fits into the setting of Assumption \ref{set2}, if we consider slightly stronger assumptions on the initial data, e.g. BV bounded initial data, cf, \cite[Assumption 2.2]{bayen2022multilane}. 
Again the uniqueness of weak solutions is proven in \cite{bayen2022multilane}, such that the numerical approximation of the NT scheme converges to the correct solution. 
In addition, we note that, the uniqueness of weak solutions $\brho$ implies that the whole sequence $\brho_{\Delta}$ converges (under the Assumption \ref{set2}).
Nevertheless, the uniqueness of weak solutions of nonlocal balance law as \eqref{eq:generalsystem} or \eqref{eq:generalsystemsimple} is still an open problem and out of scope for this work.\\
For the general case, i.e., the flux function is nonlinear in the local density, we do not expect the uniqueness of weak solutions.
Without a source term the existence and uniqueness of weak entropy solutions of \eqref{eq:generalsystemsimple} has been shown in, e.g., \cite{aggarwal2015nonlocal} and \cite{goatin2024well}, respectively.
In both works, the authors assume additional (stronger) assumptions on the flux $F$ and consider more regular kernels.
Weak entropy solutions are thereby intended in the following sense, similar to \cite{goatin2024well} and \cite{eymard2000finitevolume}.
\begin{definition}\label{def:weakentropy}
    A function $\brho:[0,T]\times \R\to\R^N$ with $\brho\in C([0,T];L^1(\R;\R^N))\cap L^\infty((0,T);BV(\R;\R^N))$ is a weak entropy solution to \eqref{eq:generalsystemsimple}, if it satisfies for all $k\in\{1,\dots,N\}$ and any constant $\zeta \in\R$
\begin{align*}
0\leq&\int_\R |\rho^k_0(x)-\zeta|\phi(0,x) dx +
   \int_{0}^T\int_\R |\rho^k-\zeta|
    \phi_t dx dt\\
    &+\int_{0}^T\int_\R\textnormal{sign}(\rho^k-\zeta)\Big((F_k(\rho^k,\omega\ast\brho)-F_k(\zeta,\omega\ast\brho)\phi_x
    -(\partial_x F_k(\zeta,\omega\ast\brho)+S_k(\brho,\omega\ast\brho))\phi \Big)dx dt 
\end{align*}
for all $\phi\in C_0^1([0,T)\times\R;\R_+)$.
\end{definition}
Stronger regularity assumptions, in particular on the kernel, allow us to compute the NT scheme slightly more accurate, such that the convergence result in the general and nonlinear case can be proven.
\begin{assumption}\label{set3}
Additionally to the Assumption \ref{set1}, we assume for $k\in\{1,\dots,N\}$
\begin{itemize}
    \item \emph{Flux function}: $F_k(\rho^k,\bm{R}) = g_k(\rho^k)V_k(\bm{R}),$ with $g_k\in \text{Lip}([\rho_m,\rho_M])$ with Lipschitz constant $L_g$ and $V_k \in C^2([\rho_m,\rho_M]^m;\R)$,
    \item \emph{Initial data}: $\brho_0\in L^{1}\cap BV(\R;[\rho_m,\rho_M)^N)$,
    \item \emph{Kernel function}: $\omega^{\ell,k}$ is of compact support with $\omega^{\ell,k}\in C^1(\text{supp}(\omega^{\ell,k}),\R_+)$.
\end{itemize}
\end{assumption}
Note that the Assumption \ref{set3} implies $\rho_m=0$. Now, we present our second convergence result.
\begin{theorem}\label{thm:L1locconv}
Let the Assumption \ref{set3} hold and let $\brho_\Delta $ be the sequence of approximate solutions generated by the NT scheme with the slopes \eqref{eq:slopesv2} satisfying the CFL condition \eqref{eq:CFLmaxNT}.   
Then, as $\Delta x, \Delta t \to 0$, there exists a sufficiently small time $T^*\in (0,T]$ and a function $\brho \in L^{\infty}\cap BV((0,T^*)\times\R)$ such that $\brho_\Delta$ converges, up to a subsequence, in $L^1_\textnormal{loc}$ to $\brho$ and $\brho$ is an entropy weak solution of the nonlocal model \eqref{eq:generalsystemsimple} in the sense of Definition \ref{def:weakentropy}.
\end{theorem}

\subsection{Non-staggered Nessyahu-Tadmor scheme}\label{subsec1}
To derive approximate solutions to \eqref{eq:generalsystem} and \eqref{eq:generalsystemsimple} we discretize space and time by an equidistant grid, where $\Delta x$ is the step size in space and $\Delta t$ is the step size in time. So, $t^n=n\Delta t$ with $n\in \mathbb{N}$ is the time grid and $x_{j}=j\Delta x$ with $j\in\Z$ the space grid. 
The cell interfaces are $\rv{x_{j-1/2}=(j-1/2)\Delta x}$ and $\rv{x_{j+1/2}=(j+1/2)\Delta x}$.\\
\rv{
Although the kernel functions might have infinite support, cf. Assumption \ref{set1}, we will discuss the numerical discretization for the case that each kernel $\omega^{\ell,k}:\R\to \R_+$ is of compact support.
In the latter case the boundaries of the support need to be approximated accurately. On the contrary, for kernels with infinite support, the discretization simplifies, as at least one of the boundary terms is not present.
Hence, we assume that the support of each kernel is given by the interval $[\eta_1^{\ell,k},\eta_2^{\ell,k}]$ with $\eta_1^{\ell,k}\leq 0 \leq \eta_2^{\ell,k},\ \eta_1^{\ell,k}\neq \eta_2^{\ell,k},$ for $k=1,\dots,N$ and $\ell=1,\dots,m$.
}
For simplicity, we assume that the size of the space step $\Delta x$ can be chosen such that for $\ell=1,\dots,m$ and $k=1,\dots,N$ every $N_1^{\ell,k}:=|\eta_1^{\ell,k}|/\Delta x$ and $N_2^{\ell,k}:=\eta_2^{\ell,k}/\Delta x$ are natural numbers or zero.
\begin{remark}
    We note that for the case where all kernels have the same support and are either solely forward/backward oriented or symmetric, the previous assumption holds.
    Otherwise, we need to set $N_1^{\ell,k}:=\left\lceil |\eta_1^{\ell,k}|/\Delta x \right\rceil$ and $N_2^{\ell,k}:=\left\lceil\eta_2^{\ell,k}/\Delta x\right\rceil$.
    This results in minor modifications when computing the nonlocal terms, which are detailed below.
\end{remark}
The initial data \eqref{eq:initcond} are approximated component-wise for $k=1,\dots,N$ and $j\in \mathbb{Z}$ by
$$\rho_{j}^{k,0}=\frac{1}{\Delta x}\int_{x_{j-1/2}}^{x_{j+1/2}}\rho^k_{0}(x)dx.$$
We assume that the cell-averages of the solution are given at a time $t^n$ by
$$\rho_{j}^k(\rv{t^n})=\frac{1}{\Delta x}\int_{x_{j-1/2}}^{x_{j+1/2}}\rho^k(t^n,x)dx.$$
To evolve the cell-averages to the next time level $t^{n+1}$, we start by \rv{constructing} a piecewise linear function in each cell, component-wise for $k=1,\dots,N$ by
\begin{equation}\label{eq:piecewise}
    \tilde{\rho}^{k}_{j}(t,x)={\rho}^{k}_{j}(t)+{s^{k}_{j}}(x-x_j) \,\,\text{for}\,\, x\in[x_{j-1/2},x_{j+1/2}).
\end{equation}
One way to compute the \rv{slope} in each component is to use the minmod limiter
\begin{equation}\label{eq:slope for X}
    s_{j}^{k}=\mm\bigg(\frac{\rho_{j}^{k}(t)-\rho_{j-1}^{k}(t)}{\Delta x},\frac{\rho_{j+1}^{k}(t)-\rho_{j}^{k}(t)}{\Delta x}\bigg),
\end{equation}
where the $\mm$ function is defined as
\begin{equation}
\mm(a, b) = 
\begin{cases} 
      a, & \text{if } |a| < |b| \text{ and } a \cdot b > 0, \\
      b, & \text{if } |b| < |a| \text{ and } a \cdot b > 0, \\
      0, & \text{if } a \cdot b \leq 0. 
\end{cases}
\end{equation}
The reconstruction \eqref{eq:piecewise} preserves the conservation (here the cell-average $\rho_j^k(t)$ in $[x_j,x_{j+1}]$).
Next, we evolve the cell-averages with the help of the piecewise linear reconstruction \eqref{eq:piecewise} by integrating \eqref{eq:generalsystem} over the domain $[t^n,t^{n+1})\times[x_j,x_{j+1})$ for each $k=1,\dots,N,$
    \begin{align}
 \rho_{j+1/2}^{k,n+1}=&\frac{1}{\Delta x}\int_{x_j}^{x_{j+1}}\tilde{\rho}^{k}(t^n,x)dx-\frac{1}{\Delta x}\int_{t^n}^{t^{n+1}}\bigg(F_k(\rho_{j+1}^k(t),\bm{R}(t,x_{j+1}))-F_k(\rho_{j}^k(t),\bm{R}(t,x_j))\bigg)dt\nonumber\\
    &+\frac{1}{\Delta x}\int_{t^n}^{t^{n+1}}\int_{x_j}^{x_{j+1}} S_k(\tilde \brho(t,x), \bm{R}(t,x)) dx dt, \nonumber\\ 
    =&\frac{1}{\Delta x}\bigg(\int_{x_j}^{x_{j+1/2}}\tilde{\rho}_j^{k}(t^n,x)dx+\int_{x_{j+1/2}}^{x_{j+1}}\tilde{\rho}_{j+1}^{k}(t^n,x)dx\bigg)-\frac{1}{\Delta x}\int_{t^n}^{t^{n+1}}\bigg(F_k(\rho_{j+1}^k(t),\bm{R}(t,x_{j+1}))\nonumber \\
    &-F_k(\rho_{j}^k(t),\bm{R}(t,x_j))\bigg)dt+\frac12\int_{t^n}^{t^{n+1}} \bigg(S_k(\brho_{j+1}(t), \bm{R}(t,x_{j+1}))+S_k(\brho_j(t), \bm{R}(t,x_{j})) \bigg)dt.\nonumber
    \end{align}
    \rv{Note that in the last step we used the trapezoidal rule in space to approximate the source term accurately.}
    Due to the finite speed of propagation we introduce the CFL condition similar to \rv{\cite{GoatinScialanga2016,belkadi2022non,nessyahu1990non}
    \begin{align}\label{eq:CFLstart}
        {\lambda \coloneq \frac{\Delta t}{\Delta x} \leq \frac{1}{2L_F}}.
    \end{align}
    Hence, we can approximate the temporal integrals of the flux and the source term by the mid point-rule to maintain the second-order accuracy}
\begin{equation}\label{eq:system predictor}
\begin{aligned}
\rho_{j+1/2}^{k,n+1}=&\frac{1}{2}(\rho_j^{k,n}+\rho_{j+1}^{k,n})+\frac{\Delta x}{8}(s_j^{k,n}-s_{j+1}^{k,n})-\rv{\lambda}\bigg[F_k(\rho_{j+1}^{k,n+1/2},\bm{R}_{j+1}^{n+1/2})-F_k(\rho_{j}^{k,n+1/2},\bm{R}_j^{n+1/2})\bigg]\\
 &+\frac{\Delta t}{2}\bigg[ S_k(\brho^{n+1/2}_{j+1}, \bm{R}^{n+1/2}_{j+1})+S_k(\brho^{n+1/2}_{j}, \bm{R}^{n+1/2}_{j})\bigg]\,.
\end{aligned}
\end{equation}
The source and flux terms use the same input, which avoids additional computational effort. 
It remains to approximate the values at the time $t^{n+1/2}$ \rv{of  \eqref{eq:system predictor}}.
Thanks to a Taylor series expansion and the balance law \eqref{eq:generalsystem}, we have
\begin{align}\label{eq:systemcorrector}
\rho_{j}^{k,n+1/2}&\approx \rho^{k}(t^n,x_j)+\frac{\Delta t}{2}\rho_{t}^{k}(t^n,x_j)\approx\rho_j^{k,n}+\frac{\Delta t}{2}S_k(\brho_j^{n},\bm{R}_j^{n})-\frac{\Delta t}{2} \sigma^{k}(\brho_j^{n},\bm{R}_j^{n}), \\
\label{eq:slopesv1}
\sigma^{k}(\brho_{j}^{n},\bm{R}_{j}^{n})&=\mm\bigg(\frac{F_k(\rho_{j}^{k,n},\bm{R}_{j}^{n})-F_k(\rho_{j-1}^{k,n},\bm{R}_{j-1}^{n})}{\Delta x},\frac{F_k(\rho_{j+1}^{k,n},\bm{R}_{j+1}^{n})-F_k(\rho_{j}^{k,n},\bm{R}_{j}^{n})}{\Delta x}\bigg).
\end{align}
Note that \rv{for $\ell=1,\dots,m$ each element of the nonlocal term $\bm{R}(t,x)$, cf. \eqref{eq:Rl}, is given by}
$$ R^\ell(t,x)= \sum_{k=1}^N R^{\ell,k}(t,x).$$
Hence, we need to approximate $R^{\ell,k}(t^{n+1/2},x_j)$ \rv{for $j\in\Z$}. To simplify the notation we will neglect the indices $\ell=1,\dots,m$ in the following lines.
This can also be viewed as considering the special case $m=1$.
\rv{To approximate the remaining terms in \eqref{eq:system predictor} we use, as before, a Taylor expansion}
\begin{align}\label{eq:system convolution taylor}
{R}_j^{k,n+1/2}\approx {R}_j^{k,n}+\frac{\Delta t}{2}{R}_t^k(t^n,x_j).
\end{align}
\rv{Here, $R_t^k$ denotes the time derivative of the corresponding nonlocal term.} We approximate the convolution terms \rv{in \eqref{eq:system predictor}--\eqref{eq:systemcorrector}} using the midpoint rule,
\begin{align}
    R^{k}(t^n,x_j)=&\int_{x_{j-N_1^{k}}}^{x_{j+1/2-N_1^{k}}}\tilde{\rho}^k(t^n,y)\omega^{k}(y-x_j)dy+\sum_{l=-N_1^{k}}^{N_2^k-2}\int_{x_{j+1/2+l}}^{x_{j+3/2+l}} \tilde{\rho}^k(t^n,y)\omega^{k}(y-x_j)dy\nonumber\\
 &+\int_{x_{j+N_2^k-1/2}}^{x_{j+N_2^k}}\tilde{\rho}^k(t^n,x)\omega^{k}(y-x_j)dy,\nonumber\\
 =&\frac{\Delta x}{2}\left(\rho_{j-N_{1}^{k}}^{k,n}+\frac{\Delta x}{4}s_{j-N_1^{k}}^{k,n}\right)\omega^{k}\left(\frac{\Delta x}{4}-N_1^{k}\Delta x\right)+\Delta x\sum_{l=-N_{1}^{k}}^{N_2^{k}-2}\rho_{j+l+1}^{k,n}\omega^{k}((l+1)\Delta x)\nonumber\\
 &+\frac{\Delta x}{2}\left(\rho_{j+N_2^{k}}^{k,n}-\frac{\Delta x}{4}s^{k,n}_{j+N_2^{k}}\right) \omega^{k}\left(N_2^{k}\Delta x-\frac{\Delta x}{4}\right).\label{eq:RNT}
 \end{align}
\rv{To approximate $R_t$ in \eqref{eq:system convolution taylor} we use first the balance law \eqref{eq:generalsystemsimple}
$$\partial_t R^k(t,x)=\int_{\R} \omega^k(y-x)\partial_t \rho^k(t,y) dy=\int_{\R} \omega^k(y-x)(S(\brho,\bm{R})-\partial_x F(\rho^k,\bm{R}))(t,y) dy.$$
Then using an appropriate quadrature rule together with the minmod limiter \eqref{eq:slopesv1}}, we apply the same procedure as in \eqref{eq:RNT}.
\begin{equation}\label{eq:system Rt}
\begin{aligned}
  R_t^{k}(t^n,x_j)\approx &\frac{\Delta x}{2}\left(S_k(\brho^n_{{j}-N_1^{k}},\bm{R}_{{j}-N_1^{k}}^n)-\sigma^{k}(\brho^n_{{j}-N_1^{k}},\bm{R}_{{j}-N_1^{k}}^n)\right)\omega^{k}\left(\frac{\Delta x}{4}-N_1^{k}\Delta x\right)\\
 &+\Delta x\sum_{l=-N_1^{k}}^{N_2^k-2}\left(S_k(\brho_{j+l+1}^n,\bm{R}_{j+l+1}^n)-\sigma^{k}(\brho_{j+l+1}^n,\bm{R}_{j+l+1}^n)\right)\omega^k((l+1)\Delta x)\\
 &+\frac{\Delta x}{2}\left(S_k({\brho}^n_{{j}+N_2^{k}},\bm{R}_{{j}+N_2^{k}}^n)-\sigma^{k}({\brho}^n_{{j}+N_2^{k}},\bm{R}_{{j}+N_2^{k}}^n)\right)\omega^k\left(N_2^{k}\Delta x-\frac{\Delta x}{4}\right).
\end{aligned}
\end{equation}
\begin{figure}
    \centering
    \input{GRid_cell_NT_staggered}
    \caption{NT scheme reconstruction from non-staggered to staggered grid for fixed $k\in\{1,\dots,N\}$, compare Figure 3.1 from  \cite{jiang1998high}}\label{fig:1}
\end{figure}
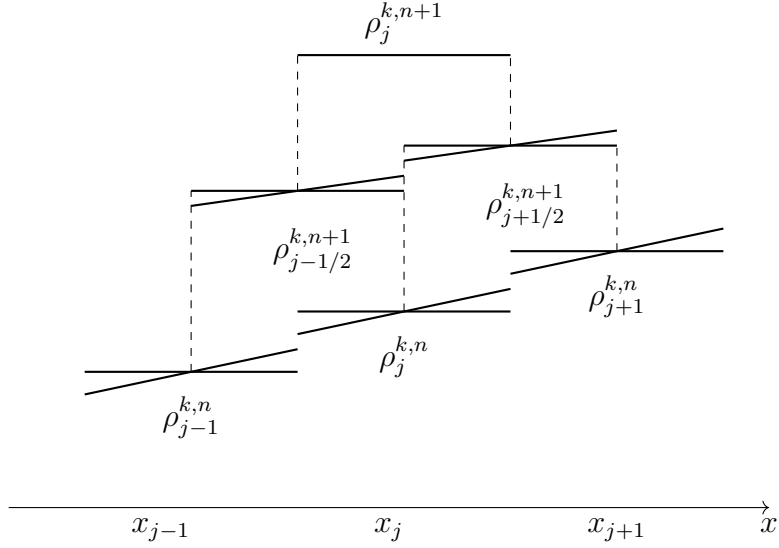
The space derivative of the flux is approximated as before by the minmod limiter, i.e., equation \eqref{eq:slopesv2}.\\
Let us make the following remarks to the approximation of the nonlocal terms in \eqref{eq:RNT} and \eqref{eq:system Rt}.
\begin{remark}\label{rem:firstorderquad}
    \rv{We note that the first interval in \eqref{eq:system Rt} is obtained by the following first-order quadrature rule
    \[\int_a^b f(x)g(x)dx\approx(b-a)f(a) g\left(\frac{a+b}{2}\right).\]
    For the last interval we use a similar quadrature rule and replace $f(a)$ by $f(b)$.}
    On the one hand, a first-order approximation of $R_t^k(t^n,x_j)$ is sufficient to maintain the second-order accuracy of $R_j^{k,n+1/2}$.
    On the other hand, this specific quadrature rule allows us to use the same weights as in \eqref{eq:RNT} and it does not involve any additional (spatial) evaluations of the source term or the slopes, such that we can rely on already computed values in the implementation.
\end{remark}
\begin{remark}
    In contrast to previous works as \cite{belkadi2022non,GoatinScialanga2016}, we approximate the time derivative of the nonlocal term by using the minmod limiter. Since the slopes are already used to compute \eqref{eq:systemcorrector}, this approach is rather simple. 
    Alternatively, one can proceed by applying integration by parts as in \cite{belkadi2022non,GoatinScialanga2016} which requires the derivative of the kernel $\omega$. \rv{We also refer to Section \ref{sec:BV}}.
\end{remark}
\begin{remark}
    In case it is not possible to choose a grid such that all $|\eta_1^k|/\Delta x$ and $\eta_2^k/\Delta x$ are natural numbers, the first and last intervals for the approximation of the integrals need to be changed.
    Instead of computing a second-order accurate approximation over the intervals $[x_{j-N_1^k},x_{j-N_1^k+1/2}]$ and $[x_{j+N_2^k-1/2},x_{j+N_2^k}]$ in \eqref{eq:RNT}, we compute it for $[x_j+\eta_1^k,x_{j-N_1^k+1/2}]$ and $[x_{j+N_2^k-1/2},x_{j}+\eta_2^k]$. 
\end{remark}
\par Now, we convert the staggered second-order scheme \eqref{eq:system predictor}–\eqref{eq:systemcorrector} to \rv{a} non-staggered grid following \cite{jiang1998high}. 
First, we construct a piecewise-linear interpolant through the calculated staggered cell-averages at time $t^{n+1}$ (see Figure \ref{fig:1}) which is, for each $k=1,\dots,N$, given by
\begin{align}\label{eq:system_poly_staggered}
\hat{\rho}_{j+1/2}^{k,n+1}(x)={\rho}_{j+1/2}^{k,n+1}+s_{j+1/2}^{k,n+1}(x-x_{j+1/2}).
\end{align}
The staggered discrete derivatives, $s_{j+1/2}^{k,n+1}$, are given by
\begin{align}\label{eq:system min_staggerd}
s_{j+1/2}^{k,n+1}=\mm\left(\frac{{\rho}_{j+3/2}^{k,n+1}-{\rho}_{j+1/2}^{k,n+1}}{\Delta x},\frac{{\rho}_{j+1/2}^{k,n+1}-{\rho}_{j-1/2}^{k,n+1}}{\Delta x}\right).
\end{align}
Then, the cell-averages at the next time step, ${\brho}_j^{n+1},$ are obtained by averaging the interpolant \eqref{eq:system_poly_staggered}.
This results in a non-staggered corrector scheme for each $k=1,\dots,N$
 \begin{equation}
    \begin{aligned}
    \label{eq:system staggered_final}
\rho_j^{k,n+1}=&\frac{1}{\Delta x}\bigg[\int_{x_{j-1/2}}^{x_j}\hat\rho_{j-1/2}^{k,n+1}(x)dx+\int_{x_j}^{x_{j+1/2}}\hat\rho_{j+1/2}^{k,n+1}(x)dx\bigg],\\
=&\frac{\rho_{j-1}^{k,n}+2\rho_j^{k,n}+\rho_{j+1}^{k,n}}{4}-\frac{\Delta x}{16}\left({s_{j+1}^{k,n}-s_{j-1}^{k,n}}\right)
    -\frac{\Delta x}{8}\left({s_{j+1/2}^{k,n+1}-s_{j-1/2}^{k,n+1}}\right)\\
    &-\rv{\frac{\lambda}{2}}\bigg[F_k(\rho_{j+1}^{k,n+1/2},\bm{R}_{j+1}^{n+1/2})-F_k(\rho_{j-1}^{k,n+1/2},\bm{R}_{j-1}^{n+1/2})\bigg]\\
    &+\frac{\Delta t}{4}\bigg[ S_k(\brho^{n+1/2}_{j+1}, \bm{R}^{n+1/2}_{j+1})+2 S_k(\brho^{n+1/2}_{j}, \bm{R}^{n+1/2}_{j})+S_k(\brho^{n+1/2}_{j-1}, \bm{R}^{n+1/2}_{j-1})\bigg].
    \end{aligned}
\end{equation}
Here, $s_j^{k,n+1}$ and\, $s_{j+1/2}^{k,n+1}$ are the discrete derivatives at time levels $t^n$ and $t^{n+1}$ given in \eqref{eq:slope for X} and \eqref{eq:system min_staggerd}, and $\brho^{n+1/2}$ and $\bm{R}^{n+1/2}$ are predicted at time level $t^{n+1/2}$ according to \eqref{eq:systemcorrector} and \eqref{eq:system convolution taylor}.
\subsection{\texorpdfstring{$L^{\infty}$}{Linf}-estimate}\label{sec:posinf}
\rv{Now, we prove that under the Assumption \ref{set1}, the numerical approximations of \eqref{eq:generalsystemsimple} satisfy an exponential bound in time for a sufficiently small time horizon, i.e., as long as the Assumption \ref{set1} hold and in particular as long as the flux and source functions remain Lipschitz continuous.}
To prove this, we first need some estimates on the nonlocal terms.
\begin{lemma}\label{lem:estimates}
   \rv{Let us consider $n\in \N$ fixed, such that the Assumption \ref{set1} holds and \newline $\max_{i\in \Z} \|\brho_i^n\|_\infty<\infty$. Then, we obtain for $\ell=1,\dots,m$ and $k=1,\dots,N$ the following estimates
    \begin{align*}
        |R_{j+1}^{\ell,k,n}-R_{j}^{\ell,k,n}|\leq \mathcal{C}_1 \Delta x \max_{i\in \Z} \|\brho_i^n\|_\infty\quad\text{and}\quad
        |R_{t,j+1}^{\ell,k,n}-R_{t,j}^{\ell,k,n}|\leq (\mathcal{C}_2 + \mathcal{C}_3 \Delta x) \max_{i\in \Z} \|\brho_i^n\|_\infty 
    \end{align*}
    with $\mathcal{C}_1,\mathcal{C}_2, \mathcal{C}_3>0$ depending only on $\omega, L_F$ and $L_S$ .}
\end{lemma}
\begin{proof}\
    For simplicity, we neglect the indexes $\ell$ and $k$ during parts of this proof.
    \rv{\begin{align*}
        R_{j+1}^n-R_j^n=&\frac{\Delta x}{2}\omega\left(\left(\frac14-N_1\right)\Delta x\right)\left(\rho_{j-N_1+1}^n+\frac{\Delta x}{4} s_{j-N_1+1}^n-\rho_{j-N_1}^n-\frac{\Delta x}{4} s_{j-N_1}^n\right)\\
        &+\Delta x \sum_{l=-N_1}^{N_2-2}\omega((l+1)\Delta x)  \left(\rho_{j+l+2}^n-\rho_{j+l+1}^n\right)\\
        &+\frac{\Delta x}{2}\omega\left(\left(N_2-\frac14\right)\Delta x\right)\left(\rho_{j+N_2+1}^n-\frac{\Delta x}{4} s_{j+N_2+1}^n-\rho_{j+N_2}^n+\frac{\Delta x}{4} s_{j+N_2}^n\right)\\
        =&\frac{\Delta x}{2}\omega\left(\left(\frac14-N_1\right)\Delta x\right)\left(\rho_{j-N_1+1}^n+\frac{\Delta x}{4} s_{j-N_1+1}^n-\rho_{j-N_1}^n-\frac{\Delta x}{4} s_{j-N_1}^n\right)\\
        &+\Delta x \sum_{l=-N_1}^{N_2-2} \rho_{j+l+1}^n \left(\omega((l+2)\Delta x)-\omega((l+1)\Delta x)\right)-\Delta x \omega((1-N_1)\Delta x)\rho_{j-N_1+1}\\
        &+\frac{\Delta x}{2}\omega\left(\left(N_2-\frac14\right)\Delta x\right)\left(\rho_{j+N_2+1}^n-\frac{\Delta x}{4} s_{j+N_2+1}^n-\rho_{j+N_2}^n+\frac{\Delta x}{4} s_{j+N_2}^n\right)
    \end{align*}
    Since $\omega$ is of bounded variation and we assumed that $\brho_j^n$ is bounded we get
        \[|R_{j+1}^{\ell,k,n}-R_j^{\ell,k,n}|\leq 2 \Delta x  |  \omega^{\ell,k}|_{BV([-\eta_1^{\ell,k},\eta_2^{\ell,k}])} \max_{i\in \Z} \|\brho_i^n\|_\infty.\]
    Analogously, to above and using the assumptions on the source term, we can obtain
    \begin{align*}
    |R_{t,j+1}^{\ell,k,n}-R_{t,j}^{\ell,k,n}|&\leq 2|   \omega^{\ell,k}|_{BV([-\eta_1^{\ell,k},\eta_2^{\ell,k}])} (\max_{i\in \Z} \Delta x |\sigma^k(\brho_i^n,\bm{R}_i^n)|+\max_{i\in \Z} \Delta x |S_k(\brho_i^n,\bm{R}_i^n)-S_k(\brho_m,\bm{R}_i^n)|)\\
    &\leq  2(2 \rv{L_F}+\Delta x L_S) |   \omega^{\ell,k}|_{BV([-\eta_1^{\ell,k},\eta_2^{\ell,k}])} \max_{i\in \Z} \|\brho_i^n\|_\infty.
    \end{align*}}
\end{proof}
\rv{
With the help of Lemma \ref{lem:estimates} we are now able to prove an upper bound on the solutions of the NT scheme for a sufficiently small time $t^{n^*}>0$.
A result on the positivity preservation of the NT scheme (under additional assumptions) can be found in Appendix \ref{app:pos}.
\begin{theorem}\label{thm:max}
Let the Assumption \ref{set1} be satisfied. Then under the CFL condition 
\begin{equation}\label{eq:CFLmaxNT}
     \rv{\lambda} \rv{L_F} \leq \frac{\sqrt{2}-1}{2},
\end{equation}
there exist $n^*\in \N$ and $\mathcal{C}_4\geq 0$ depending only on $\omega,\ F,\ S,\ M,\ m$ and $N$ such that the approximate solutions of the NT scheme satisfy
    \begin{equation}\label{eq:maxbound}
            |\rho_j^{k,n}|\leq \max_{k=1,\dots,N}\|\rho_0^k\|_{L^\infty(\R;\R_+)}\exp( \mathcal{C}_4 t^n)
    \end{equation}
    for any $j\in\Z,\ k\in\{1,\dots,N\}$ and $n\in \{0,\dots,n^*\}$.
\end{theorem}
}
\begin{proof}\
\rv{We prove the claim by induction. For $n=0$, the statement is obvious.
In particular, the Assumption \ref{set1} is satisfied. For $k\in\{1,\dots,N\}$ and $n\in \N$ \rv{such that the Assumption \ref{set1} holds,} it is sufficient to prove the bound for the approximations on the staggered grid $\rho_{j+1/2}^{k,n+1}$.  
  By construction $\rho_j^{k,n+1}$ will be bounded, too.}
We start with an estimate on the difference of the fluxes \rv{and fix $k\in\{1,\dots,N\}$}. Here, we get \rv{for $j\in\Z$}
\begin{align*}
    &| F_k(\rho_{j+1}^{k,n+1/2},\bm{R}_{j+1}^{n+1/2})-F_k(\rho_{j}^{k,n+1/2},\bm{R}_{j}^{n+1/2})|\\
   & \leq \rv{L_F} \left(|\rho_{j+1}^{k,n}-\rho_j^{k,n}|+ \frac{\Delta t}{2}\left(|\sigma^k(\brho_j^n,\bm{R}_j^n)|+|\sigma^k(\brho_{j+1}^n,\bm{R}_{j+1}^n)|\right)\rv{+
  \frac{\Delta t}{2}\left(|
  S_{k}(\brho_{j}^{n},\bm{R_{j}}^{n})
  |+|S_{k}(\brho_{j+1}^{n},\bm{R_{j+1}}^{n})|
  \right)}\right)\\
  &+ \rv{L_F}\| \bm{R}^{n+1/2}_{j+1}-\bm{R}^{n+1/2}_j\|.
\end{align*}
\rv{Thanks to Lemma \ref{lem:estimates} it becomes apparent that there exists $\mathcal{C}_5$ depending on $\mathcal{C}_1,\ \mathcal{C}_2,\ \mathcal{C}_3,\  m$ and $N$, such that}
\rv{
\begin{equation}\label{eq:estimateRph}
    \| \bm{R}^{n+1/2}_{j+1}-\bm{R}^{n+1/2}_j\|\leq \mathcal{C}_5 \Delta x \max_{i\in\Z} \|\brho^n_i\|_\infty.
\end{equation}}
Similarly, Lemma \ref{lem:estimates} allows us to get a more precise estimate on the slopes of the fluxes:
\begin{align*}
   \Delta x |\sigma^k(\brho_j^n,\bm{R}_j^n)| &\leq | F_k(\rho_{j+1}^{k,n},\bm{R}_{j+1}^{n})-F_k(\rho_{j}^{k,n},\bm{R}_{j}^{n})|\leq \rv{L_F} |\rho_{j+1}^{k,n}-\rho_{j}^{k,n} |+ \rv{L_F} \| \bm{R}^{n}_{j+1}-\bm{R}^{n}_j\|\\
   &\leq \rv{L_F} |\rho_{j+1}^{k,n}-\rho_{j}^{k,n} |+ \mathcal{C}_5 \Delta x \rv{L_F \max_{i\in\Z} \|\brho^n_i\|_\infty}.
\end{align*}
\rv{The source terms can be easily estimated by
\begin{align*}
    \left|
  S_{k}(\brho_{j}^{n},\bm{R_{j}}^{n})
  + S_{k}(\brho_{j+1}^{n},\bm{R_{j+1}}^{n}) -S_k(\brho_m,\bm{R}_{j+1}^{n})-S_k(\brho_m,\bm{R}_{j}^{n})
  \right|\leq & 2 L_s  N \max_{i\in\Z} \|\brho^n_i\|_\infty 
\end{align*}
}
\rv{Combining the estimates above results in}
\begin{align*}
    | F_k(\rho_{j+1}^{k,n+1/2},\bm{R}_{j+1}^{n+1/2})-F_k(\rho_{j}^{k,n+1/2},\bm{R}_{j}^{n+1/2})|&\leq \rv{L_F} \left(\rv{\lambda} \rv{L_F}+1\right) |\rho_{j+1}^{k,n}-\rho_j^{k,n}|\\
    &+\rv{ \left(L_F{ \mathcal{C}_5} (1+\Delta t) \Delta x+\Delta t  L_s  N \right) \max_{i\in\Z} \|\brho^n_i\|_\infty}.
\end{align*}
\rv{
Due to the Lipschitz continuity of the source terms, we also have
\begin{align*}
    &S_k(\brho_{j+1}^{n+1/2},\bm{R}_{j+1}^{n+1/2})+S_k(\brho_{j}^{n+1/2},\bm{R}_{j}^{n+1/2})-S_k(\brho_m,\bm{R}_{j+1}^{n+1/2})-S_k(\brho_m,\bm{R}_{j}^{n+1/2})\\
    &\leq N L_S (2+L_F(1+\mathcal{C}_5 \Delta x+N L_S\Delta t/2)) \max_{i\in\Z} \|\brho^n_i\|_\infty
\end{align*}
}
\rv{Therefore, we obtain
\begin{align*}
    |\rho_{j+1/2}^{k,n+1}|&\leq \frac12|\rho_j^{k,n}+\rho_{j+1}^{k,n}|+\frac14|\rho_j^{k,n}-\rho_{j+1}^{k,n}|+\rv{\lambda}\rv{L_F} \left(\rv{\lambda} \rv{L_F}+1\right) |\rho_{j+1}^{k,n}-\rho_j^{k,n}|+  \mathcal{C}_4 \Delta t \rv{\max_{i\in\Z} \|\brho^n_i\|_\infty}\\
    &\leq \frac12(|\rho_j^{k,n}+\rho_{j+1}^{k,n}|+|\rho_j^{k,n}-\rho_{j+1}^{k,n}|)+  \mathcal{C}_4 \Delta t \rv{\max_{i\in\Z} \|\brho^n_i\|_\infty}\leq (1 +  \mathcal{C}_4 \Delta t) \rv{\max_{i\in\Z} \|\brho^n_i\|_\infty}.
\end{align*}}
\rv{Note that $\lambda L_F(1+\lambda L_F)\leq 1/4$ due to the CFL condition \eqref{eq:CFLmaxNT}. If Assumption \ref{set1} is still satisfied, we can apply the formula again, otherwise we set $n^*=n+1$. Applying this procedure recursively, we obtain the desired estimate and $n^*$.}
\end{proof}
\rv{
\begin{remark}
    Due to the local Lipschitz bounds on the flux and the source function, we obtain the $L^\infty$ bound as long as the density remains below $\rho_M$, which is valid for a sufficiently small time horizon. 
    For global Lipschitz bounds the results can be extended to any finite time horizon.
\end{remark}
}
 \section{\rv{Weak-\(*\) convergence}}\label{weakconvergence}
In this section, we analyze the weak-\(*\) convergence of the NT scheme. Under the Assumption \ref{set1}, the numerical solutions remain uniformly bounded in \(L^\infty\), which allows us to discuss convergence in the weak-\(*\) sense.
We recall the following compactness result for the weak-\(*\) topology, (see, e.g., \cite[page~131]{eymard2000finitevolume})
. A sequence $(\rho^n)_{n\in\mathbb{N}}\subset L^{\infty}(\mathbb{R}\times(0,T))$ converges to $\rho\in L^{\infty}(\mathbb{R}\times(0,T))) $ if 
    \begin{align*}
        \int_{0}^{T}\int_{\mathbb{R}}(\rho^n(t,y)-\rho(t,x))\phi(t,x)dxdt\to 0\quad \text{as}\ n\rightarrow \infty\quad \forall\ \phi\in L^1(\mathbb{R}\times(0,T))
    \end{align*}
    To apply the weak-$*$ convergence the flux and source functions need to be linear in the density.
    Hence, we require the Assumption \ref{set2}.
    In addition, we need the following technical modification of the slopes
    {\begin{small}
    \begin{equation}\label{eq:slopesmodconv}
            \begin{aligned}
        \Delta x\, s_j^{k,n} =& \mm(\rho_{j+1}^{k,n}-\rho_j^{k,n},\, \rho_j^{k,n}-\rho_{j-1}^{k,n},\, \text{sign}(\rho_j^{k,n}-\rho_{j-1}^{k,n})C \Delta x^{\delta}),\\
        \Delta x\, \sigma^{k}(\brho_{j}^{n},\bm{R}_{j}^{n})=&\mm\big({F_k(\rho_{j}^{k,n},\bm{R}_{j}^{n})-F_k(\rho_{j-1}^{k,n},\bm{R}_{j-1}^{n})},{F_k(\rho_{j+1}^{k,n},\bm{R}_{j+1}^{n})-F_k(\rho_{j}^{k,n},\bm{R}_{j}^{n})},\\
        &\text{sign}({F_k(\rho_{j}^{k,n},\bm{R}_{j}^{n})-F_k(\rho_{j-1}^{k,n},\bm{R}_{j-1}^{n})})C \Delta x^{\delta})\big)\\
    \Delta xs_{j+1/2}^{k,n+1}=&\mm\left({{\rho}_{j+3/2}^{k,n+1}-{\rho}_{j+1/2}^{k,n+1}},{{\rho}_{j+1/2}^{k,n+1}-{\rho}_{j-1/2}^{k,n+1}},\text{sign}(\rho_{j+1/2}^{k,n+1}-{\rho}_{j-1/2}^{k,n+1})C\Delta x^\delta\right)
    \end{aligned}
    \end{equation}
    \end{small}}
    with $C>0$ and $\delta \in (0,1)$.
These adjustments to the slopes are only necessary for the analysis and are typically used in convergence proofs for second-order schemes. In the implementation, they are not needed. We refer to \cite[Remark 5.6]{gd2023convergence} for further details.

We define the piecewise linear approximate solution $\brho_\Delta:[0,T)\times \R\to\R^N$ for $k=1,\dots,N$ as
\begin{align*}
    \rho_\Delta^k(t,x) \coloneq \sum_{j\in \Z} \sum_{n=0}^{n^*}(\rho_j^{k,n}+s_j^{k,n}(x-x_j))\chi_{[x_{j-1/2},x_{j+1/2})}(x)\chi_{[t^n,t^{n+1})}(t).
\end{align*}
In addition, we rewrite the NT scheme in a more compact form.
{\begin{small}
        \begin{align}\label{schemelinear}
            \rho_j^{k,n+1}&=\rho_{j}^{k,n}+\lambda[G_{j-1/2}^{k,n}-G_{j+1/2}^{k,n}]+\Delta t S_j^{k,n}\quad \text{with}\\
          \nonumber          G_{j+1/2}^{k,n}&\coloneq \frac{1}{\lambda}\bigg[\frac{1}{4}(\rho_j^{k,n}-\rho_{j+1}^{k,n})+\frac{\Delta x}{16}(s_{j+1}^{k,n}+s_j^{k,n})+\frac{\Delta x}{8}(s_{j+1/2}^{k,n+1})+ \frac{\lambda}{2}\left(F(\rho_{j+1}^{k,n+1/2},\bm{R}_{j+1}^{n+1/2})+F(\rho_j^{k,n+1/2},\bm{R}_j^{n+1/2})\right)\bigg]\\
        \nonumber S_j^{k,n}&\coloneq\frac{1}{4}\bigg[S_k(\brho_{j+1}^{n+1/2},\bm{R}_{j+1}^{n+1/2})+2S_k(\brho_j^{n+1/2},\bm{R}_{j}^{n+1/2})+S_k(\brho_{j-1}^{n+1/2},\bm{R}_{j-1}^{n+1/2})\bigg]
        \end{align}            
\end{small}}
Thanks to the compactness result, the slopes \eqref{eq:slopesmodconv} and the Assumption \ref{set2} we are able to prove the convergence result in the linear case presented in Theorem \ref{thm:weakconv}.
\begin{proof}[Proof of Theorem \ref{thm:weakconv}]\
    We fix $k\in \{1,...,N\}$,
 multiply \eqref{schemelinear} by $\Delta x\phi_j^{k,n}$, where $\phi_j^{k,n}\coloneq\phi^k(x_j,t^n)$ with $\phi^k\in C_0^1((-42,T^*]\times \mathbb{R};\mathbb{R})\subset L^1((0,T^*)\times \mathbb{R})$ and sum over $n$ and $j$ 
        \begin{align*}        0=\sum_{n=0}^{n^*}\sum_{j\in\mathbb{Z}}\bigg(\underbrace{\Delta x(\rho_j^{k,n+1}-\rho_j^{k,n})\phi_j^{k,n}}_{\eqcolon\mathbb{A}_1}-\underbrace{\Delta t\left(G_{j-1/2}^{k,n}-G_{j+1/2}^{k,n}\right)\phi_j^{k,n}}_{\eqcolon\mathbb{A}_2}-\underbrace{\Delta t \Delta x S_j^{k,n} \phi_j^{k,n}\textbf{}}_{\eqcolon\mathbb{A}_3}\bigg) 
        \end{align*}
        We will treat each term separately. 
        Using summation by parts we can write $\mathbb{A}_1$ as
        \begin{align*}
           \mathbb{A}_1=\sum_{n=0}^{n^*}\sum_{j\in \mathbb{Z}}\Delta x\Delta t\frac{(\rho_j^{k,n+1}-\rho_j^{k,n})}{\Delta t} \phi_j^{k,n}
           &=-\sum_{n=1}^{n^*}\sum_{j\in\mathbb{Z}}\Delta x\Delta t\rho_j^{k,n}\frac{\phi_j^{k,n+1}-\phi_j^{k,n}}{\Delta t}-\sum_{j\in \mathbb{Z}}\rho_j^{k,0}\phi_j^{k,0}\\
        &=-\int_{\Delta t}^{T}\int_{\mathbb{R}}\rho_{\Delta }^{k}(t,x)\partial_t\phi_{\Delta}^{k}(t,x)dxdt-\int_{\mathbb{R}}\rho_{\Delta}^{k}(0,x)\phi^{k}_{\Delta}(0,x)dx
        \end{align*}
        Here, $\phi_\Delta$ is defined as $\rho_\Delta$, but for a piecewise constant function. 
        Now, we can apply the weak-\(*\) convergence for the first term and the dominated convergence theorem for the second one to obtain
        \begin{align}\label{P_1}
            \mathbb{A}_1 \to -\int_{0}^{T}\int_{\mathbb{R}}\rho^{k}(t,x)\phi_t^k(t,x)dxdt-\int_{\mathbb{R}}\rho_0^{k}(x)\phi^k(0,x)dx, \quad \text{weakly-\(*\)}.
        \end{align}
       For the flux part in $\mathbb{A}_2$ we proceed similarly and apply summation by parts
       \begin{small}

        \begin{align*}
            \mathbb{A}_2=&\sum_{n=0}^{n^*}\sum_{j\in \mathbb{Z}}\Delta t\left[G_{j+1/2}^{k,n}-G_{j-1/2}^{k,n}\right]\phi_j^{k,n}=-\sum_{n=0}^{n^*}\sum_{j\in\mathbb{Z}}\Delta x \Delta t G_{j+1/2}^{k,n}\frac{\phi_{j+1}^{k,n}-\phi_j^{k,n}}{\Delta x}\\
            =&-\sum_{n=0}^{n^*}\sum_{j\in\mathbb{Z}}\Delta x \Delta t \bigg(G_{j+1/2}^{k,n}- F_k(\rho_{j}^{k,n},\bm{R}_j^{n})\bigg)\frac{\phi_{j+1}^{k,n}-\phi_j^{k,n}}{\Delta x}
            -\sum_{n=0}^{n^*}\sum_{j\in\mathbb{Z}}\Delta x \Delta t F_k(\rho_{j}^{k,n},\bm{R}_j^{n})\frac{\phi_{j+1}^{k,n}-\phi_j^{k,n}}{\Delta x}\\
            =&-\underbrace{\sum_{n=0}^{n^*}\sum_{j\in\mathbb{Z}}\Delta x \Delta t \bigg(G_{j+1/2}^{k,n}- F_k(\rho_{j}^{k,n},\bm{R}_j^{n})\bigg)\frac{\phi_{j+1}^{k,n}-\phi_j^{k,n}}{\Delta x}}_{\eqcolon\mathbb{B}_1}
            -\underbrace{\sum_{n=0}^{n^*}\sum_{j\in\mathbb{Z}}\Delta x \Delta t \rho_j^{k,n}V_k((\omega\ast \brho)(t^n,x_j))\frac{\phi_{j+1}^{k,n}-\phi_j^{k,n}}{\Delta x}}_{\eqcolon\mathbb{B}_2}\\
            &-\underbrace{\sum_{n=0}^{n^*}\sum_{j\in\mathbb{Z}}\Delta x \Delta t \rho_j^{k,n}\left(V_k(\bm{R}_j^n)-V_k((\omega*\brho)(t^n,x_j)\right)\frac{\phi_{j+1}^{k,n}-\phi_j^{k,n}}{\Delta x}}_{\eqcolon\mathbb{B}_3}
        \end{align*}
                   
       \end{small}
        Passing to the limit $\Delta x\rightarrow 0$ and using the weak-\(*\) convergence, we obtain
        \begin{align}\label{fluxP_2}
           \mathbb{B}_2 \to -\int_0^{T}\int_{\R}\rho^{k}(t,x)V_k((\omega\ast \brho)(t,x))\phi^k_x(t,x)dxdt
        \end{align}
        We will now show that the remaining parts converge to zero. We start with $\mathbb{B}_3$.
        Due to the compact support of $\phi^k$ there exist $R_c$ and $T_c$ with $\phi^k(t,x)=0$ for $t>T_c$ and $|x|>R_c$ and indices $j_1,\ j_2\in \Z,\ n_1,\ n_2\in \{0,\dots,n^*\}$ such that $(|j_1|+|j_2|)\Delta x\leq R_c$ and $(n_2-n_1)\Delta t\leq T_c$, we can estimate 
        \begin{align*}
        |\mathbb{B}_3| \leq \mathcal{C}_6\Delta t\Delta x \sum_{n=n_1}^{n_2}\sum_{j=j_1}^{j_2}\sum_{e=1}^{m}\sum_{k=1}^{N}\underbrace{|{R}_j^{e,k,n}-(\omega^{e,k}\ast\rho^{k})(t^n,x_j)|}_{\eqcolon\mathbb{B}_4}
        \end{align*}
        Here, $\mathcal{C}_6$ depends on the $L^\infty$ bound \eqref{eq:maxbound}, $L_F$ and $\phi^k_x$.
        Since $\omega^{\ell,k}$ is in $L^1$, we can use the weak-\(*\) convergence for $\mathbb{B}_4$ and obtain
        $            \int_{\R}(\rho_\Delta^k-\rho^{k})\omega^{\ell,k}(x_j-y)dy\rightarrow 0,$
        such that $\mathbb{B}_4$ (and consequently $\mathbb{B}_3$) converges to zero in the weak-\(*\) topology.
        We still have $\mathbb{B}_1$
        \begin{align*}
\mathbb{B}_1=&-\sum_{n=0}^{n^*}\sum_{j\in\mathbb{Z}}\Delta x\Delta t (G_{j+1/2}^{k,n}-F_{k}(\rho_j^{k,n},\bm{R}_j^n))\frac{\phi_{j+1}^{k,n}-\phi_j^{k,n}}{\Delta x}\\
=&-\sum_{n=0}^{n^*}\sum_{j\in \mathbb{Z}}\Delta x\Delta t\bigg[\frac{\lambda}{4}\underbrace{(\rho_j^{k,n}-\rho_{j+1}^{k,n})}_{\mathbb{B}_5}+\frac{1}{16}\underbrace{\Delta x(s_{j+1}^{k,n}+s_{j}^{k,n})}_{\leq 2C\Delta x^{\delta}}+\frac{1}{8}\underbrace{\Delta x s_{j+1/2}^{k,n+1}}_{\leq C\Delta x^{\delta}}\\
&+\frac{\lambda}{2}\underbrace{\bigg(F_k(\rho_{j+1}^{k,n+1/2},\bm{R}_{j+1}^{n+1/2})+F_k(\rho_j^{k,n+1/2},\bm{R}_{j}^{n+1/2})\bigg)-F_k(\rho_j^{k,n},\bm{R}_j^n)}_{\mathbb{B}_6}\bigg]\frac{\phi_{j+1}^{k,n}-\phi_j^{k,n}}{\Delta x}
        \end{align*}
        Due to the weak convergence, $\mathbb{B}_5$ converges to zero and using the modified slopes, only $\mathbb{B}_6$ remains.
Here, we add several zeros to use the Lipschitz continuity of $F$ and to estimate each difference        

\begin{align*}
&F_k(\rho_{j+1}^{k,n+1/2},\bm{R}_{j+1}^{n+1/2})+F_k(\rho_j^{k,n+1/2},\bm{R}_{j}^{n+1/2})-2F_k(\rho_j^n,\bm{R}_j^n)\\
=&F_k(\rho_{j+1}^{k,n+1/2},\bm{R}_{j+1}^{n+1/2})\mp F_k(\rho_{j+1}^{k,n+1/2},\bm{R}_{j}^{n+1/2})\mp F_k(\rho_{j+1}^{k,n+1/2},\bm{R}_{j}^{n})\mp F_k(\rho_{j+1}^{k,n},\bm{R}_{j}^{n})\\
&-F_k(\rho_{j}^{k,n},\bm{R}_{j}^{n})+F_k(\rho_j^{k,n+1/2},\bm{R}_{j}^{n+1/2})\mp F_k(\rho_j^{k,n},\bm{R}_{j}^{n+1/2})-F_k(\rho_j^{k,n},\bm{R}_{j}^{n})\\
\leq &L_F \bigg(\|\bm{R}_{j+1}^{n+1/2}-\bm{R}_j^{n+1/2}\|+\Delta t\|\bm{R}_{t,j}^{n}\|+\left|\frac{\Delta t}{2}S_k(\brho_{j+1}^n,\bm{R}_{j+1}^n)-\frac{\Delta t}{2}\sigma^{k}(\brho_{j+1}^n,\bm{R}_{j+1}^n)\right|\\
&+\left|\frac{\Delta t}{2}S_k(\brho_{j}^n,\bm{R}_{j}^n)-\frac{\Delta t}{2}\sigma^{k}(\brho_{j}^n,\bm{R}_{j}^n)\right|\bigg)+F(\rho_{j+1}^{k,n},\bm{R}_{j}^{n})-F(\rho_{j}^{k,n},\bm{R}_{j}^{n})\\
\leq & \mathcal{C}_7 (\Delta x +\Delta x^\delta)+F(\rho_{j+1}^{k,n},\bm{R}_{j}^{n})-F(\rho_{j}^{k,n},\bm{R}_{j}^{n}).
\end{align*}
For the last inequality we used the estimate \eqref{eq:estimateRph} and that $S_k$ is Lipschitz continuous as well as that $\brho_j^n,\ \bm{R}_j^n$ are assumed to be bounded, cf. \eqref{eq:maxbound}.
Hence, $S_k(\brho_j^n,\bm{R}_j^n)$ is bounded.
Further, we use the modified slopes such that $\Delta t\sigma^k(\brho_j^n,\bm{R}_j^n)\leq C\Delta x^{\delta}.$
Using this we are also able to bound $\bm{R}_{t,j}^n$, which gives us the existence of a suitable $\mathcal{C}_7$.
As we bound everything by $\Delta x$ it is obvious that the corresponding parts in $\mathbb{B}_1$ converge to zero.
The remaining difference in the flux $F(\rho_{j+1}^{k,n},\bm{R}_{j}^{n})-F(\rho_{j}^{k,n},\bm{R}_{j}^{n})=(\rho_{j+1}^{k,n}-\rho_{j}^{k,n})V_k(\bm{R}_{j}^{n})$ converges to zero using the weak convergence as for the part $\mathbb{B}_5$. Hence, $\mathbb{B}_1$ converges to zero.
Now, we consider $\mathbb{A}_3$, 
\begin{align*}
\mathbb{A}_3=&-\Delta t\sum_{n=0}^{n^*}\Delta x\sum_{j\in\mathbb{Z}}\phi_j^{k,n}\frac{1}{4}\biggl(S_{k}(\brho_{j+1}^{n+1/2},\bm{R}_{j+1}^{n+1/2})+2S_k(\brho_j^{n+1/2},\bm{R}_{j}^{n+1/2})\\
&+S_k(\brho_{j-1}^{n+1/2},\bm{R}_{j-1}^{n+1/2})\pm 4S_k(\brho_j^n,\bm{R}_j^n)\pm 4S_k(\brho_j^n,(\omega \ast \brho)(t^n,x_j))\biggr)\\
=&-\Delta t\sum_{n=0}^{n^*}\Delta x\sum_{j\in\mathbb{Z}}\phi_j^{k,n}\frac{1}{4}\biggl(S_{k}(\brho_{j+1}^{n+1/2},\bm{R}_{j+1}^{n+1/2})+2S_k(\brho_j^{n+1/2},\bm{R}_{j}^{n+1/2})\\
&+S_k(\brho_{j-1}^{n+1/2},\bm{R}_{j-1}^{n+1/2})-4S_k(\brho_j^n,\bm{R}_j^n)\biggr)-\underbrace{\Delta t\sum_{n=0}^{n^*}\Delta x\sum_{j\in\mathbb{Z}}\phi_j^{k,n}S_k(\brho_j^n,(\omega \ast \brho)(t^n,x_j))}_{\mathbb{S}_1}\\
&-\underbrace{\Delta t\sum_{n=0}^{n^*}\Delta x\sum_{j\in\mathbb{Z}}\phi_j^{k,n}\left(S_k(\brho_j^n,\bm{R}_j^n)-S_k(\brho_j^n,(\omega \ast \brho)(t^n,x_j))\right)}_{\mathbb{S}_2}
\end{align*}
We recall that $S_k$ is linear in each entry of $\brho_j^n$, such that we get 
\begin{align}\label{sourceP_3}
\mathbb{S}_1
\rightarrow -\int_0^{T}\int_{\mathbb{R}}\phi^{k}(t,x)S_k(\brho,\omega \ast \brho)dxdt, \quad \text{weakly-\(*\)}.
\end{align}
For the remaining terms we proceed similarly as above. First,
$$\mathbb{S}_2=\sum\sum\phi_j^{k,n}\biggl(S_k(\rho_j,\bm{R}_j^n)-S_{k}(\rho_j,\omega \ast \brho)\biggr)\rightarrow 0 $$ in the weak-\(*\) topology with the same argument we used for $\mathbb{B}_3$ and $\mathbb{B}_4$, i.e., using the convergence of the nonlocal term.
We are left with the difference of the source terms, which we can treat similarly as the difference of the fluxes. By adding several zeros and using the Lipschitz continuity of $S_k$ we can bound most of the differences by $\mathcal{C}_8 (\Delta x+\Delta x^{\delta})$. As before, we use the estimate on the modified slopes, the estimate \eqref{eq:estimateRph} and the boundedness of $\brho_j^n$ and $\bm{R}_j^n$ to obtain such an estimate. In addition, we use the linearity of $S_k$ and the weak convergence to deal with appearing differences of $S_k(\brho_{j+1}^{n},\bm{R}_{j}^{n})-S_k(\brho_{j}^{n},\bm{R}_{j}^{n})$.
We do not go into detail here, as the proof and steps are the same as for the flux above.
Finally collecting the results \eqref{P_1}, \eqref{fluxP_2}, \eqref{sourceP_3}, we can conclude that the approximate solutions of \eqref{schemelinear} converge in the weak-\(*\) sense to weak solution 
\begin{align*}
    \lim_{\Delta x \rightarrow 0}(\mathbb{A}_1+\mathbb{A}_2+\mathbb{A}_3)
    =&\int_{0}^{T}\int_{\R}\rho^k(t,x)\phi_t^{k}(t,x)dxdt+\int_0^{T}\int_{\R}\rho^{k}(t,x)V_k(\omega\ast \brho)(t,x)\phi_x^{k}(t,x)dxdt\\
    &+\int_0^{T}\int_{\mathbb{R}}\phi^{k}(t,x)S_k(\brho,\omega \ast \brho)(t,x)dxdt+\int_{\R}\rho^k_0(x)\phi^k(0,x)dx.
\end{align*}
    \end{proof}

\section{\rv{$L^1_{loc}$ convergence}}\label{sec:BV}
To obtain a convergence result for nonlinear flux and source functions, we postulate more regularity, in particular on the kernel functions in Assumption \ref{set3}.
Hence, we are able to approximate the spatial derivatives of the nonlocal terms more accurately and prove a total variation bound.
The compact support of the kernel functions postulated in Assumption \ref{set3} is needed for specific estimates during the proofs of this section.
We note that without the modifications below the total variation seems to be bounded, too, as can be seen by the numerical tests in Section \ref{sec:numerical}.\\
Due to Assumption \ref{set3}, i.e. $\omega^{\ell,k}$ is differentiable on its support, we can compute $\partial_x F_k$ more accurately
\begin{align}
\partial_xF_k(\rho^k,\bm{R})=\partial_x(g_k(\rho^k))V_k(\bm{R})+g_k(\rho^k)\sum_{\ell=1}^{m}\partial_{R_\ell}V_k(\bm{R})\partial_x R_\ell.
\end{align}
In particular, we obtain the following approximation of $\partial_x R^\ell$ 
\begin{align}\nonumber
    &\partial_x R^\ell(t^n,x_j)=\sum_{k=1}^{N}\partial_x R^{\ell,k}(t^n,x_j)\\
    =&\sum_{k=1}^{N}(-\omega^{\ell,k}(-\eta_1^{\ell,k})\rho^{k}(t^n,x_j-\eta_1^{\ell,k})+\omega^{\ell,k}(\eta_2^{\ell,k})\rho^{k}(t^n,x_j+\eta_2^{\ell,k})-\int_{-\eta_1^{\ell,k}}^{\eta_2^{\ell,k}} (\omega^{\ell,k})'(y-x_j)\rho^{k}(t,y)dy) \nonumber\\
    \approx&\sum_{k=1}^{N}\bigg(-\omega^{\ell,k}(-\eta_1^{\ell,k})\rho_{j-N_1^{\ell,k}}^{k,n}+\omega^{\ell,k}(\eta_2^{\ell,k})\rho_{j+N_2^{\ell,k}}^{k,n}-\frac{\Delta x}{2}\left(\rho_{j-N_{1}^{\ell,k}}^{k,n}+\frac{\Delta x}{4}s_{j-N_1^{\ell,k}}^{k,n}\right)(\omega^{\ell,k})'\left(\frac{\Delta x}{4}-N_1^{\ell,k}\Delta x\right) \nonumber\\
    &-\Delta x\sum_{l=-N_{1}^{\ell,k}}^{N_2^{\ell,k}-2}\rho_{j+l+1}^{k,n}(\omega^{\ell,k})'((l+1)\Delta x)
    -\frac{\Delta x}{2}\left(\rho_{j+N_2^{\ell,k}}^{k,n}-\frac{\Delta x}{4}s^{k,n}_{j+N_2^{\ell,k}}\right) (\omega^{\ell,k})'\left(N_2^{\ell,k}\Delta x-\frac{\Delta x}{4}\right)\bigg)\nonumber\\
    \eqcolon& \partial_x{R}_{j}^{\ell,n} \label{approximation of R}.
\end{align}
We use~\eqref{approximation of R} to approximate the spatial derivative of the nonlocal part in $\partial_x F_k$ and apply the minmod limiter to $\partial_x(g_k(\rho))$.
Hence, we obtain a more accurate approximation of $\sigma^{k}(\brho_j, \bm{R}
_j^n)$ compared to \eqref{eq:slopesv1}:
\begin{align}\label{eq:slopesv2}
    \sigma^{k}(\brho_j^n,\bm{R}_j^n)=\mm\left(\frac{g(\rho_{j+1}^{k,n})-g(\rho_{j}^{k,n})}{\Delta x},\frac{g(\rho_{j}^{k,n})-g(\rho_{j-1}^{k,n})}{\Delta x}\right) V_k(\bm{R}_j^n)+g_k(\rho_j^{k,n})\sum_{\ell=1}^{m}\partial_{R_\ell}V_k(\bm{R}_j^n)\partial_x{R}_{j}^{\ell,n}
\end{align}
Thereby, the maximum principle from Theorem \ref{thm:max} can be extended:
\begin{corollary} Let the Assumption \ref{set3} hold.
Then under the CFL condition \eqref{eq:CFLmaxNT} there exist $n^*\in \N$ and $\mathcal{C}_9\geq 0$ depending only on $\omega,\ F,\ S,\ \brho_0,\ M,\ m$ and $N$ such that the approximate solutions of the NT scheme with the slopes \eqref{eq:slopesv2} satisfy
    \[|\rho_j^{k,n}|\leq \max_{k=1,\dots,N}\|\rho_0^k\|_{L^\infty(\R;\R_+)}\exp( \mathcal{C}_9 t^n)\]
    for any $j\in\Z,\ k\in\{1,\dots,N\}$ and $n\in \{0,\dots,n^*\}$.    
\end{corollary}
\begin{proof}\
From the proof of Theorem \ref{thm:max}, it becomes apparent that the following bound is needed:
$$ \Delta x|\sigma^k(\brho_j^n,\bm{R}_j^n)|\leq \mathcal{C}\left(|\rho_{j+1}^{k,n}+\rho_j^{k,n}|+\Delta x\max_{j\in \Z}\|\brho_j^n\|_\infty\right)$$
Using the definition of the slope in \eqref{eq:slopesv2}, we deduce
\begin{align*}
    \Delta x|\sigma^k(\brho_j^n,\bm{R}_j^n)|\leq&\|V_k\|_\infty L_g|\rho_{j+1}^{k,n}-\rho_j^{k,n}|\\
    &+\Delta x m\max_{\ell=1,\dots,m}\left(\|\partial_\ell V_k\|_\infty [2\|\omega^{\ell,k}\|_\infty+(\eta_1^{\ell,k}+\eta_2^{\ell,k})\|(\omega^{\ell,k})'\|_\infty]\right)\max_{i\in \Z}\|\brho_i^n\|_\infty g(\rho_{j}^{k,n}).
    \end{align*}
In addition, we note that for all $k\in \{1,\dots,N\}:\ \|V_k\|_\infty L_g\leq L_F$, such that the estimate can be obtained and the CFL condition \eqref{eq:CFLmaxNT} remains valid.
\end{proof}   
As long as the approximate solutions remain bounded, i.e. as long as the flux and source terms remain Lipschitz continuous, we can derive the following estimate on the total variation of each time step. The bounded variation estimate for a sufficiently small time follows immediately.
\begin{proposition}\label{BV estiamte}
Let the Assumption \ref{set3} hold. Then under the CFL condition \eqref{eq:CFLmaxNT} there exists $n^*\in \N$ such that the approximate solutions of the NT scheme with the slopes \eqref{eq:slopesv2} satisfy
$$\sum_{k=1}^N\sum_{j\in\Z}|\rho_{j+1}^{k,n+1}-\rho_{j}^{k,n+1}|\leq (1+\mathcal{C}_{10}\Delta t)\sum_{k=1}^N\sum_{j\in\Z}|\rho_{j+1}^{k,n}-\rho_{j}^{k,n}|$$
for $n\in \{0,\dots,n^*\}$ and $\mathcal{C}_{10}$ independent of $\Delta x,\ \Delta t$ defined below. 
\end{proposition}
\begin{proof}\
We note, that it is sufficient to consider the difference of the non-staggered version.
Since $\rho_j^{k,n+1}$ is obtained by a projection of a piecewise-linear interpolant through the staggered cell averages, we can use the triangle inequality to estimate the total variation from above by the total variation of the linear interpolant. The linear interpolation does not change the total variation, see \cite[Lemma~3.1, Chapter~4]{godlewski1991hyperbolic}, such that we obtain for every $k\in\{1,\dots,N\}$
\begin{align}
    \sum_{j\in\mathbb{Z}}|\rho_{j+1}^{k,n+1}-\rho_j^{k,n+1}|\leq\sum_{j\in\mathbb{Z}}|\rho_{j+1/2}^{k,n+1}-\rho_{j-1/2}^{k,n+1}|.
\end{align} 
        Before the projection the scheme is given by
        \begin{align*}
    \rho_{j+1/2}^{k,n+1}=&\frac{\rho_j^{k,n}+\rho_{j+1}^{k,n}}{2}+\frac{\Delta x}{8}(s_j^{k,n}-s_{j+1}^{k,n})-\lambda \bigg[F_k(\rho_{j+1}^{k,n+1/2},\bm{R}_{j+1}^{n+1/2})-F_k(\rho_{j}^{k,n+1/2},\bm{R}_j^{n+1/2})\bigg]+\frac{\Delta t}{2} \tilde{S}_{k,j+1/2}
 \end{align*}
 with $ \tilde{S}_{k,j+1/2}\coloneq S_k(\brho^{n+1/2}_{j+1}, \bm{R}^{n+1/2}_{j+1})+S_k(\brho^{n+1/2}_{j}, \bm{R}^{n+1/2}_{j})
 $. We introduce the notation $\Delta \rho_j\coloneq \rho_{j+1}-\rho_j$ and obtain
   \allowdisplaybreaks
   \begin{align}
  \nonumber &\Delta\rho_{j-1/2}^{k,n+1}\\
        \nonumber=&\frac{\Delta \rho_j^{k,n}}{2}+\frac{\Delta \rho_{j-1}^{k,n}}{2}+\frac{\Delta x \Delta s_{j-1}^{k,n}}{8}-\frac{\Delta x \Delta s_{j}^{k,n}}{8}
        +\lambda\left[F_k(\rho_j^{k,n+1/2},\bm{R}_j^{n+1/2})-F_k(\rho_{j-1}^{k,n+1/2},\bm{R}_{j-1}^{n+1/2})\right]\\
        \nonumber&-\lambda\left[F_k(\rho_{j+1}^{k,n+1/2},\bm{R}_{j+1}^{n+1/2})-F_k(\rho_{j}^{k,n+1/2},\bm{R}_{j}^{n+1/2})\right] +\frac{\Delta t}{2}\Delta\tilde{S}_{k,j-1/2}\\
        \nonumber=&\bigg[1/2-\lambda\bigg(\frac{F_k(\rho_{j+1}^{k,n+1/2},\bm{R}_{j+1}^{n+1/2})-F_k(\rho_{j}^{k,n+1/2},\bm{R}_{j}^{n+1/2})}{\Delta \rho_{j}^{k,n+1/2}} \frac{\Delta \rho_{j}^{k,n+1/2}}{\Delta \rho_{j}^{k,n}}
        +\frac{\Delta x\Delta s_{j}^{k,n}}{8\lambda\Delta \rho_{j}^{k,n}}\bigg)\bigg]\Delta \rho_{j}^{k,n} \\
        \nonumber&+\bigg[1/2+\lambda\bigg(\frac{F_k(\rho_{j}^{k,n+1/2},\bm{R}_{j}^{n+1/2})-F_k(\rho_{j-1}^{k,n+1/2},\bm{R}_{j-1}^{n+1/2})}{\Delta \rho_{j-1}^{k,n+1/2}} \frac{\Delta \rho_{j-1}^{k,n+1/2}}{\Delta \rho_{j-1}^{k,n}}+\frac{\Delta x\Delta s_{j-1}^{k,n}}{8\lambda\Delta \rho_{j-1}^{k,n}}\bigg)\bigg]\Delta \rho_{j-1}^{k,n}\\
        \nonumber&+\frac{\Delta t}{2}\Delta\tilde{S}_{k,j-1/2}\\
        \nonumber=&\bigg[1/2-\lambda\bigg(\frac{F_k(\rho_{j+1}^{k,n+1/2},\bm{R}_{j+1}^{n+1/2})- F_k(\rho_j^{k,n+1/2},\bm{R}_{j+1}^{n+1/2})}{\Delta \rho_{j}^{k,n+1/2}} \frac{\Delta \rho_{j}^{k,n}-\frac{\Delta t\Delta \sigma_j^{k,g}}{2}}{\Delta \rho_{j}^{k,n}}+\frac{\Delta x\Delta s_{j}^{k,n}}{8\lambda\Delta \rho_{j}^{k,n}}\bigg)\bigg]\Delta \rho_{j}^{k,n} \\
        \nonumber&+\bigg[1/2+\lambda\bigg(\frac{F_k(\rho_{j}^{k,n+1/2},\bm{R}_{j}^{n+1/2})- F_k(\rho_{j-1}^{k,n+1/2},\bm{R}_{j}^{n+1/2})}{\Delta \rho_{j-1}^{k,n+1/2}} \frac{\Delta \rho_{j-1}^{k,n}-\frac{\Delta t\Delta \sigma_{j-1}^{k,g}}{2}}{\Delta \rho_{j-1}^{k,n}}+\frac{\Delta x\Delta s_{j-1}^{k,n}}{8\lambda\Delta \rho_{j-1}^{k,n}}\bigg)\bigg]\Delta \rho_{j-1}^{k,n}\\
        \nonumber&+\frac{\Delta t}{2}\Delta\tilde{S}_{k,j-1/2}\\
        \nonumber&+\lambda\left[F_k(\rho_j^{k,n+1/2},\bm{R}_j^{n+1/2})-F_k(\rho_j^{k,n+1/2},\bm{R}_{j+1}^{n+1/2})+F_k(\rho_{j-1}^{k,n+1/2},\bm{R}_{j}^{n+1/2})-F_k(\rho_{j-1}^{k,n+1/2},\bm{R}_{j-1}^{n+1/2} )\right]\\
        \nonumber&-\frac{\lambda\Delta t}{2}\frac{F_k(\rho_{j+1}^{k,n+1/2},\bm{R}_{j}^{n+1/2})-F_k(\rho_{j}^{k,n+1/2},\bm{R}_{j}^{n+1/2})}{\Delta \rho_j^{k,n+1/2}}\bigg(S_k(\brho_{j+1}^n,\bm{R}_{j+1}^n)-S_k(\brho_{j}^n,\bm{R}_{j}^n)-\Delta A_j^{k,n}\bigg)\\
        \nonumber&+\frac{\lambda\Delta t}{2}\frac{F_k(\rho_{j}^{k,n+1/2},\bm{R}_{j-1}^{n+1/2})-F_k(\rho_{j-1}^{k,n+1/2},\bm{R}_{j-1}^{n+1/2})}{\Delta \rho_j^{k,n+1/2}}\bigg(S_k(\brho_{j}^n,\bm{R}_{j}^n)-S_k(\brho_{j-1}^n,\bm{R}_{j-1}^n)-\Delta A_{j-1}^{k,n}\bigg)\\
        \nonumber=&(1/2-a_j^{k,n})\Delta \rho_j^{k,n}+(1/2+a_{j-1}^{k,n})\Delta \rho_{j-1}^{k,n}\\
       \label{eq:diffSource} &+\frac{\Delta t}{2}\Delta\tilde{S}_{k,j-1/2}\\
        \label{eq:lastterm}&+\lambda\left[F_k(\rho_j^{k,n+1/2},\bm{R}_j^{n+1/2})-F_k(\rho_j^{k,n+1/2},\bm{R}_{j+1}^{n+1/2})+F_k(\rho_{j-1}^{k,n+1/2},\bm{R}_{j}^{n+1/2})-F_k(\rho_{j-1}^{k,n+1/2},\bm{R}_{j-1}^{n+1/2} )\right]\\
      \label{eq:term1}  &-\frac{\lambda\Delta t}{2}\frac{F_k(\rho_{j+1}^{k,n+1/2},\bm{R}_{j}^{n+1/2})-F_k(\rho_{j}^{k,n+1/2},\bm{R}_{j}^{n+1/2})}{\Delta \rho_j^{k,n+1/2}}\bigg(S_k(\brho_{j+1}^n,\bm{R}_{j+1}^n)-S_k(\brho_{j}^n,\bm{R}_{j}^n)-\Delta A_j^{k,n}\bigg)\\
      \label{eq:term2}  &+\frac{\lambda\Delta t}{2}\frac{F_k(\rho_{j}^{k,n+1/2},\bm{R}_{j-1}^{n+1/2})-F_k(\rho_{j-1}^{k,n+1/2},\bm{R}_{j-1}^{n+1/2})}{\Delta \rho_j^{k,n+1/2}}\bigg(S_k(\brho_{j}^n,\bm{R}_{j}^n)-S_k(\brho_{j-1}^n,\bm{R}_{j-1}^n)-\Delta A_{j-1}^{k,n}\bigg)
        \end{align}
During the calculations we used the following notations:
\begin{align*}
    \sigma_j^{k,g}&\coloneq \mm\left(\frac{g(\rho_{j+1}^{k,n})-g(\rho_{j}^{k,n})}{\Delta x},\frac{g(\rho_{j}^{k,n})-g(\rho_{j-1}^{k,n})}{\Delta x}\right) V_k(\bm{R}_j^n)\\
    A_j^{k,n}&\coloneq g_k(\rho_j^{k,n})\sum_{\ell=1}^{m}\partial_{R_\ell}V_k(\bm{R}_j^n)\partial_x\bm{R}_{j}^{\ell,n}\\
    a_{j}^{k,n}&\coloneq \lambda\bigg(\frac{F_k(\rho_{j+1}^{k,n+1/2},\bm{R}_{j+1}^{n+1/2})- F_k(\rho_j^{k,n+1/2},\bm{R}_{j+1}^{n+1/2})}{\Delta \rho_{j}^{k,n+1/2}} \frac{\Delta \rho_{j}^{k,n}-\frac{\Delta t\Delta \sigma_j^{k,g}}{2}}{\Delta \rho_{j}^{k,n}}+\frac{\Delta x\Delta s_{j}^{k,n}}{8\lambda\Delta \rho_{j}^{k,n}}\bigg)
\end{align*}
Note that for the difference in the slopes, we obtain $\Delta x|\Delta s_j^{k,n}|\leq|\Delta \rho_j^{k,n}|$ and we also get $\Delta x|\Delta \sigma_j^{k,g}|\leq 2\|V_k\|_\infty L_g|\Delta \rho_j^{k,n}|$, such that under the CFL condition \eqref{eq:CFLmaxNT}
        $$|a_j^{k,n}|\leq\lambda L_g\|V_k\|_\infty(1+\lambda\|V_k\|_\infty L_g)+1/8\leq \lambda L_F(1+\lambda L_F)+1/8\leq 3/8 < 1/2.$$
Let us start with the difference in the sources, i.e. \eqref{eq:diffSource}
\begin{align*}
   \frac{\Delta t}{2}|\Delta\tilde{S}_{k,j-1/2}|&\leq\frac{\Delta t}{2}|S_k(\brho_{j+1}^{n+1/2},\bm{R}_{j+1}^{n+1/2})\pm S_k(\brho_j^{n+1/2},\bm{R}_j^{n+1/2})-S_k(\brho_{j-1}^{n+1/2},\bm{R}_{j-1}^{n+1/2})|\\
   &\leq \frac{\Delta t}{2}L_{S} \sum_ {\tilde k=1}^{N} \bigg[|\Delta\rho_j^{\tilde k,n+1/2}|+\sum_{\ell=1}^m|\Delta R_{j}^{\ell,\tilde k,n+1/2}|+|\Delta\rho_{j-1}^{\tilde k,n+1/2}|+\sum_{\ell=1}^m|\Delta R_{j-1}^{\ell,\tilde k,n+1/2}|\bigg]
\end{align*}
To estimate this, we consider, for fixed $k$,
\begin{align*}
    |\Delta\rho_j^{k,n+1/2}|&\leq \Delta\rho_j^{k,n}+\frac{\Delta t|\Delta \sigma_j^{k,g}|}{2}+\frac{\Delta t|\Delta A_j^{k,n}|}{2}+\frac{\Delta t}{2}\bigg|S_k(\brho_{j+1}^{n},\bm{R}_{j+1}^{n})-S_k(\brho_j^{n},\bm{R}_{j}^{n})\bigg|
\end{align*}
We recall that we have $\Delta t|\sigma_j^{k,g}|\leq 2\lambda \|V_k\|_\infty L_g|\Delta \rho_j^{k,n}|$. In addition, we get
$$|S_k(\brho_{j+1}^{k,n},\bm{R}_{j+1}^{n})-S_k(\brho_j^{k,n},\bm{R}_{j}^{n})|\leq L_{S}\bigg(\sum_{\tilde k=1}^N|\Delta \rho_{j}^{\tilde k,n}|+\sum_{\ell=1}^{m}|\Delta R_j^{\ell,\tilde k,n}|\bigg),$$
where $|\Delta R_j^{\ell,k,n}|\leq \Delta x \|\omega^{\ell,k}\|_\infty \sum_{l=-N_1^{\ell,k}}^{N_2^{\ell,k}-1}|\Delta\rho_{j+l}^{k,n}|$ follows immediately.
It remains to bound $\Delta A_j^{k,n}$
\begin{small}
\begin{align*}
    |\Delta A_j^{k,n}|&\leq L_g |\Delta\rho_j^{k,n}| \sum_{\ell=1}^m \|\partial_\ell V_k\|_\infty \sum_{\tilde k=1}^N \left( 2\|\omega^{\ell,\tilde k}\|_\infty+(\eta_1^{\ell,\tilde k}+\eta_2^{\ell,\tilde k})\|(\omega^{\ell,\tilde k})'\|_\infty\right)\max_{i\in\Z} |\rho_i^{\tilde k,n}| \\
    &+\|g_k\|_\infty \sum_{\ell=1}^m  \left( \sum_{K=1}^N \left( 2\|\omega^{\ell,K}\|_\infty+(\eta_1^{\ell,K}+\eta_2^{\ell,K})\|(\omega^{\ell,K})'\|_\infty\right)\max_{i\in\Z} |\rho_i^{K,n}| \right)\sum_{L=1}^m \|\partial^2_{\ell L}V_k\|_\infty \sum_{\tilde{k}=1}^{N}|\Delta R_{j}^{L,\tilde{k},n}|\\
    &+\|g_k\|_\infty\sum_{\ell=1}^{m}\|\partial_{e}V_k\|_\infty \sum_{\tilde{k}=1}^N \bigg(\omega^{\ell,\tilde k}(-\eta_1^{\ell,\tilde{k}})|\Delta\rho_{j-N_1^{\ell,\tilde k}}|+\omega^{\ell,\tilde k}(\eta_2^{\ell,\tilde{k}})|\Delta\rho_{j+N_2^{\ell,\tilde k}}|+\Delta x\|(\omega^{\ell,\tilde k})'\|_\infty\sum_{l=-N_1^{\ell,\tilde k}}^{N_2^{\ell,\tilde k}-1}|\Delta \rho_{j+l}^{\tilde k}|\bigg)
    \end{align*}     
\end{small}
Summing over $j$ and $k$ it becomes apparent that
$$ \sum_{k=1}^N \sum_{j\in\Z}|\Delta A_j^{k,n}| \leq \mathcal{C}_{11} \sum_{k=1}^N \sum_{j\in\Z}|\Delta \rho_j^{k,n}|$$
with $\mathcal{C}_{11}$ depending on $N,m$, the kernel function, its support, the flux and the $L^\infty$ bound on $\rho$.
Combining this with the estimates discussed above, we deduce that there exists $\mathcal{C}_{12}$ such that
\begin{align*}
\sum_{k=1}^N\sum_{j\in\mathbb{Z}}|\Delta \rho_j^{k,n+1/2}|&\leq\mathcal{C}_{12}\sum_{k=1}^{N}\sum_{j\in\mathbb{Z}}|\Delta \rho_j^{k,n}|
\end{align*}
Since we are still estimating \eqref{eq:diffSource}, we need to estimate the nonlocal terms after the half time step, too. Hence, we estimate
$$|\Delta R_j^{\ell,k,n+1/2}|\leq |\Delta R_{j}^{\ell,k,n}|+\frac{\Delta t}{2}|\Delta (R_{t})_j^{\ell,k,n}|.$$
As above the first term fulfills $|\Delta R_j^{\ell,k,n}|\leq \Delta x \|\omega^{\ell,k}\|_\infty \sum_{l=-N_1^{\ell,k}}^{N_2^{\ell,k}-1}|\Delta\rho_{j+l}^{k,n}|$. To bound $|\Delta (R_{t})_j^{\ell,k,n}|,$ we can proceed similarly as above, since by its definition in \eqref{eq:system Rt}, we obtain differences of $A_j^k$, $\sigma_j^{k,g}$ and the source term at the time level $t^n$. Hence, using the estimates established above, we deduce
$$\sum_{k=1}^N\sum_{j\in \mathbb{Z}}\sum_{\ell=1}^{m}|\Delta R_j^{\ell,k,n+1/2}|\leq \mathcal{C}_{13}\sum_{k=1}^{N}\sum_{j\in\mathbb{Z}}|\Delta\rho_j^n|$$
and collecting all the bounds we can finally bound \eqref{eq:diffSource} 
$$\sum_{k=1}^{N}\sum_{j\in\mathbb{Z}}|\Delta\tilde{S}_{k,j-1/2}|\leq 2L_S(\mathcal{C}_{12}+\mathcal{C}_{13})\sum_{k=1}^{N}\sum_{j\in\mathbb{Z}}|\Delta \rho_j^{k,n}|.$$
The terms \eqref{eq:term1} and \eqref{eq:term2} can be estimated in the same way. For \eqref{eq:term1}, we obtain immediately,
\begin{align*}
   &\frac{\lambda \Delta t}{2} \left|\frac{F_k(\rho_{j+1}^{k,n+1/2},\bm{R}_{j}^{n+1/2})-F_k(\rho_{j}^{k,n+1/2},\bm{R}_{j}^{n+1/2})}{\Delta \rho_j^{k,n+1/2}}\bigg(S_k(\brho_{j+1}^n,\bm{R}_{j+1}^n)-S_k(\brho_{j}^n,\bm{R}_{j}^n)-\Delta A_j^{k,n}\bigg)\right|\\
   &\leq  \frac{\lambda \Delta t}{2} \|V_k\|_\infty L_g \left(|S_k(\brho_{j+1}^n,\bm{R}_{j+1})-S_{k}(\brho_j^n,\bm{R}_j^n)|+|\Delta A_j^{k,n}|\right)
   \end{align*}
Using the estimates established above we get that there exists $\mathcal{C}_{14}$ such that
\begin{align*}
\sum_{k=1}^{N}\sum_{j\in \mathbb{Z}} |\eqref{eq:term2}|=\sum_{k=1}^{N}\sum_{j\in \mathbb{Z}} |\eqref{eq:term1}|  &\leq \mathcal{C}_{14} \Delta t \sum_{k=1}^{N}\sum_{j\in \mathbb{Z}}|\Delta \rho_j^{k,n}|
\end{align*}
Now, we concentrate on \eqref{eq:lastterm} and use the mean value theorem to obtain the following
\begin{align}
   \nonumber \eqref{eq:lastterm}=&\lambda\bigg[-g_{k}(\rho_{j}^{k,n+1/2})\sum_{\ell=1}^m\partial_{\ell}V_k(\bm{\xi}_j^{k,n})\sum_{\tilde{k}=1}^{N}\Delta R_{j}^{\ell,\tilde{k},n+1/2}
   +g_k(\rho_{j-1}^{k,n+1/2})\sum_{\ell=1}^{m}\partial_{\ell}V_k(\bm{\xi}_{j-1}^{k,n})\sum_{\tilde{k}=1}^{N}\Delta R_{j-1}^{\ell,\tilde{k},n+1/2}\bigg]\\
   \nonumber =&\lambda \bigg[ -g_k(\rho_j^{k,n+1/2})\sum_{\ell=1}^m \bigg(\partial_\ell V_k(\bm{\xi}_j^{k,n})\sum_{\tilde{k}=1}^N \Delta R_j^{\ell,\tilde{k},n+1/2}+\partial_{\ell}V_k(\bm{\xi}_{j-1}^{k,n})\sum_{\tilde{k}=1}^N \Delta R_{j-1}^{\ell,\tilde{k},n+1/2}\bigg)\\
   \nonumber &-\left(g_{k}(\rho_j^{k,n+1/2})-g(\rho_{j-1}^{k,n+1/2})\right)\sum_{\ell=1}^{m}\partial_\ell V_k(\bm{\xi}_{j-1}^n)\sum_{\tilde{k}=1}^N \Delta R_{j-1}^{\ell,\tilde{k},n+1/2}\\
    \label{eq:lastterm1}=&-\lambda g_k(\rho_j^{k,n+1/2})\sum_{\ell=1}^m\langle \nabla \left( \partial_{\ell }V_k(\bm{\theta}_{j-1}^{\ell,k,n})\right), (\bm{\xi}_j^{k,n}-\bm{\xi}_{j-1}^{k,n})\rangle \sum_{\tilde{k}=1}^{N} \Delta R_j^{\tilde k, \ell, n+1/2}\\
    \label{eq:lastterm2}&+\lambda g_{k}(\rho_j^{k,n+1/2})\sum_{\ell=1}^{m}\partial_{\ell}V_k(\bm{\xi}_{j-1}^{k,n})\sum_{\tilde{k}=1}^N (-R_{j-1}^{\ell,\tilde{k},n+1/2}+2R_{j}^{\ell,\tilde{k},n+1/2}-R_{j+1}^{\ell,\tilde{k},n+1/2})\\
    \label{eq:lastterm3}&-\lambda \left(g_{k}(\rho_j^{k,n+1/2})-g(\rho_{j-1}^{k,n+1/2})\right)\sum_{\ell=1}^{m}\partial_\ell V_k(\bm{\xi}_{j-1}^{k,n})\sum_{\tilde{k}=1}^n \Delta R_{j-1}^{\ell,\tilde{k},n+1/2}
\end{align}
Here, the terms $\bm{\xi}_j^{k,n}$ and $\bm{\theta}_{j}^{\ell,k,n}$ come from the mean value theorem. In particular, $\bm{\xi}_j^{k,n}=(1-\kappa) \bm{R}_j^{n+1/2}+\kappa \bm{R}_{j+1}^{n+1/2}$ for $\kappa \in(0,1) $ such that $\langle \nabla V_k(\bm{\xi}_j^{k,n}),\bm{R}_j^{n+1/2}-\bm{R}_{j+1}^{n+1/2}\rangle=V_k(\bm{R}_j^{n+1/2})- V_k(\bm{R}_{j+1}^{n+1/2})$. Now, we estimate the terms separately 
\begin{align*}
    |\eqref{eq:lastterm3}|&\leq \lambda L_g|\Delta \rho_{j-1}^{k,n+1/2}|\sum_{\ell=1}^{m}|\partial_\ell V_k(\bm{\xi}_{j-1}^{k,n})|\sum_{\tilde{k}=1}^N|\Delta R_{j-1}^{\ell,\tilde{k},n+1/2}|
\end{align*}
Recall that there exists a $\mathcal{C}>0$ such that $|\Delta R_{j-1}^{\ell,\tilde{k},n+1/2}|\leq \mathcal{C} \Delta x$\footnote{We can follow here the same steps as in the proof Lemma \ref{lem:estimates}.} and that we have already proven $\sum_{k=1}^N \sum_{j\in \Z} |\Delta \rho_{j}^{k,n+1/2}| \leq \mathcal{C}_{12} \sum_{k=1}^N \sum_{j\in \Z} |\Delta \rho_{j}^{k,n}|$. Hence, we obtain
$$  \sum_{k=1}^N \sum_{j\in \Z} |\eqref{eq:lastterm3}|\leq \lambda L_gNm\mathcal{C}_{12}\mathcal{C} \Delta x \sum_{k=1}^N \max_{\ell=1,\dots,m}\|\partial_\ell V_k\|_\infty \sum_{j\in \Z} |\Delta \rho_{j}^{k,n}|$$
To estimate \eqref{eq:lastterm1}, we first note that due to the definition of $\bm{\xi}_j^{k,n}$
\begin{align*}
    |\Delta\bm{\xi}_{j-1}^{k,n}|\leq |\bm{\xi}_j^{k,n}-\bm{R}_j^{n+1/2}|+|\bm{R}_j^{n+1/2}-\bm{\xi}_{j-1}^{k,n}|\leq |\bm{R}_{j+1}^{n+1/2}-\bm{R}_j^{n+1/2}|+|\bm{R}_j^{n+1/2}-\bm{R}_{j-1}^{n+1/2}|, 
\end{align*}
where the absolute value and inequalities are applied to each entry of the vector. Using the differentiability of $V_k$ and $|\Delta R_{j}^{\ell,k,n+1/2}|\leq \mathcal{C} \Delta x$ it becomes apparent that $\langle \nabla \left( \partial_{\ell }V_k(\bm{\theta}_{j-1}^{\ell,k,n})\right), (\bm{\xi}_j^{k,n}-\bm{\xi}_{j-1}^{k,n})\rangle$ can be bounded by a constant times $\Delta x$ and since we have already proven \newline $\sum_{k=1}^N\sum_{j\in \mathbb{Z}}\sum_{\ell=1}^{m}|\Delta R_j^{\ell,k,n+1/2}|\leq \mathcal{C}_{13}\sum_{k=1}^{N}\sum_{j\in\mathbb{Z}}|\Delta\rho_j^{k,n}|$, we obtain
\begin{align*}
\sum_{k=1}^N\sum_{j\in \mathbb{Z}}\sum_{\ell=1}^{m}|\eqref{eq:lastterm1}|\leq 2\lambda \mathcal{C}_{13} \mathcal{C}\Delta x\sum_{k=1}^{N}  \| g_k\|_\infty \max_{\ell=1,\dots m} \| \nabla (\partial_\ell V_k)\|_\infty \sum_{j\in\mathbb{Z}}|\Delta\rho_j^{k,n}|
    \end{align*}
To estimate \eqref{eq:lastterm2} we need to find an appropriate bound on the second-order differences of the nonlocal terms. Obviously, we can estimate
\begin{align*}
|R_{j+1}^{\ell,\tilde{k},n+1/2}-2R_j^{\ell,\tilde{k},n+1/2}+R_{j-1}^{\ell,\tilde{k},n+1/2}|\leq&|R_{j+1}^{\ell,\tilde{k},n}-2R_j^{\ell,\tilde{k},n}+R_{j-1}^{\ell,\tilde{k},n}|+\frac{\Delta t}{2}|(R_t)_{j+1}^{\ell,\tilde{k},n}-2(R_t)_{j}^{\ell,\tilde{k},n}+(R_t)_{j-1}^{\ell,\tilde{k},n}|.
\end{align*}
We note that the second-order difference $(R_t)_{j+1}^{\ell,\tilde{k},n}-2(R_t)_{j}^{\ell,\tilde{k},n}+(R_t)_{j-1}^{\ell,\tilde{k},n}$ depends on differences in $S_k(\brho_j^k,\bm{R}_j^k)$, $A_j^{k,n}$ and $\sigma_j^{k,g}$.
Here, the second-order difference is only necessary to bound the differences of $\sigma_j^{k,g}$ appropriately. Since we have already shown $ \sum_{k=1}^N \sum_{j\in\Z}|\Delta A_j^{k,n}| \leq \mathcal{C}_{11} \sum_{k=1}^N \sum_{j\in\Z}|\Delta \rho_j^{k,n}|$ and a similar estimate for the difference of $S_k(\brho_j^k,\bm{R}_j^k)$, a first-order difference is sufficient to get the desired estimate (by using that all kernel functions have a bounded integral).
To estimate the difference of $\sigma_j^{k,g}$ we use the definition of $(R_t)_j^{\ell,k,n}$ in \eqref{eq:system Rt}, sum over $j$ and use summation by parts
\begin{small}
\begin{align*}
    &\sum_{j\in\Z} \Bigg| \frac{\Delta x}{2}\Delta \sigma_{j-N_1^{\ell,k}}^{k,g,n} \omega\left( \left(\frac{1}{4}-N_1^{\ell,k}\right)\Delta x\right)+\Delta x \sum_{l=-N_1^{\ell,k}}^{N_2^{\ell,k}-2} \Delta \sigma_{j+l+1}^{k,g,n} \omega\left( \left(l+1\right)\Delta x\right) +\frac{\Delta x}{2}\Delta \sigma_{j+N_2^{\ell,k}}^{k,g,n} \omega\left( \left(N_2^{\ell,k}-\frac{1}{4}\right)\Delta x\right)\\
    &-\bigg[\frac{\Delta x}{2}\Delta \sigma_{j-N_1^{\ell,k}-1}^{k,g,n} \omega\left( \left(\frac{1}{4}-N_1^{\ell,k}\right)\Delta x\right)+\Delta x \sum_{l=-N_1^{\ell,k}}^{N_2^{\ell,k}-2} \Delta \sigma_{j+l}^{k,g,n} \omega\left( l\Delta x\right) +\frac{\Delta x}{2}\Delta \sigma_{j+N_2^{\ell,k}-1}^{k,g,n} \omega\left( \left(N_2^{\ell,k}-\frac{1}{4}\right)\Delta x\right)\bigg]\Bigg|\\
    \leq& 4 \|\omega^{\ell,k}\|_\infty \sum_{j\in\Z} \Delta x | \Delta \sigma_j^{k,g,n}| +\Delta x \sum_{j\in\Z} \sum_{l=-N_1^{\ell,k}}^{N_2^{\ell,k}-2} |\sigma_{j+l+1}^{k,g,n} \left( \omega\left( \left(l+1\right)\Delta x\right)-\omega\left( l\Delta x\right)\right)|\\
    \leq & 4 \|\omega^{\ell,k}\|_\infty \sum_{j\in\Z} \Delta x | \Delta \sigma_j^{k,g,n}| +\Delta x \|(\omega^{\ell,k})'\|_\infty \sum_{l=-N_1^{\ell,k}}^{N_2^{\ell,k}-2}  \sum_{j\in\Z}\Delta x | \sigma_{j+l+1}^{k,g,n}|\\
    \leq &(4 \|\omega^{\ell,k}\|_\infty+ ( \eta_1^{\ell,k}+\eta_2^{\ell,k})\|(\omega^{\ell,k})'\|_\infty)\sum_{j\in\Z} \Delta x | \Delta \sigma_j^{k,g,n}|
\end{align*}
\end{small}
Now, we can use $\Delta x|\Delta \sigma_j^{k,g}|\leq 2\|V_k\|_\infty L_g|\Delta \rho_j^{k,n}|$ to obtain the desired estimate
$$ \frac{\Delta t}{2}\sum_{j\in\Z} |(R_t)_{j+1}^{\ell,\tilde{k},n}-2(R_t)_{j}^{\ell,\tilde{k},n}+(R_t)_{j-1}^{\ell,\tilde{k},n}| \leq  \mathcal{C}_{15} \Delta x\sum_{j\in\mathbb{Z}}|\Delta\rho_j^{\tilde k,n}|$$
For the second-order difference $R_{j+1}^{\ell,\tilde{k},n}-2R_j^{\ell,\tilde{k},n}+R_{j-1}^{\ell,\tilde{k},n}$ we can proceed similarly as for $\sigma_g^{k,n}.$ Using summation by parts, we get
\begin{align*}
   & \sum_{j\in \mathbb{Z}}|R_{j+1}^{\ell,\tilde k,n}-2R_j^{\ell,\tilde k,n}+R_{j-1}^{\ell,\tilde k,n}|\leq \mathcal{C}_{16} \Delta x\sum_{j\in\mathbb{Z}}|\Delta\rho_j^{\tilde k,n}|,
\quad \text{
such that}
   & \sum_{j\in \mathbb{Z}}\eqref{eq:lastterm2}\leq (\mathcal{C}_{15}+\mathcal{C}_{16}) \Delta x\sum_{j\in\mathbb{Z}}|\Delta\rho_j^{\tilde k,n}|,
\end{align*}
Finally, using all the estimates obtained above, we deduce
$$\sum_{k=1}^N\sum_{j\in \Z}|\Delta \rho_{j-1/2}^{k,n}|\leq |(1+\mathcal{C}_{10}\Delta t)| \sum_{k=1}^N\sum_{j\in \Z} |\Delta \rho_{j}^{k,n}|$$
    \end{proof}

In addition, the BV estimate allows to prove that the $L^1$ norm stays bounded for a sufficiently small time as well as a time continuity estimate.
\begin{proposition}\label{Prop:timecont}
Let the Assumption \ref{set3} hold. Then under the CFL condition \eqref{eq:CFLmaxNT} there exists $n^*\in \N$ such that the approximate solutions of the NT scheme with the slopes \eqref{eq:slopesv2} satisfy
\begin{align*}
 \Delta x\sum_{k=1}^N \sum_{j\in\Z} |\rho_j^{k,n}|\leq \exp(\mathcal{C}_{17}t^n)\left(\sum_{k=1}^N \|\rho_o^k\|_{L^1(\R)}+\mathcal{C}_{18} t^n\right)\quad \text{and}\quad
   \Delta x\sum_{k=1}^N \sum_{j\in\Z} |\rho_j^{k,n+1}-\rho_j^{k,n}|\leq  \mathcal{C}_{19} \Delta t 
\end{align*}
for $n\in \{0,\dots,n^*\}$
\end{proposition}
\begin{proof}\
    Let us reuse the notation introduced in \eqref{schemelinear}. For fixed $k\in \{1,,\dots,N\}$ we obtain
    \begin{align*}
       \Delta x\sum_{k=1}^N \sum_{j\in\Z} |\rho_j^{k,n+1}|\leq \Delta x\sum_{k=1}^N \sum_{j\in\Z} |\rho_j^{k,n}| +\Delta t \sum_{k=1}^N \sum_{j\in\Z} \left| G_{j-1/2}^{k,n}-G_{j+1/2}^{k,n} \right|+\Delta t \Delta x \sum_{k=1}^N \sum_{j\in\Z}|S_j^{k,n}|
    \end{align*}
    Using the definition of $G_{j+1/2}^{k,n}$ and the Lipschitz continuity of $g_k$ we can estimate its differences using differences in $\rho_j^{k,n}$, $\rho_j^{k,n+1/2}$, $\rho_{j+1/2}^{k,n+1}$ and $\bm{R}_j^{n+1/2}$. As shown in the proof of Proposition \ref{BV estiamte}, all these differences can be bounded by differences in $\rho_j^{k,n}$ for which we can use the BV estimate to obtain
    \begin{align*}
        \Delta t \sum_{k=1}^N \sum_{j\in\Z} \left| G_{j-1/2}^{k,n}-G_{j+1/2}^{k,n} \right|\leq \mathcal{C}_{18} \Delta t.
    \end{align*}
    The source term $\Delta x \sum_{k=1}^N \sum_{j\in\Z}|S_j^{k,n}|$ can be bounded using its Lipschitz continuity and the $L^1$ norm. In particular (due to $\rho_m=0$), we have
    \begin{align*}
        |S_k(\brho_j^{n+1/2},\bm{R}_j^{n+1/2})|\leq &  L_S\left( \sum_{\ell=1}^N |\rho_j^{\ell,n}|+\frac{\Delta t}{2}|S_\ell(\brho_j^n,\bm{R}_j^n)|+\frac{\lambda \Delta x}{2}|\sigma_\ell(\brho_j^n,\bm{R}_j^n)|\right)\\
            \leq & L_S\left( \sum_{\ell=1}^N |\rho_j^{\ell,n}|+\frac{\Delta t L_S}{2}\sum_{l=1}^N|\rho_j^{l,n}|+\frac{\lambda L_F}{2}(|\rho_j^{\ell,n}|+|\rho_{j+1}^{\ell,n}|)\right)          
    \end{align*}
    This allows us to bound 
    \begin{align*}
        \Delta x \sum_{k=1}^N \sum_{j\in\Z} |S_k(\brho_j^{n+1/2},\bm{R}_j^{n+1/2})|\leq \mathcal{C}_{17} \Delta x\sum_{k=1}^N \sum_{j\in\Z} |\rho_j^{k,n}|,
    \end{align*}
    such that we get the estimate
    \begin{align*}
        \Delta x\sum_{k=1}^N \sum_{j\in\Z} |\rho_j^{k,n+1}|\leq \Delta x\sum_{k=1}^N \sum_{j\in\Z} |\rho_j^{k,n}|(1+\mathcal{C}_{17}\Delta t)+ \Delta t \mathcal{C}_{18}\leq (\Delta x\sum_{k=1}^N \sum_{j\in\Z} |\rho_j^{k,n}|+ \Delta t \mathcal{C}_{17})(1+\mathcal{C}_{18}\Delta t).
    \end{align*}
    Applying the last formula recursively gives the upper bound on the $L^1$ norm. Similar to above we derive the following time continuity estimate
        \begin{align*}
       \Delta x\sum_{k=1}^N \sum_{j\in\Z} |\rho_j^{k,n+1}-\rho_j^{k,n}|\leq \Delta t \mathcal{C}_{17}\Delta x\sum_{k=1}^N \sum_{j\in\Z} |\rho_j^{k,n}| +\Delta t \mathcal{C}_{18}
      \end{align*}
      The statement follows by applying the established estimate on the $L^1$ norm.
\end{proof}
This allows us to apply Helly's theorem, such that we already get the following corollary.
\begin{corollary}\label{cor:weak}
Let the Assumption \ref{set3} hold. Then, under the CFL condition \eqref{eq:CFLmaxNT} there exists a $T^*\in (0,T]$ such that the approximate solutions of the NT scheme with the slopes \eqref{eq:slopesv2} converge, up to a subsequence, to a weak solution of \eqref{eq:generalsystem}.
\end{corollary}
\begin{proof}\
     From Propositions \ref{BV estiamte} and \ref{Prop:timecont} a BV estimate in space and time can be derived.
     The BV and $L^\infty$ estimate allow to apply Helly's Theorem, guaranteeing the existence of $\brho \in L^{\infty}\cap BV((0,T^*)\times\R)$ and the $L^1_\textnormal{loc}$-convergence (up to a subsequence) of $\brho_\Delta \to \brho$. 
     We can follow the same steps as in the proof of the weak-$*$ convergence in Theorem \ref{thm:weakconv} with minor modifications. 
     In parts where we used the linearity of the flux or source, we can now either pass to the limit directly due to Helly's theorem or use the BV estimate to show that the corresponding terms converge to zero. Note that we do not need to modify the flux further as in \eqref{eq:slopesmodconv}, since the BV estimate can be used.
\end{proof}
To prove the convergence against entropy weak solutions we need a discrete analog and hence we define, similar to the definition in \eqref{schemelinear}:
\begin{align*}
    G_{j+1/2}^{k,n}(u,w)\coloneq& \frac{1}{\lambda}\bigg[\frac{1}{4}(u-w)+\frac{\Delta x}{16}(s_{j+1}^{k,n}+s_j^{k,n})+\frac{\Delta x}{8}(s_{j+1/2}^{k,n+1})\\
    &+ \frac{\lambda}{2}\bigg(F_k(w+\frac{\Delta t}{2}(S_k(\brho_{j+1}^n,\bm{R}_{j+1}^n)-\sigma^k(\brho_{j+1}^n,\bm{R}_{j+1}^n)),\bm{R}_{j+1}^{n+1/2})\\
    &+F_k(u+\frac{\Delta t}{2}(S_k(\brho_j^n,\bm{R}_j^n)-\sigma^k(\brho_j^n,\bm{R}_j^n)),\bm{R}_j^{n+1/2})\bigg)\bigg]
\end{align*}
For $\zeta \in \R$ we define $a\wedge b\coloneq \max(a,b),\ a\vee b\coloneq \min(a,b)$ and
\begin{align*}
    \mathcal{F}_{j+1/2}^{\zeta,k,n}(u,w)\coloneqq G_{j+1/2}^{k,n}(u \wedge \zeta ,w \wedge \zeta )-G_{j+1/2}^{k,n}(u \vee \zeta ,w\vee \zeta ).
\end{align*}
\begin{proposition}
Let the Assumption \ref{set3} hold. Then under the CFL condition \eqref{eq:CFLmaxNT} there exists $n^*\in \N$ such that the approximate solutions of the NT scheme with the slopes \eqref{eq:slopesv2}  scheme fulfill the following discrete entropy inequality for each $k=1,\dots,N$, $n\in \{0,\dots,n^*\}$ and $\zeta \in \R$
\begin{align*}
    &|\rho_j^{k,n+1}-\zeta|-|\rho_j^{k,n}-\zeta| + \textnormal{sgn}(\rho_j^{k,n+1}-\zeta)\bigg( \frac{\Delta x}{16}(s_{j+1}^{k,n}-s_{j-1}^{k,n})+\frac{\Delta x}{8} (s_{j+1/2}^{k,n}-s_{j-1/2}^{k,n}) \\
    &+\frac{\lambda}{2} (F_k(\zeta+\frac{\Delta t}{2}(S_k(\brho_{j+1}^n,\bm{R}_{j+1}^n)-\sigma^k(\brho_{j+1}^n,\bm{R}_{j+1}^n)),\bm{R}_{j+1}^{n+1/2})\\
    &-F_k(\zeta+\frac{\Delta t}{2}(S_k(\brho_{j-1}^n,\bm{R}_{j-1}^n)-\sigma^k(\brho_{j-1}^n,\bm{R}_{j-1}^n)),\bm{R}_{j-1}^{n+1/2}))\\
    &- \frac{\Delta t}{4}\bigg[S_k(\brho_{j+1}^{n+1/2},\bm{R}_{j+1}^{n+1/2})+2S_k(\brho_j^{n+1/2},\bm{R}_{j}^{n+1/2})+S_k(\brho_{j-1}^{n+1/2},\bm{R}_{j-1}^{n+1/2})\bigg]\bigg)\\
    &+\lambda \left(\mathcal{F}_{j+1/2}^{\zeta,k,n}(\rho_j^{k,n},\rho_{j+1}^{k,n})-\mathcal{F}_{j-1/2}^{\zeta,k,n}(\rho_{j-1}^{k,n},\rho_{j}^{k,n})\right)\leq 0
\end{align*}
\end{proposition}
\begin{proof}\
    Although the numerical flux $G_{j+1/2}^{k,n}(u,w)$ seems to be more complex, the proof is completely analogous to the ones in e.g. \cite{amorim2015numerical, friedrich2018godunov,chiarello2018global}, since under the CFL condition \eqref{eq:CFLstart} (and hence also under \eqref{eq:CFLmaxNT}) the flux $G_{j+1/2}^{k,n}(u,w)$ is monotone increasing in $u$ and decreasing in $w$. 
\end{proof}
Finally, we are able to prove our main result of this section:
\begin{proof}[Proof of Theorem \ref{thm:L1locconv}]\
    The existence of $\brho \in L^{\infty}\cap BV((0,T^*)\times\R)$ and the $L^1_\textnormal{loc}$-convergence (up to a subsequence) of $\brho_\Delta \to \brho$ can be obtained as in the proof of Corollary \ref{cor:weak}.
    Then starting from the discrete entropy inequality and using Lax-Wendroff type arguments, one can show that this limit function is indeed a weak entropy solution in the sense of Definition \ref{def:weakentropy}. 
    The procedure is thereby very similar to the steps in the proof of Theorem \ref{thm:weakconv} and to the proofs in \cite{BlandinGoatin2016,friedrich2020onetoone}. 
    The main differences are the additional terms introduced by the slopes $s_j^{k,n}$ and $s_{j+1/2}^{k,n}$ as well as the terms coming from the half-time steps, e.g. $ \frac{\Delta t}{2}(S_k(\brho_{j}^n,\bm{R}_{j}^n)-\sigma^k(\brho_{j}^n,\bm{R}_{j}^n))$. 
    Nevertheless, thanks to the BV estimate, all these additional terms converge to zero for $\Delta x\to 0$.
\end{proof}
\section{Numerical examples}\label{sec:numerical}
In this section, we demonstrate the applicability and performance of the central schemes from Sections \ref{sec:staggered central} and \ref{sec:BV} to different nonlocal models. We show how to apply them to further nonlocal models not covered in the previous sections, too.
\par We will present convergence tests for smooth initial data and an approximation of the solution for discontinuous initial data to demonstrate the performance of the numerical schemes.
\rv{We will start with an example in which the convergence of the NT schemes with the slopes \eqref{eq:slopesv1} and \eqref{eq:slopesv2} is guaranteed. Then we will consider examples in which our theoretical results only guarantee the convergence of \eqref{eq:slopesv2} and demonstrate that the more convenient approximation \eqref{eq:slopesv1} still provides an appropriate solution.}
\par \rv{We compare the NT schemes to first-order accurate LxF type schemes and their second-order extensions. We note that first-order LxF type schemes have been widely studied in the literature of nonlocal conservation laws. In contrast, there are only a few works considering their second-order extension. Here, we follow \cite{manoj2025convergence,manoj2024positivity} and use the same linear reconstruction as for the NT scheme, the corresponding first-order LxF type fluxes, a two-step Runge-Kutta method in time and similarly as in the NT scheme the trapezoidal rule for the source term. We note that the convergence properties of this scheme for systems (even without a source) have not been fully studied in the literature and are out of scope for this work. We only use it for numerical comparisons.} If not stated otherwise, we use periodic boundary conditions.

\subsection{Keyfitz-Kranzer-Model}
As a \rv{first} test case we consider the Keyfitz-Kranzer model similar to \cite{aggarwal2015nonlocal,aggarwal2024well}.
\rv{For this model the convergence results for both versions of the NT scheme hold, as the fluxes are linear. In particular, } the modeling equations are given by:
\begin{align*}
    \partial_t \rho^1+ \partial_x (\rho^1 v(\omega_\eta\ast \rho^1, \omega_\eta\ast \rho^2))&=0,\\
    \partial_t \rho^2+ \partial_x (\rho^2 v(\omega_\eta\ast \rho^1, \omega_\eta\ast \rho^2))&=0.
\end{align*}
The convergence of a LxF type scheme has been studied in \cite{aggarwal2015nonlocal,aggarwal2024well} \rv{and we choose the corresponding diffusion parameter as $1/3$.} 
\rv{We use the same parameter value for its second-order extension. The CFL condition is given in \cite{aggarwal2024well}, as it is the most restrictive one.}
Following \cite{aggarwal2024well}, we consider the kernel:
\[\omega_\eta(x)= L(-x(\eta+x))^{5/2} \chi_{(-\eta,0)}(x),\]
where $L$ is chosen such that the integral is normalized to $1$.
The velocity function is $v(a,b)=(1-a^2-b^2)^3$.
\rv{First}, we consider a discontinuous \rv{initial datum} to see the approximation of the different numerical schemes\rv{, i.e. the second-order accurate NT schemes with slopes \eqref{eq:slopesv1} and \eqref{eq:slopesv2} as well as a first-order accurate LxF scheme and its second-order extension}. The initial conditions are given by
\[\rho_0^1(x)=0.25 \chi_{(1,3)}(x),\quad \rho_0^2(x)=1\chi_{(1,3)}(x),\]
with $x\in[0,4]$, the final time is $T=0.3$ and $\eta=1$.
The solution is shown in Figure \ref{fig:figureSys}.
All schemes approximate the solution with the LxF scheme being the most diffusive one.
\rv{There is no visible difference between the two different versions of the NT schemes.}
\rv{The grid size is given by $\Delta x= \frac{1}{40}$ and the reference solution is computed by the NT scheme with the slopes given by \eqref{eq:slopesv2} and $\Delta x= \frac{1}{20}\cdot 2^{-7}$.}
 \begin{figure}
    \centering
     \setlength{\fwidth}{0.75\linewidth}
     \input{KeyfitzKranzerdiscNew}
     \caption{\rv{Numerical solutions of the nonlocal Keyfitz-Kranzer model at $T=0.3$ obtained by different numerical schemes for $\Delta x= \frac{1}{40}$.}}\label{fig:figureSys}
 \end{figure}
The order of accuracy \rv{in the $L^1$ norm} is tested for the smooth initial data 
\[\rho_0^1(x)=-0.1-0.2\sin(\pi x),\quad \rho_0^2(x)=0.2+0.1\sin(\pi x),\]
on the interval $[-1,1]$ with $\eta=0.5$ and final time $T=0.15.$ 
\rv{Here, we compare the approximate solutions (with $\Dx=\frac{1}{20}\cdot2^{-n}$ for $n=0,\dots,5$) to a reference solution, still computed by the NT scheme with the slopes \eqref{eq:slopesv2} and $\Dx=\frac{1}{20}\cdot2^{-9}$.}
The results in Table \ref{tab:convTestSystemL1} display the expected order of accuracy for each numerical scheme.
\rv{We note that using the more accurate approximation of the slopes, i.e. \eqref{eq:slopesv2} results in slightly smaller errors. }

\begin{table}[hbt!]
\centering
\caption{\rv{$L^1$ errors and convergence rates for the smooth initial data and the Keyfitz-Kranzer system}}
\label{tab:convTestSystemL1}
    \begin{tabular}{c | c  c | c  c | c  c | c  c }
        &\multicolumn{2}{c|}{LxF} & \multicolumn{2}{c|}{2nd order LxF} &\multicolumn{2}{c|}{NT with \eqref{eq:slopesv1}} & \multicolumn{2}{c}{NT with \eqref{eq:slopesv2}}\\
      $n$&$L^1$-error&c.r. &$L^1$-error&c.r. &$L^1$-error&c.r. &$L^1$-error&c.r. \\
      \hline \hline
0&8.53e-02&-&1.15e-02&-&1.77e-02&-&1.76e-02&-\\
1&4.56e-02&0.90&3.61e-03&1.67&5.52e-03&1.68&5.51e-03&1.68\\
2&2.37e-02&0.95&1.01e-03&1.85&1.56e-03&1.82&1.56e-03&1.82\\
3&1.21e-02&0.97&2.71e-04&1.89&4.25e-04&1.88&4.24e-04&1.88\\
4&6.09e-03&0.99&7.01e-05&1.95&1.10e-04&1.95&1.10e-04&1.95\\
5&3.06e-03&0.99&1.75e-05&2.00&2.77e-05&1.99&2.77e-05&1.99\\
\hline\hline
    \end{tabular} 
\end{table}

\subsection{\rv{Nonlinear flux:} Scalar case}
In the \rv{next} test case, we demonstrate how the schemes perform in the presence of a nonlinearity in the flux.
Since the NT schemes belong to the class of central schemes (as the LxF), they can be directly applied to models with a nonlinearity in the flux and no Riemann solvers need to be constructed.
\rv{While the convergence of the NT scheme with the slopes approximated by \eqref{eq:slopesv2} is proven, the convergence for the more convenient approximation by \eqref{eq:slopesv1} remains open. 
Hence, we will study it numerically. In the previous example we have already seen that both approximation only result in small differences.}
\par In particular, we consider the nonlocal Arrhenius-type look ahead dynamics for $N=1$ similar to \cite{bellman1983vibrational, sopasakis2006stochastic,chiarello2018global}: 
\[\partial_t \rho + \partial_x (\rho(1-\rho)v(\omega_\eta \ast \rho))=0,\]
with $v(\rho)=\exp(-\rho)$ and a kernel function of compact support on $[0,\eta]$, non-increasing and unit integral.
LxF-type schemes for this and similar models have already been discussed in several works, see e.g. \cite{friedrich2023numerical, chiarello2018global, friedrich2018godunov, huang2024asymptotic}.
Here, we employ the schemes considered in \cite{friedrich2023numerical} \rv{and use the same time step size given by the CFL condition \eqref{eq:CFLmaxNT}.}
\par In the following, the nonlocal range is set to $\eta=0.2$.
The discontinuous initial data $\rho_0(x)$ is
\begin{align}
    \rho_0(x)=0.2+0.8\chi_{[-1/4,1/4]}(x),
\end{align}
and the kernel is constant, i.e., $\omega_\eta(x)=1/\eta$ for $x\in [0,\eta]$. Figure \ref{fig:figure4} shows \rv{a zoom into} the solution at $T=1.5$ for the different numerical schemes the grid size \rv{$\Delta x= \frac{1}{20}\cdot 2^{-3}$}. 
In addition, a reference solution with the NT scheme \rv{\eqref{eq:slopesv2} and $\Delta x= \frac{1}{20}\cdot 2^{-9}$ is displayed}.
\rv{It can be {observed} that the NT schemes provide more accurate approximations than the first-order LxF scheme but less accurate ones than the second-order LxF scheme, and that both NT variants provide a very similar approximation.}
 \begin{figure}
     \centering
     \setlength{\fwidth}{0.6\linewidth}
     \input{ArrheniusdiscNew}
     \caption{\rv{Zoom into the numerical solutions of the Arrhenuis-type look ahead dynamics at $T=1.5$ obtained by different numerical schemes with \rv{$\Delta x= \frac{1}{160}$}}.}\label{fig:figure4} 
 \end{figure}
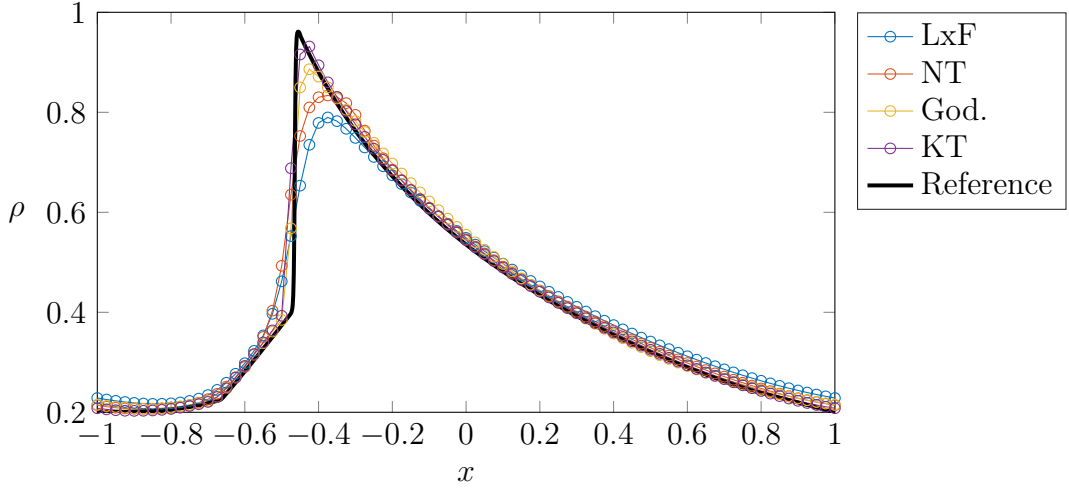
 \par Furthermore, we test the numerical convergence rate in the $L^1$-norm.
Here, we compare the approximate solutions \rv{at $T=0.15$} to a reference solution, still computed by the NT type scheme \rv{\eqref{eq:slopesv2} and $\Delta x= \frac{1}{20}\cdot 2^{-9}$}.
We consider the smooth initial data
\begin{align}
    \rho_0(x)=0.5+0.4\sin(\pi x),
\end{align}
and $\Dx=\frac{1}{20}\cdot2^{-n}$ for $n=0,\dots,5$. Additionally, we consider the following kernels with compact support on $[0,\eta]$
\begin{align*}
        \text{constant:} \quad\omega_{\eta}(x)=\frac{1}{\eta},\qquad \text{linear:} \quad            \omega_{\eta}(x)=\frac{2}{\eta}(1-\frac{x}{\eta}),\qquad
        \text{concave:}             \quad\omega_{\eta}(x)=3\frac{\eta^{2}-x^2}{2\eta^3}.
\end{align*}
\rv{Table \ref{table 3} shows the corresponding $L^1$-errors and demonstrates that all schemes converge to the reference solution with the expected order of convergence.
Similar as before, the slopes \eqref{eq:slopesv2} result in a more accurate approximation of the solution, although the differences to \eqref{eq:slopesv1} are small.}
\begin{table}[hbt!]
\centering
\caption{\rv{$L^1$ errors and convergence rates for the smooth initial data and the Arrhenius-type model}}\label{table 3}
    \begin{tabular}{c c | c  c | c  c | c  c | c  c }
    &  &\multicolumn{2}{c|}{LxF} & \multicolumn{2}{c|}{2nd order LxF} &\multicolumn{2}{c|}{NT with \eqref{eq:slopesv1}} & \multicolumn{2}{c}{NT with \eqref{eq:slopesv2}}\\
& $n$& $L^1$ error & c.r.& $L^1$ error & c.r.& $L^1$ error & c.r.& $L^1$ error & c.r.\\
    \hline\hline
\multirow{6}{*}{const.}
&0&1.17e-02&-&2.95e-03&-&7.52e-03&-&7.50e-03&-\\
&1&5.89e-03&0.99&8.15e-04&1.85&2.03e-03&1.89&2.02e-03&1.90\\
&2&2.95e-03&1.00&2.16e-04&1.92&5.40e-04&1.91&5.37e-04&1.91\\
&3&1.48e-03&1.00&5.82e-05&1.89&1.46e-04&1.89&1.45e-04&1.89\\
&4&7.39e-04&1.00&1.53e-05&1.93&3.87e-05&1.92&3.81e-05&1.93\\
&5&3.70e-04&1.00&3.89e-06&1.97&1.01e-05&1.94&9.85e-06&1.95\\
\hline
\multirow{6}{*}{linear}
&0&1.22e-02&-&2.80e-03&-&7.39e-03&-&7.37e-03&-\\
&1&6.07e-03&1.00&7.63e-04&1.88&1.95e-03&1.92&1.94e-03&1.93\\
&2&3.03e-03&1.00&2.01e-04&1.93&5.22e-04&1.90&5.19e-04&1.90\\
&3&1.51e-03&1.00&5.36e-05&1.90&1.39e-04&1.91&1.38e-04&1.91\\
&4&7.55e-04&1.00&1.39e-05&1.94&3.64e-05&1.93&3.60e-05&1.94\\
&5&3.77e-04&1.00&3.55e-06&1.97&9.42e-06&1.95&9.28e-06&1.96\\
\hline
\multirow{6}{*}{concave}
&0&1.20e-02&-&2.84e-03&-&7.42e-03&-&7.41e-03&-\\
&1&6.02e-03&1.00&7.76e-04&1.87&1.96e-03&1.92&1.96e-03&1.92\\
&2&3.01e-03&1.00&2.05e-04&1.92&5.28e-04&1.89&5.25e-04&1.90\\
&3&1.50e-03&1.00&5.48e-05&1.90&1.41e-04&1.91&1.40e-04&1.91\\
&4&7.51e-04&1.00&1.42e-05&1.94&3.69e-05&1.93&3.65e-05&1.94\\
&5&3.75e-04&1.00&3.63e-06&1.97&9.57e-06&1.95&9.41e-06&1.95\\
\hline\hline
\end{tabular}    
\end{table}

\subsection{Nonlocal systems with source term}

\subsubsection{Multilane traffic flow model}
We start by considering a system used to model multilane traffic flow introduced in \rv{\cite{bayen2022multilane,friedrich2020nonlocal,chiarello2024singular}} for simplicity with $N=2$:
\begin{align*}
    \partial_t \rho^1+ \partial_x (\rho^1 v(\omega_\eta\ast \rho^1))&=-S(\rho^1,\rho^2,\omega_\eta\ast \rho^1,\omega_\eta\ast \rho^2),\\
    \partial_t \rho^2+ \partial_x (\rho^2 v(\omega_\eta\ast \rho^2))&=S(\rho^1,\rho^2,\omega_\eta\ast \rho^1,\omega_\eta\ast \rho^2),
\end{align*}
and the source term is given similar to \cite{friedrich2020nonlocal} as
\[S(\rho^1,\rho^2,R^1,R^2)=(v(R^2)-v(R^1))\begin{cases}
    \rho^1(1-\rho^2),&\text{if }v(R^2)\geq v(R^1),\\
    \rho^2(1-\rho^1),&\text{if }v(R^2)< v(R^1).
\end{cases}\]
The kernel is forward looking with compact support on $[0,\eta]$ and unit integral. 
Here, we consider the linear decreasing kernel with $\eta=0.5$.
\rv{Although the fluxes are linear in the local variable the sources are nonlinear. Hence, the convergence is only guaranteed for the NT scheme \eqref{eq:slopesv2}. We are interested in} the following test case from \cite{bayen2022multilane,friedrich2020nonlocal}:
\[\rho_0^1(x)=q(2x-0.5),\quad \rho_0^2(x)=q(2x)\quad\text{with}\quad q(x)=4x^2(1-x^2)\chi_{(0,1)}(x),\]
\rv{and $v(\rho)=1-\rho^2$.}
The solution at $T=0.5$ is displayed in Figure \ref{fig:multilane} with \rv{$\Delta x=5\cdot 10^{-3}$}.
\rv{We use the same numerical schemes as in the previous section, only adding the source terms for the LxF schemes.}
Note that in this example we compute the solution \rv{on a domain large enough that no interaction of the solution with the boundaries occur and thus we use zero boundary conditions.} The reference solution is obtained with \rv{$\Delta x=10^{-4}$ and the NT scheme \eqref{eq:slopesv2}.
Similarly as before, there is no visible difference between the two NT schemes and only a small difference in comparsion to the second-order LxF type scheme.}
\begin{figure}
     \centering
     \setlength{\fwidth}{0.75\linewidth}
     \input{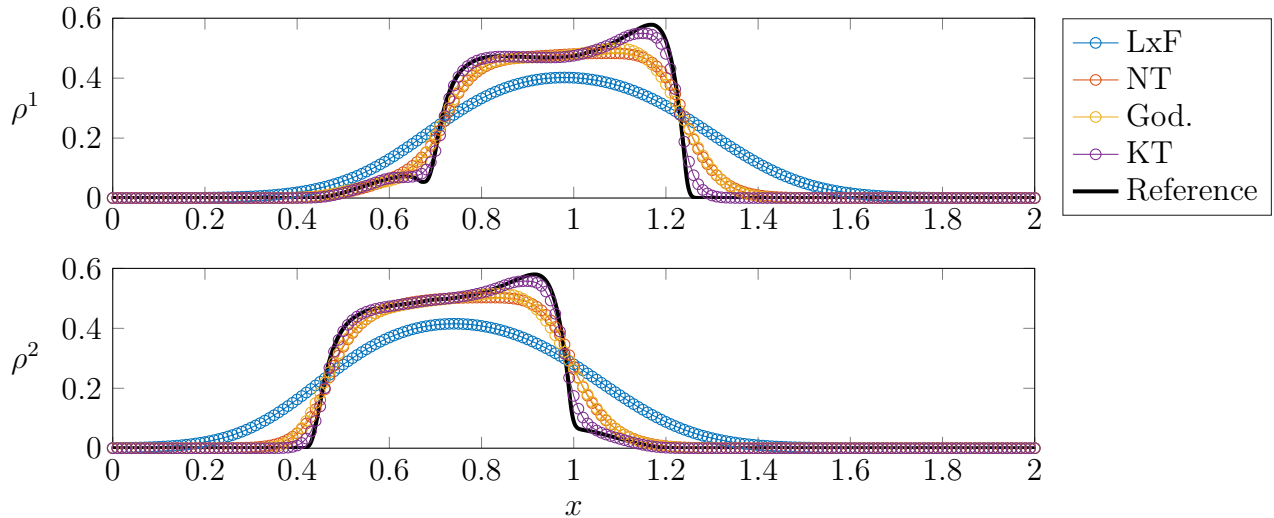}
     \caption{\rv{Numerical solutions of the nonlocal multilane model at $T=0.5$ obtained by different numerical schemes $\Delta x=5\cdot 10^{-3}$.}}\label{fig:multilane}
\end{figure}

The smooth initial data to test the convergence are given by
\[\rho_0^1(x)=0.5+0.5\sin(\pi x),\quad \rho_0^2(x)=0.25+0.25\cos(2\pi x),\]
on the interval $[-1,1]$ with the final time $T=0.15$. Here, the step sizes are given by $\Delta x=\frac{1}{20}\cdot 2^{-n}$, $n=0,\dots,5$ and \rv{$n=9$} for the reference solution obtained by the NT scheme \rv{with the slopes \eqref{eq:slopesv2}}.
The expected convergence orders can be seen in Table \ref{tab:convTest1L1Multilane}.
\rv{Again, we observe that the NT scheme \eqref{eq:slopesv1} converges to the correct solution.}
\begin{table}[hbt!]
\caption{\rv{$L^1$ errors and convergence rates for the smooth initial data and the multilane model.}}\label{tab:convTest1L1Multilane}
\centering
\begin{tabular}{ c | c  c | c  c | c  c | c  c}
&\multicolumn{2}{c|}{LxF} & \multicolumn{2}{c|}{2nd order LxF} &\multicolumn{2}{c|}{NT with \eqref{eq:slopesv1}} & \multicolumn{2}{c}{NT with \eqref{eq:slopesv2}}\\
$n$&$L^1$-error&c.r. &$L^1$-error&c.r. &$L^1$-error&c.r. &$L^1$-error&c.r. \\ \hline \hline
0&3.14e-01&-&6.37e-02&-&9.78e-02&-&9.77e-02&-\\
1&1.85e-01&0.76&2.26e-02&1.50&3.21e-02&1.61&3.21e-02&1.61\\
2&1.02e-01&0.86&6.92e-03&1.70&1.05e-02&1.62&1.05e-02&1.62\\
3&5.38e-02&0.92&1.95e-03&1.83&2.96e-03&1.82&2.95e-03&1.82\\
4&2.77e-02&0.96&5.34e-04&1.87&8.18e-04&1.86&8.16e-04&1.86\\
5&1.41e-02&0.98&1.42e-04&1.91&2.19e-04&1.90&2.18e-04&1.90\\
\hline\hline
\end{tabular} 
\end{table}
\subsubsection{Nonlocal Euler equations}
Recently in \cite{bhatnagar2021well}, a nonlocal extension of the Euler equation with relaxation was derived. In one space dimension the equations read:
\begin{equation}\label{eq:NLEuler}
\begin{aligned}
    \partial_t \rho +\partial_x (\rho (\omega \ast u))&=0,\\
    \partial_t u + u \partial_x u&=\rho (\omega \ast u-u).
\end{aligned}
\end{equation}
Here, a symmetric kernel with unit integral is considered. We consider the kernel defined by
\[\omega(x)=\rv{\frac{3(\eta^2-x^2)}{4\eta^3}}\chi_{[-\eta,\eta]}(x).\]
\rv{Interestingly, the systems involves a flux driven for $\rho$ by a nonlocal component and for $u$ by a local one, which makes it favorable for central schemes as LxF and NT, since they can be directly applied.
As before due to the nonlinearity in the source and flux the convergence result for \eqref{eq:slopesv1} does not hold.}
\par In \cite{bhatnagar2021well}, the local existence of \eqref{eq:NLEuler} has been established as well as a subcritical region for the initial data to avoid $u_x\to -\infty$, which is given by $u_0'(x)+\rho_0(x)\geq 0\ ,\forall x\in\R$.
We choose the initial data
\[\rho_0(x)=0.2+0.1\sin(\pi x),\qquad u_0(x)=0.4+\frac{0.3\cos(\pi x)}{\pi},\]
which fulfills the subcritical condition, to test the convergence. 
The final time is given by $T=0.15$, $\eta=0.05$, $\Delta x=\frac{1}{20}\cdot 2^{-n}$ with a reference solution computed by the \rv{NT scheme with \eqref{eq:slopesv2} and $n=9$}.
The expected convergence rates can be obtained in Table \ref{tab:convTest1L1Euler}.
\begin{table}[hbt!]
\caption{\rv{$L^1$ errors and convergence rates for the smooth initial data and the nonlocal Euler equations.}}\label{tab:convTest1L1Euler}
\centering
\begin{tabular}{ c | c  c | c  c | c  c | c  c}
&\multicolumn{2}{c|}{LxF} & \multicolumn{2}{c|}{2nd order LxF} &\multicolumn{2}{c|}{NT with \eqref{eq:slopesv1}} & \multicolumn{2}{c}{NT with \eqref{eq:slopesv2}}\\
$n$&$L^1$-error&c.r. &$L^1$-error&c.r. &$L^1$-error&c.r. &$L^1$-error&c.r. \\ \hline \hline
0&6.06e-02&-&9.43e-03&-&1.47e-02&-&1.47e-02&-\\
1&3.20e-02&0.92&2.70e-03&1.81&4.36e-03&1.75&4.35e-03&1.75\\
2&1.64e-02&0.96&7.55e-04&1.84&1.24e-03&1.81&1.24e-03&1.81\\
3&8.33e-03&0.98&2.14e-04&1.81&3.45e-04&1.85&3.45e-04&1.85\\
4&4.19e-03&0.99&5.86e-05&1.87&9.42e-05&1.87&9.41e-05&1.87\\
5&2.10e-03&1.00&1.56e-05&1.91&2.51e-05&1.91&2.50e-05&1.91\\
\hline\hline
\end{tabular} 
\end{table}
To test the performance of the different numerical schemes, we consider numerically the limit $\eta \to 0$.
Formally, we expect that the velocity converges to the \rv{Burgers'} equation (decoupled from the density).
The solution of the density strongly depends on the velocity.
We consider a Riemann problem in $u$, which results for the Burger's equation in a rarefaction wave.
Formally, this creates a vacuum in the density.
The solutions by the different schemes at $T=0.5$ and for $\eta=2\cdot 10^{-3}$, $\Delta x=10^{-3},$
\[\rho_0(x)=\begin{cases}
0.5, &x\leq 0,\\
1.5, &x>0,
\end{cases}\qquad u_0(x)=\begin{cases}
-1, &x\leq 0,\\
1, &x>0,
\end{cases}\]
are shown in Figure \ref{fig:Euler}.
The numerical approximations confirm the formal results.
Nevertheless, we note that to best of our knowledge the limit of $\eta \to 0$ has not been investigated yet and a further more deep analysis is necessary.
In particular, the numerical evidence might be deceiving as discussed in \cite{colombo2019role}. 
 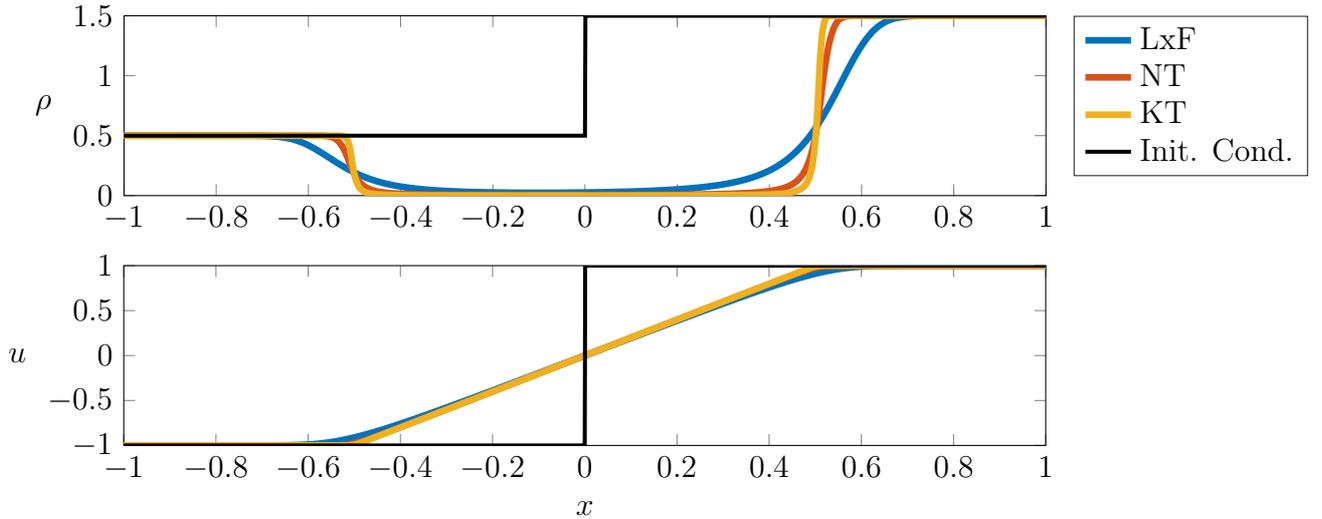
\begin{figure}
     \centering
     \setlength{\fwidth}{0.75\linewidth}
     \input{NLEulerNew}
     \caption{\rv{Numerical solutions of the nonlocal Euler equations at $T=0.5$ obtained by different numerical schemes for $\eta=2\cdot 10^{-3}$.}}\label{fig:Euler}
 \end{figure}
\subsection{Nonlocal in velocity}
In \eqref{eq:generalsystem} the convolution is computed between the kernel function and the state variable as a quantity of interest. 
Some nonlocal models consider different quantities, which might depend in a nonlinear fashion on the state variables.
For the particular example of traffic flow models this might be the convolution of the velocity, see e.g. \cite{friedrich2018godunov,friedrich2020micromacro,friedrich2022conservation}.
Following the strategy from Sections \ref{sec:staggered central} and \ref{sec:BV}, both NT schemes can be derived in a similar fashion.
Exemplarily, we consider the nonlocal generalized Aw-Rascle-Zhang (GARZ) model introduced in  \cite{friedrich2020micromacro}.
The modeling equations are given by
\begin{align*}
    \partial_t \rho +\partial_x (\rho (\omega_\eta\ast v(\rho, \rho/q)))&=0,\\
    \partial_t q +\partial_x (q (\omega_\eta\ast v(\rho, \rho/q)))&=0.
\end{align*}
Here, $\rho$ is the density, $q$ the momentum, $\omega_\eta$ a non-increasing kernel function with compact support on $[0,\eta]$ and unit integral, and $v$ is a suitable velocity function.
\par Due to the different structure of the nonlocal term slight modifications are necessary to the NT schemes.
To approximate the nonlocal term in the NT scheme, we can adapt \eqref{eq:RNT} in a straightforward manner by replacing $\rho$ with suitable evaluations of the velocity.
More modifications are necessary to approximate the temporal derivative of the nonlocal term, even though the computations are straightforward.
In particular, we have
\[\partial_t v(\rho,q/\rho)=\partial_t \rho(v_1(\rho,q/\rho)-q/\rho^2 v_2(\rho,q/\rho))+\partial_t q v_2(\rho,q/\rho),\]
with $v_1$ and $v_2$ the partial derivatives of $v$ with respect to the first and second argument.
Using this, we can compute $\partial_t\left(\omega_\eta\ast v(\rho, \rho/q)\right)(t,x)$ and additionally, the temporal derivative of $\rho$ and $q$ can be replaced by the spatial derivative of the flux.
For the NT scheme the resulting term is then approximated similar to \eqref{eq:system Rt} by using an appropriate quadrature rule and approximating the spatial derivatives by the minmod limiter \rv{as in \eqref{eq:slopesv1}. 
We note that the slope can be computed more accurately, similar to \eqref{eq:slopesv2}. For simplicity, we only consider the approximation by the minmod limiter in this subsection.}
\par \rv{We compare the numerical results to the LxF type schemes.}
The velocity function is given by $v(\rho,w)=w-6\rho$, the kernel is linearly decreasing with $\eta=0.1$. For the discontinuous test case we consider a jump in the momentum:
\[\rho_0(x)=0.05,\quad q_0(x)=\begin{cases}
    7/400,&\text{if }x\leq 0,\\
    1/25,&\text{if }x>0,
\end{cases}\]
with constant boundary conditions.
Note that this initial condition corresponds to a jump in the so-called Lagrangian maker $w=q/\rho$ from $0.35$ to $0.8$.
In particular, this jump remains in the analytic solution and is transported with a nonlocal speed.
Hence,  in Figure \ref{fig:GARZ} we display the approximate solutions at $T=1$ with \rv{$\Delta x=5\cdot 10^{-3}$} for $\rho$ and $w$.
We can obtain that the jump in $w$ is approximated, as expected, most accurately by the second-order schemes and the diffrences between the NT and second-order LxF scheme are small.
 \begin{figure}
     \centering
     \setlength{\fwidth}{0.75\linewidth}
     \input{NLGarzNew}
     \caption{\rv{Numerical solutions of the nonlocal GARZ model at $T=1$ obtained by different numerical schemes with $\Delta x=5\cdot 10^{-3}$.}}\label{fig:GARZ}
 \end{figure}
For the smooth test case we consider the interval $[-1,1]$ and
\[ \rho_0(x)=0.3+0.2\sin(\pi x),\quad q_0(x)=(0.3+0.2\sin(\pi x))(1.9+1.25\sin(\pi x)).\]
The reference solution is computed by the \rv{NT scheme and $\Delta x=\frac{1}{20}\cdot 2^{-9}.$}  Similar observations as in all the previous test cases can be made. The convergence rates are displayed in the Table \ref{tab:convTest1L1GARZ}.
\begin{table}[hbt!]
\caption{\rv{$L^1$ errors and convergence rates for the smooth initial data and the nonlocal GARZ model.}}\label{tab:convTest1L1GARZ}
\centering
    \begin{tabular}{ c | c  c | c  c | c  c }
          &\multicolumn{2}{c|}{LxF} & \multicolumn{2}{c|}{2nd order LxF} &\multicolumn{2}{c}{NT}\\
$n$&$L^1$-error&c.r. &$L^1$-error&c.r. &$L^1$-error&c.r. \\
\hline\hline
0&5.99e-01&-&8.55e-02&-&1.33e-01&-\\
1&3.67e-01&0.71&3.25e-02&1.40&4.42e-02&1.59\\
2&2.09e-01&0.82&1.02e-02&1.68&1.51e-02&1.55\\
3&1.12e-01&0.90&2.96e-03&1.78&4.36e-03&1.79\\
4&5.84e-02&0.94&8.22e-04&1.85&1.23e-03&1.82\\
5&2.98e-02&0.97&2.27e-04&1.85&3.42e-04&1.85\\
\hline\hline
\end{tabular} 
\end{table}
\section{\rv{Conclusion}}\label{sec:conclusion}
We considered non-staggered central approaches for systems of nonlocal balance laws. 
Based on the approach by Nessyahu and Tadmor in \cite{nessyahu1990non}, we extended the approaches \cite{GoatinScialanga2016,belkadi2022non} from scalar nonlocal conservation laws to systems with source terms. We prove the weak-$*$ convergence towards weak solutions in the linear case. With additional regularity assumptions we compute the slopes more accurately such that we can show the strong convergence in the nonlinear case to weak entropy solution.
The main advantage of the NT scheme is that it can be easily applied to any system of nonlocal balance laws. Several numerical examples demonstrate this.
Additionally, we explain how to derive these schemes for further nonlocal models, which might involve a nonlinear function of the state variables.

To reduce the numerical dissipation in the future, we aim to investigate an approach by Kurganov and Tadmor \cite{kurganov2000new}, which has been considered for local conservation laws and splits the cells into smooth and non-smooth parts. In addition, we aim to investigate the convergence of the second-order extension of the LxF scheme to systems of nonlocal balance laws introduced in Section \ref{sec:numerical}.

\section*{Funding}
J.~F. is supported by the German Research Foundation (DFG) through SPP 2410 `Hyperbolic Balance Laws in Fluid Mechanics: Complexity, Scales, Randomness' under grant FR 4850/1-1. S.~R. is supported by NBHM, DAE, India (Ref. No. 02011/46/2021 NBHM(R.P.)/R \& D II/14874).
J.~.F. and S.~R. acknowledge support from the Indo-German Science and Technology Centre (IGSTC) through the PECFAR-2024 grant.
\section*{Conflict of interest}
The authors declare that there is no conflict of interest.

\begin{appendix}
    \section{\rv{Positivity preservation for systems of nonlocal conservation and scalar nonlocal balance laws}}\label{app:pos}

Under additional assumptions on the source term the numerical approximations of \eqref{eq:generalsystemsimple} satisfy a strict lower bound $\rho_m$ which implies the positivity of solutions for a sufficiently small time horizon. In particular, we require:
    \begin{assumption}\label{setpos} Additionally to the Assumption \ref{set1}, we assume
        \begin{itemize}
            \item\ \emph{Source term}: if $\rho^k=\rho_m$ then $S_k((\rho^1,\dots,\rho^{k-1},\rho_m,\rho^{k+1},\dots,\rho^N)^T,\bm{R})\geq 0$ and either
            $S_k(\brho,\bm{R})\leq 0$ or $S_k(\brho,\bm{R})=S(\rho^k,\bm{R})$ for all $\brho \in \R^N$ and $\bm{R}\in\R^m$
        \end{itemize}
    \end{assumption}
    The first assumption, is rather standard to ensure the positivity of the solution, see e.g. \cite{bayen2022multilane,friedrich2020nonlocal}.
    The additional assumptions let us control the influence coming from the source term after the half-time step, see \eqref{eq:estimateS} below.
    These assumptions include systems of nonlocal conservation laws and scalar nonlocal balance laws.
    \begin{theorem}\label{thm:pos} Let the Assumption \ref{setpos} hold, then under the CFL condition
\begin{equation}
    \frac{\Delta t L_S}{2}\leq \tau, \quad {\lambda} {L_F}\leq \kappa\quad \text{and}\quad  \kappa +\tau \leq \frac{\sqrt{2}-1}{2},
\end{equation} there exists $n^*\in \N$ such that the approximate solutions of the NT scheme \eqref{eq:system staggered_final} stay bounded from below by $\rho_m\geq 0$.
\end{theorem}
\begin{proof}\
  Similar to the proof of the $L^\infty$ bound, we prove the statement by induction and only for the approximations on the staggered grid $\rho_{j+1/2}^{k,n+1}$.
  For $n=0$ the statement is obvious. 
  For simplicity and with some abuse of notation, we denote by $\tilde{\brho}=(\rho^1,\dots,\rho^{k-1},\rho_m,\rho^{k+1},\dots,\rho^N)^T$, the vector which has $\rho_m$ as the $k$-th entry.
  We start with the following observation:
  \begin{align*}
      \rho_j^{k,n}+\Delta t S_k(\brho_j^n,\bm{R}_j^n)\geq \rho_j^{k,n}+\Delta t (S_k(\brho_j^n,\bm{R}_j^n)-S_k(\tilde \brho_j^n,\bm{R}_j^n))\geq \rho_m.
  \end{align*}
 Then, the difference in the flux can be estimated by adding some zeros 
  \begin{align*}
  &| F_k(\rho_{j+1}^{k,n+1/2},\bm{R}_{j+1}^{n+1/2})-F_k(\rho_{j}^{k,n+1/2},\bm{R}_{j}^{n+1/2})|\\
  &=
  | F_k(\rho_{j+1}^{k,n+1/2},\bm{R}_{j+1}^{n+1/2})-F_k(\rho_{j}^{k,n+1/2},\bm{R}_{j}^{n+1/2})-F_k(\rho_m,\bm{R}_{j+1}^{n+1/2})+F_k(\rho_m,\bm{R}_{j}^{n+1/2})|\\
  & \leq {L_F} \left( \rho_{j+1}^{k,n}+\rho_j^{k,n}{-2\rho_m}+\frac{\Dt}{2} \left(|\sigma^k(\brho_j^k,\bm{R}_j^k)|+|\sigma^k(\brho_{j+1}^k,\bm{R}_{j+1}^k)|\right){+
  \frac{\Delta t}{2}\left(
  S_{k}(\brho_{j+1}^{n},\bm{R_{j+1}}^{n})
  + S_{k}(\brho_{j}^{n},\bm{R_{j}}^{n})
  \right)}\right).
  \end{align*}
   Similarly, we can estimate
  \begin{align*}
      \Delta x |\sigma^k(\brho_j^n,\bm{R}_j^n)| &\leq | F_k(\rho_{j+1}^{k,n},\bm{R}_{j+1}^{n})-F_k(\rho_{j}^{k,n},\bm{R}_{j}^{n})|\leq {L_F} (\rho_{j+1}^{k,n}+\rho_j^{k,n}{-2\rho_m}).
  \end{align*}
  Note that we have the same estimate for $ \Delta x |\sigma^k(\brho_{j+1}^n,\bm{R}_{j+1}^n)|\leq  {L_F} (\rho_{j+1}^{k,n}+\rho_j^{k,n}{-2\rho_m})$.
  Furthermore, the additional assumptions on the source allow for the following estimate
  \begin{align}\label{eq:estimateS}
  S_{k}(\brho_{j+1}^{n},\bm{R_{j+1}}^{n})
  + S_{k}(\brho_{j}^{n},\bm{R_{j}}^{n})
  \leq L_S(\rho_{j+1}^{k,n}+\rho_j^{k,n}{-2\rho_m})      
  \end{align}
  by either using $S(\brho_j^n,\bm{R}_j^n)\leq 0$ or adding $S(\rho_m,\bm{R}_j^n)=0$. Hence, we obtain
  \begin{align*}
      | F_k(\rho_{j+1}^{k,n+1/2},\bm{R}_{j+1}^{n+1/2})-F_k(\rho_{j}^{k,n+1/2},\bm{R}_{j}^{n+1/2})|\leq {L_F} \left(1+{L_F}{\lambda}+{\frac{\Delta t L_S}{2}}\right) (\rho_{j+1}^{k,n}+\rho_j^{k,n}{-2\rho_m}).
  \end{align*}
  Now, we can proceed similarly for the source terms at the half time step
  \begin{align*}        &S_k(\brho_{j+1}^{n+1/2},\bm{R}_{j+1}^{n+1/2})+S_k(\brho_{j}^{n+1/2},\bm{R}_{j}^{n+1/2})\\
      &\geq S_k(\brho_{j+1}^{n+1/2},\bm{R}_{j+1}^{n+1/2})+S_k(\brho_{j}^{n+1/2},\bm{R}_{j}^{n+1/2})-S_k(\tilde{\brho}_{j+1}^{n+1/2},\bm{R}_{j+1}^{n+1/2})-S_k(\tilde{\brho}_{j}^{n+1/2},\bm{R}_{j}^{n+1/2})\\
      &\geq -L_S (\rho_{j+1}^{k,n}+\rho_{j}^{k,n}+\frac{\Delta t}{2}(S_{k}(\brho_{j+1}^{n},\bm{R_{j+1}}^{n})
  + S_{k}(\brho_{j}^{n},\bm{R_{j}}^{n}))-2\rho_m)
  \\
  &\geq -L_S \left(1+{L_F}{\lambda}+{\frac{\Delta t L_S}{2}}\right) \left(\rho_{j+1}^{k,n}+\rho_{j}^{k,n}-2\rho_m\right)
  \end{align*}
  For the last inequality we used \eqref{eq:estimateS}.
  This allows us to prove the desired lower bound {
  \begin{align*}
      \rho_{j+1/2}^{k,n+1}&\geq \frac12 (\rho_{j+1}^{k,n}+\rho_j^{k,n})-\frac14|\rho_{j+1}^{k,n}-\rho_j^{k,n}|-{ \left(\frac{\Delta t L_s}{2}+{\lambda}{L_F} \right)\left(1+{L_F}{\lambda}+{\frac{\Delta t L_S}{2}}\right) (\rho_{j+1}^{k,n}+\rho_j^{k,n}-2\rho_m)}\\
      &\geq \frac12 (\rho_{j+1}^{k,n}+\rho_j^{k,n})-\frac14|\rho_{j+1}^{k,n}-\rho_j^{k,n}|- {(\tau + \kappa)}\left(1+\kappa+\tau\right) (\rho_{j+1}^{k,n}+\rho_j^{k,n}-2\rho_m)\\
      &\geq \frac12 (\rho_{j+1}^{k,n}+\rho_j^{k,n})-\frac14|\rho_{j+1}^{k,n}-\rho_j^{k,n}|- \frac14 (\rho_{j+1}^{k,n}+\rho_j^{k,n}-2\rho_m)=\frac12 \min\{\rho_{j+1}^{k,n},\rho_j^{k,n}\}+\frac12 \rho_m\geq \rho_m.
  \end{align*}
  Note that the conditions on $\tau$ and $\kappa$ guarantee $(\kappa+\tau)(1+\kappa+\tau)<1/4$. If the Assumption \ref{setpos}, in particular the Lipschitz continuity of $F_k$ and $S_k$, are not valid anymore, we set $n^*=n$, otherwise we can repeat the procedure recursively.}
\end{proof}
{\begin{remark}
Note that similar to other works, e.g.  \cite{amorim2015numerical,manoj2025convergence}, further assumptions on the flux ensure upper bounds on the solution {independent of the final time $T$}.
In particular, if there exists a $\rho_M$ such that $\rho^k_0(x)\leq \rho_M$ for $x\in\R$, $F_k(\rho_M,\bm{R})=0$ and $S_k((\rho^1,\dots,\rho^{k-1},\rho_M,\rho^{k+1},\dots,\rho^N)^T,\bm{R})\leq 0$ for $\bm{R}\in\R^m$, the analytical solution posses an upper bound $\rho_M$.
Following the proof above we need to either assume $S_k(\brho,\bm{R}) = 0$ or $S_k(\brho,\bm{R})=S(\rho^k,\bm{R})$ for all $\brho \in \R^N$ and $\bm{R}\in\R^m$ to guarantee the upper and lower bound.
We note that while the existence of $\rho_m$ is given in several models, this does not apply to $\rho_M$.
\end{remark}
}

\end{appendix}
\bibliographystyle{abbrv}
\bibliography{reference.bib}

\end{document}

%% file: GRid_cell_NT_staggered.tex
	\newcommand\checkindexnew[1]{
		\pgfmathsetmacro{\var}{#1}
		\pgfmathparse{ifthenelse(\var==0, "",ifthenelse(\var>0, "+#1","#1"))} \pgfmathresult}%


\begin{tikzpicture}[scale=1.5]

\draw[->] (-1,-0.5) -- (4,-0.5) node[anchor=north] {$x$};
    \node[below] at (0, -0.5) {$x_{j-1}$};
    \node[below] at (1.5, -0.5) {$x_{j}$};
    \node[below] at (3, -0.5) {$x_{j+1}$};

\draw[thick] (-0.5, 0.4) -- (0.9, 0.4);
\draw[thick] (0.9,0.8) -- (2.3,0.8);  
\draw[thick] (2.3,1.2) -- (3.7,1.2);

    \coordinate (mid1) at (0.2, 0.4);
    \coordinate (mid2) at (1.6, 0.8);
    \coordinate (mid3) at (3, 1.2);
    
    \draw[thick] (mid1) -- ++(0.7, 0.15);  
    \draw[thick] (mid1) -- ++(-0.7, -0.15); 
    
    \draw[thick] (mid2) -- ++(0.7, 0.15);  
    \draw[thick] (mid2) -- ++(-0.7, -0.15); 
    
    \draw[thick] (mid3) -- ++(0.7, 0.15);  
    \draw[thick] (mid3) -- ++(-0.7, -0.15); 

\node at (0.2, 0.1) {${\rho}^{k,n}_{j-1}$};
\node at (1.6, 0.5) {${\rho}^{k,n}_{j}$};
\node at (3, 0.9) {${\rho}^{k,n}_{j+1}$};

\draw[dashed] (0.2, 0.4) -- (0.2, 1.6); 
\draw[dashed] (1.6, 0.8) -- (1.6, 1.9); 
\draw[dashed] (3, 1.2) -- (3, 1.9); 

  \draw[thick] (0.2, 1.6) -- (1.6, 1.6);
 \draw[thick] (1.6,1.9) -- (3,1.9);

\node at (1, 1.2) {${\rho}^{k,n+1}_{j-1/2}$};
\node at (2.4, 1.5) {${\rho}^{k,n+1}_{j+1/2}$};
\node[above] at (1.6, 2.5) {${\rho}^{k,n+1}_{j}$};

    \coordinate (mid1) at (0.9, 1.6);
    \coordinate (mid2) at (2.3, 1.9);

    \draw[thick] (mid1) -- ++(0.7, 0.1);  
    \draw[thick] (mid1) -- ++(-0.7, -0.1); 
    
    \draw[thick] (mid2) -- ++(0.7, 0.1);  
    \draw[thick] (mid2) -- ++(-0.7, -0.1); 
    

\draw[dashed] (0.9, 1.6) -- (0.9, 2.5); 
\draw[dashed] (2.3, 1.9) -- (2.3, 2.5); 
  \draw[thick] (0.9, 2.5) -- (2.3, 2.5);
\end{tikzpicture}


%% file: ArrheniusdiscNew.tex
%
%
\definecolor{mycolor1}{rgb}{0.00000,0.44700,0.74100}%
\definecolor{mycolor2}{rgb}{0.85000,0.32500,0.09800}%
\definecolor{mycolor3}{rgb}{0.92900,0.69400,0.12500}%
\definecolor{mycolor4}{rgb}{0.49400,0.18400,0.55600}%
\begin{tikzpicture}

\begin{axis}[%
width=0.951\fwidth,
height=0.3\fwidth,
at={(0\fwidth,0\fwidth)},
scale only axis,
xmin=-0.5,
xmax=-0.25,
xlabel={$x$},
ymin=0.3,
ymax=1.1,
ylabel={$\rho$},
ylabel style={rotate=-90},
axis background/.style={fill=white},
legend style={legend cell align=left, align=left, draw=white!15!black},
legend pos=outer north east
]
\addplot [color=mycolor1, dashdotted, mark=o, mark options={solid, mycolor1}]
  table[row sep=crcr]{%
-0.60625	0.277967713389395\\
-0.6	0.283564716307737\\
-0.59375	0.289190076355883\\
-0.5875	0.294841632038571\\
-0.58125	0.300517264137185\\
-0.575	0.30621489443325\\
-0.56875	0.311932486106604\\
-0.5625	0.317668044527485\\
-0.55625	0.323419618697709\\
-0.55	0.329185306209073\\
-0.54375	0.334963273637039\\
-0.5375	0.340751836100225\\
-0.53125	0.346549747530187\\
-0.525	0.352357207663957\\
-0.51875	0.358179226172997\\
-0.5125	0.364036518603538\\
-0.50625	0.369999803331621\\
-0.5	0.376294582753768\\
-0.49375	0.383609547618377\\
-0.4875	0.393950657884056\\
-0.48125	0.412709556058177\\
-0.475	0.452052495777922\\
-0.46875	0.529914511834827\\
-0.4625	0.649928183637611\\
-0.45625	0.776219474140357\\
-0.45	0.8604845279284\\
-0.44375	0.894103303035928\\
-0.4375	0.899917390333123\\
-0.43125	0.89632310521063\\
-0.425	0.890226762137642\\
-0.41875	0.883510707417104\\
-0.4125	0.876645972356748\\
-0.40625	0.869754453017632\\
-0.4	0.862873123837909\\
-0.39375	0.856017520849162\\
-0.3875	0.849197036589698\\
-0.38125	0.842418706973069\\
-0.375	0.835688231692245\\
-0.36875	0.829010317620873\\
-0.3625	0.822388847870961\\
-0.35625	0.815826998730023\\
-0.35	0.809327334988478\\
-0.34375	0.802891891937633\\
-0.3375	0.796522246885232\\
-0.33125	0.790219581669597\\
-0.325	0.783984737267255\\
-0.31875	0.777818261434754\\
-0.3125	0.771720450224736\\
-0.30625	0.765691384128982\\
-0.3	0.759730959519568\\
-0.29375	0.753838915983014\\
-0.2875	0.748014860072054\\
-0.28125	0.742258285935773\\
-0.275	0.736568593231538\\
-0.26875	0.730945102671016\\
-0.2625	0.725387069507481\\
-0.25625	0.719893695231807\\
-0.25	0.714464137709823\\
-0.24375	0.709097519963249\\
-0.2375	0.703792937770005\\
-0.23125	0.698549466236622\\
-0.225	0.693366165475498\\
-0.21875	0.688242085502389\\
-0.2125	0.683176270454432\\
-0.20625	0.678167762215986\\
-0.2	0.673215603528217\\
-0.19375	0.668318840648549\\
-0.1875	0.663476525617584\\
-0.18125	0.658687718183693\\
-0.175	0.653951487429087\\
-0.16875	0.649266913135566\\
-0.1625	0.644633086923334\\
-0.15625	0.640049113192037\\
-0.15	0.635514109889503\\
-0.14375	0.6310272091305\\
-0.1375	0.626587557685013\\
-0.13125	0.622194317353124\\
-0.125	0.617846665241465\\
-0.11875	0.613543793954345\\
-0.1125	0.609284911711063\\
-0.10625	0.605069242399453\\
-0.1	0.600896025574534\\
-0.09375	0.596764516409979\\
-0.0874999999999999	0.592673985609226\\
-0.0812499999999999	0.588623719282175\\
-0.075	0.584613018792717\\
-0.06875	0.580641200581654\\
-0.0625	0.576707595969058\\
-0.0562499999999999	0.57281155093956\\
-0.0499999999999999	0.568952425913673\\
-0.04375	0.565129595507839\\
-0.0375	0.56134244828555\\
-0.03125	0.557590386501622\\
-0.0249999999999999	0.553872825841399\\
-0.0187499999999999	0.550189195156464\\
-0.0125	0.546538936198219\\
-0.00624999999999998	0.542921503350504\\
0	0.539336363362295\\
0.00624999999999998	0.535782995081349\\
};
\addlegendentry{LxF}

\addplot [color=mycolor2, dashed, mark=o, mark options={solid, mycolor2}]
  table[row sep=crcr]{%
-0.60625	0.278234965845654\\
-0.6	0.283895941984082\\
-0.59375	0.289578799650197\\
-0.5875	0.295281276899678\\
-0.58125	0.301002579794664\\
-0.575	0.306742228554978\\
-0.56875	0.312499262159064\\
-0.5625	0.318272716433938\\
-0.55625	0.324062058939402\\
-0.55	0.329866877432016\\
-0.54375	0.335686450235288\\
-0.5375	0.341519651657509\\
-0.53125	0.347364986254193\\
-0.525	0.35322072896251\\
-0.51875	0.359085779721502\\
-0.5125	0.364956666921752\\
-0.50625	0.370822520481676\\
-0.5	0.376680738559493\\
-0.49375	0.382567990434517\\
-0.4875	0.388766930897354\\
-0.48125	0.397579423388233\\
-0.475	0.42661405026098\\
-0.46875	0.579886182743032\\
-0.4625	0.853590353435063\\
-0.45625	0.951136380104926\\
-0.45	0.956590485080332\\
-0.44375	0.945067751902369\\
-0.4375	0.934507754779226\\
-0.43125	0.925256415945918\\
-0.425	0.916205154852058\\
-0.41875	0.907301931127577\\
-0.4125	0.898563321767055\\
-0.40625	0.889983732332208\\
-0.4	0.881558174111764\\
-0.39375	0.873283573404812\\
-0.3875	0.865156730086198\\
-0.38125	0.857174210961329\\
-0.375	0.849332498610832\\
-0.36875	0.841628032172074\\
-0.3625	0.83405722328107\\
-0.35625	0.826616481882823\\
-0.35	0.819302244889666\\
-0.34375	0.812111000087534\\
-0.3375	0.805039303707546\\
-0.33125	0.798083792615891\\
-0.325	0.791241192413779\\
-0.31875	0.784508322502157\\
-0.3125	0.77788209889908\\
-0.30625	0.771359535393635\\
-0.3	0.764937743471847\\
-0.29375	0.758613931341243\\
-0.2875	0.752385402300307\\
-0.28125	0.746249552639045\\
-0.275	0.740203869211919\\
-0.26875	0.7342459267904\\
-0.2625	0.728373385276648\\
-0.25625	0.722583986840179\\
-0.25	0.716875553024326\\
-0.24375	0.711245981857764\\
-0.2375	0.705693244997465\\
-0.23125	0.700215384922585\\
-0.225	0.69481051219345\\
-0.21875	0.689476802785727\\
-0.2125	0.684212495506665\\
-0.20625	0.679015889497818\\
-0.2	0.673885341826775\\
-0.19375	0.668819265168972\\
-0.1875	0.66381612557955\\
-0.18125	0.658874440354363\\
-0.175	0.653992775978623\\
-0.16875	0.649169746161189\\
-0.1625	0.644404009952161\\
-0.15625	0.639694269941209\\
-0.15	0.63503927053391\\
-0.14375	0.630437796303263\\
-0.1375	0.625888670413502\\
-0.13125	0.621390753113315\\
-0.125	0.616942940295583\\
-0.11875	0.612544162120805\\
-0.1125	0.608193381701399\\
-0.10625	0.603889593844161\\
-0.1	0.599631823848233\\
-0.09375	0.595419126356002\\
-0.0874999999999999	0.591250584254487\\
-0.0812499999999999	0.587125307624806\\
-0.075	0.583042432737479\\
-0.06875	0.579001121091355\\
-0.0625	0.575000558494103\\
-0.0562499999999999	0.571039954182288\\
-0.0499999999999999	0.567118539979189\\
-0.04375	0.563235569488584\\
-0.0375	0.559390317322678\\
-0.03125	0.555582078362387\\
-0.0249999999999999	0.551810167048443\\
-0.0187499999999999	0.548073916702369\\
-0.0125	0.544372678876678\\
-0.00624999999999998	0.540705822733079\\
0	0.53707273444673\\
0.00624999999999998	0.533472816634885\\
};
\addlegendentry{2nd order LxF}

\addplot [color=mycolor3, dotted, mark=o, mark options={solid, mycolor3}]
  table[row sep=crcr]{%
-0.60625	0.277145181067124\\
-0.6	0.282858396005227\\
-0.59375	0.288581084367384\\
-0.5875	0.294312743239749\\
-0.58125	0.300055251884301\\
-0.575	0.305810080690509\\
-0.56875	0.311577784251424\\
-0.5625	0.31735843465265\\
-0.55625	0.323152179149985\\
-0.55	0.328959859525997\\
-0.54375	0.334784247774914\\
-0.5375	0.340632729570366\\
-0.53125	0.346523019039019\\
-0.525	0.352495564812223\\
-0.51875	0.358640376258301\\
-0.5125	0.365153873682665\\
-0.50625	0.37245707474685\\
-0.5	0.381436238008488\\
-0.49375	0.393919779794386\\
-0.4875	0.413553870484356\\
-0.48125	0.447190589143027\\
-0.475	0.506404022339981\\
-0.46875	0.60685310999512\\
-0.4625	0.738859914374532\\
-0.45625	0.829274183439929\\
-0.45	0.878565798243478\\
-0.44375	0.902214471413763\\
-0.4375	0.911788083084284\\
-0.43125	0.914434995962493\\
-0.425	0.914421525640906\\
-0.41875	0.91230005347249\\
-0.4125	0.906745594438528\\
-0.40625	0.897633405905582\\
-0.4	0.886076718991629\\
-0.39375	0.875400613801128\\
-0.3875	0.86611072237907\\
-0.38125	0.857585099197345\\
-0.375	0.849507847198119\\
-0.36875	0.84171745260722\\
-0.3625	0.834130588280182\\
-0.35625	0.826704025262286\\
-0.35	0.819415107442514\\
-0.34375	0.81225160507314\\
-0.3375	0.805206458426928\\
-0.33125	0.798275090511153\\
-0.325	0.791454063389141\\
-0.31875	0.784740430391187\\
-0.3125	0.778131442549582\\
-0.30625	0.771624429833732\\
-0.3	0.765216764075238\\
-0.29375	0.758905856440488\\
-0.2875	0.752689166624709\\
-0.28125	0.746564213569606\\
-0.275	0.740528583836439\\
-0.26875	0.734579936744015\\
-0.2625	0.72871600662214\\
-0.25625	0.722934602915991\\
-0.25	0.717233608876064\\
-0.24375	0.711610979424102\\
-0.2375	0.706064738613927\\
-0.23125	0.700592976956515\\
-0.225	0.695193848765764\\
-0.21875	0.689865569604367\\
-0.2125	0.684606413861008\\
-0.20625	0.679414712462977\\
-0.2	0.674288850714996\\
-0.19375	0.669227266250415\\
-0.1875	0.664228447081012\\
-0.18125	0.659290929733988\\
-0.175	0.654413297467942\\
-0.16875	0.649594178562597\\
-0.1625	0.644832244679393\\
-0.15625	0.640126209291316\\
-0.15	0.635474826180374\\
-0.14375	0.630876887999951\\
-0.1375	0.626331224897278\\
-0.13125	0.62183670318905\\
-0.125	0.617392224081603\\
-0.11875	0.61299672242686\\
-0.1125	0.608649165506894\\
-0.10625	0.604348551843633\\
-0.1	0.600093910035357\\
-0.09375	0.595884297627345\\
-0.0874999999999999	0.591718800031561\\
-0.0812499999999999	0.587596529506366\\
-0.075	0.58351662418906\\
-0.06875	0.579478247167702\\
-0.0625	0.575480585587471\\
-0.0562499999999999	0.571522849794162\\
-0.0499999999999999	0.567604272519373\\
-0.04375	0.563724108112161\\
-0.0375	0.559881631822155\\
-0.03125	0.556076139138772\\
-0.0249999999999999	0.552306945189356\\
-0.0187499999999999	0.548573384194515\\
-0.0125	0.544874808971427\\
-0.00624999999999998	0.54121059046669\\
0	0.53758011729217\\
0.00624999999999998	0.533982795232697\\
};
\addlegendentry{NT with \eqref{eq:slopesv1}}

\addplot [color=mycolor4, dashdotted, mark=o, mark options={solid, mycolor4}]
  table[row sep=crcr]{%
-0.60625	0.277143545855231\\
-0.6	0.282856828928422\\
-0.59375	0.288579634143413\\
-0.5875	0.294311408343124\\
-0.58125	0.300054044435302\\
-0.575	0.305809014967499\\
-0.56875	0.311576853827556\\
-0.5625	0.317357614829018\\
-0.55625	0.323151448182017\\
-0.55	0.328959210455366\\
-0.54375	0.334783676234145\\
-0.5375	0.340632203398183\\
-0.53125	0.346522435048111\\
-0.525	0.352494673473953\\
-0.51875	0.35863871053853\\
-0.5125	0.365150918101867\\
-0.50625	0.372452787893515\\
-0.5	0.381431909120766\\
-0.49375	0.39391641265303\\
-0.4875	0.413548504245867\\
-0.48125	0.447180013006356\\
-0.475	0.506383382848362\\
-0.46875	0.606815451093964\\
-0.4625	0.738824246403036\\
-0.45625	0.829246333730588\\
-0.45	0.878544286869888\\
-0.44375	0.902197715094371\\
-0.4375	0.911776229831513\\
-0.43125	0.914431895077452\\
-0.425	0.914422028971981\\
-0.41875	0.912302330346913\\
-0.4125	0.90675092420612\\
-0.40625	0.897640306065042\\
-0.4	0.886083088889904\\
-0.39375	0.875406483181314\\
-0.3875	0.866117054581696\\
-0.38125	0.857591966326944\\
-0.375	0.849515110448242\\
-0.36875	0.841724931233375\\
-0.3625	0.834138111780781\\
-0.35625	0.826711454275624\\
-0.35	0.819422339882527\\
-0.34375	0.812258573892155\\
-0.3375	0.805213125502893\\
-0.33125	0.79828143952516\\
-0.325	0.791460093203254\\
-0.31875	0.784746149571288\\
-0.3125	0.778136865191201\\
-0.30625	0.771629572596245\\
-0.3	0.765221644219449\\
-0.29375	0.758910490636225\\
-0.2875	0.752693570308353\\
-0.28125	0.746568400664585\\
-0.275	0.740532566692143\\
-0.26875	0.734583726192555\\
-0.2625	0.728719612090328\\
-0.25625	0.722938032556601\\
-0.25	0.717236869698512\\
-0.24375	0.711614077415048\\
-0.2375	0.706067678847589\\
-0.23125	0.700595763704341\\
-0.225	0.695196485622644\\
-0.21875	0.689868059645087\\
-0.2125	0.684608759821453\\
-0.20625	0.679416916919632\\
-0.2	0.674290916238338\\
-0.19375	0.669229195532973\\
-0.1875	0.664230243055228\\
-0.18125	0.659292595675347\\
-0.175	0.654414837054365\\
-0.16875	0.649595595881824\\
-0.1625	0.644833544242888\\
-0.15625	0.640127396167256\\
-0.15	0.63547590634721\\
-0.14375	0.63087786895868\\
-0.1375	0.626332116529255\\
-0.13125	0.621837518908796\\
-0.125	0.617392980194087\\
-0.11875	0.612997434866611\\
-0.1125	0.608649847633979\\
-0.10625	0.604349214603919\\
-0.1	0.600094562738108\\
-0.09375	0.595884948634266\\
-0.0874999999999999	0.591719457367742\\
-0.0812499999999999	0.587597201366771\\
-0.075	0.583517319298822\\
-0.06875	0.579478974952335\\
-0.0625	0.575481356106605\\
-0.0562499999999999	0.571523673394353\\
-0.0499999999999999	0.56760515917775\\
-0.04375	0.563725066477531\\
-0.0375	0.559882668011291\\
-0.03125	0.55607725540516\\
-0.0249999999999999	0.552308138638819\\
-0.0187499999999999	0.548574645768071\\
-0.0125	0.544876122943516\\
-0.00624999999999998	0.541211934703553\\
0	0.53758146445687\\
0.00624999999999998	0.533984114991804\\
};
\addlegendentry{NT with \eqref{eq:slopesv2}}

\addplot [color=black, line width=1.5pt]
  table[row sep=crcr]{%
-0.60009765625	0.281039844170117\\
-0.58427734375	0.295341826104294\\
-0.56630859375	0.311724666424753\\
-0.54931640625	0.327352126301737\\
-0.52958984375	0.345635045524732\\
-0.47490234375	0.39641412384481\\
-0.4650390625	0.405539242641605\\
-0.46484375	0.405912906768129\\
-0.46474609375	0.406359137149258\\
-0.4646484375	0.407359786295606\\
-0.46455078125	0.40977294912473\\
-0.464453125	0.415760625931452\\
-0.46435546875	0.430639356724198\\
-0.4642578125	0.466674259913502\\
-0.46416015625	0.547754794857472\\
-0.46396484375	0.84999604435443\\
-0.4638671875	0.923718283162627\\
-0.46376953125	0.955534022414305\\
-0.463671875	0.968235034199208\\
-0.46357421875	0.973102006032496\\
-0.4634765625	0.974888834892408\\
-0.46337890625	0.975488282760572\\
-0.46328125	0.975643339416628\\
-0.46318359375	0.97565298619885\\
-0.4630859375	0.975585960797387\\
-0.462890625	0.975178087128493\\
-0.4625	0.974443098927178\\
-0.45830078125	0.967408780254614\\
-0.452734375	0.958254352967197\\
-0.4470703125	0.949131331370826\\
-0.44130859375	0.94004291669267\\
-0.435546875	0.931140984165499\\
-0.4296875	0.922272422762348\\
-0.42373046875	0.913439590337038\\
-0.4177734375	0.904785109345295\\
-0.41171875	0.896164913690207\\
-0.40556640625	0.887581044679861\\
-0.3994140625	0.879167587621348\\
-0.3931640625	0.870788817468418\\
-0.38681640625	0.862446475039582\\
-0.38037109375	0.854142158964149\\
-0.373828125	0.845877333103685\\
-0.36728515625	0.837773105338779\\
-0.36064453125	0.82970627941878\\
-0.35390625	0.821678107823023\\
-0.3470703125	0.813689732843663\\
-0.34013671875	0.805742192398518\\
-0.33310546875	0.797836425613616\\
-0.3259765625	0.789973278182271\\
-0.31884765625	0.78225819005579\\
-0.31162109375	0.7745832298576\\
-0.304296875	0.766949123343224\\
-0.296875	0.759356519275077\\
-0.28935546875	0.751805993706397\\
-0.28173828125	0.744298054101518\\
-0.2740234375	0.736833143295162\\
-0.2662109375	0.729411643293362\\
-0.25830078125	0.722033878919261\\
-0.2501953125	0.71461150468258\\
-0.2419921875	0.70723656056641\\
-0.23369140625	0.699909110630602\\
-0.22529296875	0.692629178043858\\
-0.216796875	0.685396747989725\\
-0.208203125	0.678211770432648\\
-0.1994140625	0.67099468509554\\
-0.19052734375	0.663827671154359\\
-0.18154296875	0.656710489839287\\
-0.1724609375	0.649642880165201\\
-0.16318359375	0.642550559717321\\
-0.15380859375	0.635509815964802\\
-0.1443359375	0.628520241926884\\
-0.13466796875	0.621511236736897\\
-0.12490234375	0.614554984639463\\
-0.11494140625	0.607583204197257\\
-0.1048828125	0.600665464367912\\
-0.09462890625	0.593735724485399\\
-0.08427734375	0.586861048150676\\
-0.07373046875	0.579977557890467\\
-0.0630859374999999	0.573149918731098\\
-0.05224609375	0.566316340475257\\
-0.04130859375	0.559539192985411\\
-0.03017578125	0.552758699207306\\
-0.01884765625	0.545977067956135\\
-0.00742187499999991	0.539253377355988\\
9.76562499999112e-05	0.534890307807145\\
};
\addlegendentry{Reference}

\end{axis}

\end{tikzpicture}%

%% file: NLEulerNew.tex
%
%
\definecolor{mycolor1}{rgb}{0.06600,0.44300,0.74500}%
\definecolor{mycolor2}{rgb}{0.86600,0.32900,0.00000}%
\definecolor{mycolor3}{rgb}{0.92900,0.69400,0.12500}%
\definecolor{mycolor4}{rgb}{0.52100,0.08600,0.81900}%
\begin{tikzpicture}

\begin{axis}[%
width=0.951\fwidth,
height=0.2\fwidth,
at={(0\fwidth,0.25\fwidth)},
scale only axis,
xmin=-1,
xmax=1,
ymin=0,
ymax=1.5,
ylabel={$\rho$},
ylabel style={rotate=-90},
axis background/.style={fill=white},
legend style={legend cell align=left, align=left, draw=white!15!black},
legend pos=outer north east
]
\addplot [color=mycolor1, dashdotted, line width=2.5pt]
  table[row sep=crcr]{%
-1	0.5\\
-0.729	0.49967364814349\\
-0.709	0.498928157186316\\
-0.696	0.497826394875575\\
-0.686	0.496386496331827\\
-0.677	0.494439906287453\\
-0.669	0.492015141478147\\
-0.662	0.489217364733782\\
-0.655	0.485658639405182\\
-0.649	0.481905987614527\\
-0.643	0.477422081329317\\
-0.637	0.472133995164393\\
-0.631	0.465978315278194\\
-0.625	0.458905091997828\\
-0.619	0.450881491644259\\
-0.612	0.44030375426773\\
-0.605	0.428434422934871\\
-0.598	0.415334847857272\\
-0.59	0.398991501623369\\
-0.581	0.379114119699977\\
-0.57	0.35321249990028\\
-0.55	0.304014330514265\\
-0.533	0.262884094057468\\
-0.521	0.235513052916405\\
-0.511	0.214185595866414\\
-0.501	0.194392649257422\\
-0.492	0.177960920226083\\
-0.483	0.162850594579738\\
-0.474	0.149036166989056\\
-0.465	0.136465921641801\\
-0.456	0.125070483380154\\
-0.447	0.114769849055827\\
-0.438	0.105478892966573\\
-0.429	0.0971114810537417\\
-0.42	0.0895833975464362\\
-0.41	0.0821057156132148\\
-0.4	0.075462419330951\\
-0.39	0.0695595381641758\\
-0.379	0.0638205552222262\\
-0.367	0.0583470722243371\\
-0.355	0.0535848632799965\\
-0.342	0.0491153545714456\\
-0.328	0.044986459573823\\
-0.313	0.0412283427618731\\
-0.296	0.0376645113757799\\
-0.278	0.0345570830672688\\
-0.258	0.0317617374308694\\
-0.235	0.0292399302044228\\
-0.21	0.0271741096706613\\
-0.182	0.0255290634881198\\
-0.151	0.0243734270255695\\
-0.117	0.0237701406182225\\
-0.081	0.0237851830499527\\
-0.044	0.0244476329820278\\
-0.00700000000000012	0.0257667472840954\\
0.028	0.0276686839581817\\
0.0600000000000001	0.030046477061104\\
0.089	0.0328278026795141\\
0.116	0.0360649118856695\\
0.14	0.0395797140447445\\
0.162	0.043441961197406\\
0.182	0.0475950667749561\\
0.201	0.052221686368549\\
0.218	0.0570332603596295\\
0.234	0.0622501939418789\\
0.249	0.0678556916510606\\
0.263	0.0738190299185517\\
0.276	0.0800931770732471\\
0.288	0.0866128985632419\\
0.299	0.0932936442040877\\
0.31	0.10074322586652\\
0.32	0.10827262423494\\
0.33	0.116618385939601\\
0.34	0.125886302354618\\
0.349	0.135115402063178\\
0.358	0.145288946568252\\
0.367	0.156519232355718\\
0.375	0.167490433842307\\
0.383	0.17949449554132\\
0.391	0.192640027515858\\
0.399	0.20704635326382\\
0.407	0.222843965111559\\
0.415	0.240174742833264\\
0.423	0.259191821369735\\
0.431	0.280058961491561\\
0.439	0.302949240941443\\
0.447	0.328042844826132\\
0.455	0.355523696415287\\
0.463	0.385574639179217\\
0.471	0.418370867347307\\
0.479	0.454071319031233\\
0.487	0.492807810851441\\
0.495	0.534671827341017\\
0.504	0.585548996316005\\
0.513	0.640366528175821\\
0.523	0.705634817350778\\
0.534	0.782078314220295\\
0.548	0.884554315962327\\
0.579	1.1138226148076\\
0.589	1.18200764853616\\
0.597	1.23252046560079\\
0.605	1.27869474419955\\
0.612	1.31514040031657\\
0.618	1.34327304361543\\
0.624	1.36848252522554\\
0.63	1.39078808322341\\
0.636	1.41027242894648\\
0.641	1.4244522058125\\
0.646	1.43687132758447\\
0.651	1.44765011653803\\
0.656	1.45692062548137\\
0.661	1.46482186651196\\
0.666	1.47149537988442\\
0.671	1.47708130986998\\
0.676	1.48171510045407\\
0.681	1.48552486981642\\
0.687	1.48917550081447\\
0.693	1.49199533621709\\
0.7	1.49444971583124\\
0.707	1.49620935451042\\
0.716	1.49772926566307\\
0.727	1.4988264272554\\
0.741	1.4995197391498\\
0.764	1.49990273266198\\
0.825	1.49999934686419\\
1	1.4999999999998\\
};
\addlegendentry{LxF}

\addplot [color=mycolor2, dashed, line width=2.5pt]
  table[row sep=crcr]{%
-1	0.5\\
-0.572	0.499848855621092\\
-0.565	0.499518375299461\\
-0.56	0.49894616013831\\
-0.557	0.498345459628253\\
-0.554	0.497439891468006\\
-0.551	0.496097670015031\\
-0.549	0.494875551314149\\
-0.547	0.493317788656014\\
-0.545	0.491348762841831\\
-0.543	0.488881270407336\\
-0.541	0.48581645775879\\
-0.539	0.482044352219319\\
-0.537	0.477445130187439\\
-0.535	0.471891234072422\\
-0.533	0.465250394072648\\
-0.531	0.457389524087257\\
-0.529	0.448179342085772\\
-0.527	0.437499417599463\\
-0.525	0.425243181546757\\
-0.523	0.411322260533707\\
-0.521	0.395669336795925\\
-0.518	0.368847482683072\\
-0.515	0.33794575349066\\
-0.512	0.302898385350312\\
-0.507	0.241216146248601\\
-0.504	0.209176939973378\\
-0.501	0.181239576247901\\
-0.498	0.15714540970532\\
-0.495	0.136531957885445\\
-0.492	0.118992000473823\\
-0.489	0.104114871494149\\
-0.486	0.091512561179905\\
-0.483	0.0808340529166542\\
-0.48	0.0717713100063235\\
-0.477	0.0640598247933084\\
-0.474	0.0574759679742349\\
-0.471	0.0518327200817605\\
-0.468	0.0469748209176353\\
-0.465	0.0427739612422691\\
-0.462	0.0391243541689121\\
-0.459	0.0359388370019802\\
-0.456	0.0331455410222747\\
-0.452	0.0299301744132514\\
-0.448	0.0271943624061743\\
-0.444	0.0248494530468213\\
-0.439	0.0223631453109374\\
-0.434	0.0202742058359697\\
-0.428	0.0181810963707432\\
-0.422	0.0164431779416896\\
-0.415	0.0147642740889886\\
-0.407	0.0131994340465764\\
-0.398	0.0117800087306814\\
-0.387	0.0104073896096544\\
-0.374	0.00915511678629888\\
-0.359	0.00805925044259093\\
-0.341	0.00708208339183658\\
-0.318	0.0061888076522576\\
-0.289	0.00542355578795362\\
-0.252	0.00480206070957245\\
-0.203	0.00433434021118395\\
-0.137	0.00405678248902164\\
-0.03	0.00398770234648604\\
0.0529999999999999	0.00416650735166746\\
0.112	0.00463330062297063\\
0.16	0.00533229651862177\\
0.2	0.00623612806645046\\
0.233	0.00730598568780505\\
0.26	0.00849956105201533\\
0.283	0.00984009900493632\\
0.302	0.0112618469670021\\
0.319	0.012865306227988\\
0.333	0.0145014899685885\\
0.346	0.0163594715250173\\
0.357	0.018263219757928\\
0.367	0.0203337460243067\\
0.376	0.0225488682958257\\
0.384	0.0248697342461004\\
0.391	0.0272378418597945\\
0.397	0.0295737732110737\\
0.403	0.0322517886321412\\
0.408	0.0347961062821638\\
0.413	0.0376795420036227\\
0.418	0.040966404493576\\
0.422	0.0439396611745331\\
0.426	0.0472725520937884\\
0.43	0.0510261665026746\\
0.434	0.0552749637952779\\
0.438	0.0601103289230174\\
0.441	0.0641887678569246\\
0.444	0.06871766150888\\
0.447	0.0737650990911489\\
0.45	0.0794119967163009\\
0.453	0.0857549052585236\\
0.456	0.0929094832455652\\
0.459	0.10101478749527\\
0.462	0.11023855883583\\
0.465	0.120783700086607\\
0.468	0.132896150261566\\
0.471	0.146874338711854\\
0.474	0.163080333284911\\
0.477	0.181952643612874\\
0.48	0.204020355766099\\
0.482	0.220815874099615\\
0.484	0.239527856188922\\
0.486	0.26039814129467\\
0.488	0.283694537192049\\
0.49	0.309710941866298\\
0.492	0.338766266768178\\
0.494	0.371201705429603\\
0.496	0.407375811097051\\
0.499	0.46945237914251\\
0.502	0.54198943717139\\
0.505	0.626070890379508\\
0.508	0.722466134366605\\
0.511	0.831431301219709\\
0.515	0.979398674531757\\
0.518	1.07628914675031\\
0.521	1.16089496487742\\
0.524	1.23341148403662\\
0.526	1.27523122582101\\
0.528	1.31205594752097\\
0.53	1.34414969029258\\
0.532	1.37183127044302\\
0.534	1.3954612970305\\
0.536	1.41542714727539\\
0.538	1.4321278148552\\
0.54	1.44596002831449\\
0.542	1.45730653516415\\
0.544	1.46652700434466\\
0.546	1.4739516424486\\
0.548	1.47987735809917\\
0.55	1.4845661412981\\
0.552	1.48824523817371\\
0.554	1.49110867918536\\
0.556	1.49331974199682\\
0.558	1.49501398179456\\
0.56	1.49630252755744\\
0.562	1.49727541193347\\
0.565	1.49829650199204\\
0.568	1.49894933427198\\
0.572	1.4994596357859\\
0.578	1.49980900380878\\
0.588	1.49996970835231\\
0.625	1.49999998841677\\
1	1.5\\
};
\addlegendentry{2nd order LxF}

\addplot [color=mycolor3, dotted, line width=2.5pt]
  table[row sep=crcr]{%
-1	0.5\\
-0.587	0.49969996732835\\
-0.578	0.498984369670082\\
-0.572	0.497820071656158\\
-0.568	0.496455004209523\\
-0.564	0.49434408307443\\
-0.561	0.492073948618376\\
-0.558	0.489018562225643\\
-0.555	0.484962382486468\\
-0.552	0.479653748079217\\
-0.549	0.472807728779391\\
-0.546	0.464112874187548\\
-0.543	0.453242380735977\\
-0.54	0.43986958043318\\
-0.537	0.423686702573786\\
-0.534	0.404424607627615\\
-0.53	0.373592181526772\\
-0.526	0.336584697450339\\
-0.522	0.293234811180438\\
-0.516	0.224692863349108\\
-0.512	0.18687165082843\\
-0.508	0.155368665908182\\
-0.504	0.129478331800903\\
-0.5	0.108394501933605\\
-0.496	0.0913174583635254\\
-0.492	0.0775159160338716\\
-0.488	0.0663556715896685\\
-0.484	0.0573063594857512\\
-0.48	0.0499355768099781\\
-0.476	0.0438969330315855\\
-0.472	0.0389161750517408\\
-0.467	0.0338540612739229\\
-0.462	0.029790328679987\\
-0.457	0.0264885647140571\\
-0.451	0.0232898191612205\\
-0.444	0.0203422177015662\\
-0.436	0.0177219067584189\\
-0.427	0.0154503282290293\\
-0.416	0.0133468302387338\\
-0.402	0.0113848345197369\\
-0.385	0.00969267209518399\\
-0.363	0.00819016523418581\\
-0.334	0.00689973901847507\\
-0.294	0.00582277268343834\\
-0.237	0.00500191769765213\\
-0.153	0.00450591892204377\\
-0.00600000000000001	0.00439805422225437\\
0.0940000000000001	0.00479176198643616\\
0.167	0.00571004210690695\\
0.223	0.00704400082594869\\
0.265	0.0086602858043201\\
0.298	0.0105499045515451\\
0.324	0.012655436515987\\
0.345	0.014968964463514\\
0.363	0.0176043382723312\\
0.378	0.0204708375789029\\
0.39	0.0233853975178655\\
0.401	0.026728295259689\\
0.41	0.0301071101928909\\
0.418	0.0337557396055932\\
0.425	0.0375951144670086\\
0.432	0.0422149644957914\\
0.438	0.0469713431570449\\
0.444	0.0526724309174234\\
0.449	0.0583379617459732\\
0.454	0.0650547469274265\\
0.459	0.0730963536376521\\
0.463	0.0807229388795634\\
0.467	0.0896654974786213\\
0.471	0.100225470967683\\
0.475	0.112784325687508\\
0.479	0.127825304468882\\
0.483	0.145960093096254\\
0.487	0.167960509021742\\
0.491	0.194794386617889\\
0.495	0.227663087889612\\
0.499	0.268035100469779\\
0.502	0.304283245778115\\
0.506	0.362176719395248\\
0.51	0.43302483133334\\
0.514	0.519124522756877\\
0.518	0.622621830506956\\
0.522	0.745103740727327\\
0.531	1.03935242254052\\
0.535	1.14383704902142\\
0.539	1.23091799967321\\
0.543	1.30147518633433\\
0.546	1.34441449816246\\
0.549	1.37967819352674\\
0.552	1.40816258596907\\
0.555	1.43080103624719\\
0.558	1.44851205890322\\
0.561	1.46215846100086\\
0.564	1.47251944712377\\
0.567	1.48027520812316\\
0.57	1.48600204776953\\
0.573	1.49017545367489\\
0.576	1.49317846774049\\
0.579	1.49531302875314\\
0.583	1.49720310972644\\
0.587	1.49836014946834\\
0.593	1.49928690823855\\
0.601	1.49977810504517\\
0.618	1.49998484295181\\
0.765	1.5\\
1	1.5\\
};
\addlegendentry{NT with \eqref{eq:slopesv1}}

\addplot [color=mycolor4, dashdotted, line width=2.5pt]
  table[row sep=crcr]{%
-1	0.5\\
-0.587	0.499700731703631\\
-0.578	0.498986782297003\\
-0.572	0.497824982585177\\
-0.568	0.496462686124346\\
-0.564	0.494355834024103\\
-0.561	0.492089865539327\\
-0.558	0.489039827845174\\
-0.555	0.484990390800409\\
-0.552	0.47969009195736\\
-0.549	0.472854164153167\\
-0.546	0.464171255961138\\
-0.543	0.453314568409239\\
-0.54	0.439957316359128\\
-0.537	0.423791472819459\\
-0.534	0.404547504049277\\
-0.53	0.373740066609456\\
-0.526	0.336757342744811\\
-0.522	0.293431010262414\\
-0.516	0.224857398494904\\
-0.512	0.187003597844795\\
-0.508	0.155471969335949\\
-0.504	0.129557412899499\\
-0.5	0.108453682120018\\
-0.496	0.0913606341055817\\
-0.492	0.0775464182684504\\
-0.488	0.0663762441617626\\
-0.484	0.0573192040922526\\
-0.48	0.0499424278522718\\
-0.476	0.0438991397598023\\
-0.472	0.0389147787451922\\
-0.467	0.0338492634024345\\
-0.462	0.0297830297743036\\
-0.457	0.0264794123992724\\
-0.451	0.0232790483681886\\
-0.444	0.0203301316857043\\
-0.436	0.0177088117442765\\
-0.427	0.0154364939435629\\
-0.416	0.013332433566396\\
-0.402	0.0113700260335845\\
-0.385	0.00967759809107771\\
-0.363	0.00817492137024201\\
-0.334	0.00688438585349638\\
-0.294	0.00580732427098951\\
-0.237	0.00498633014833527\\
-0.153	0.00449008952529573\\
-0.00600000000000001	0.00438190993610865\\
0.0940000000000001	0.00477372502808149\\
0.167	0.00568869875152522\\
0.223	0.00701864809356367\\
0.265	0.00863078488319236\\
0.298	0.0105162273924189\\
0.324	0.0126177348741268\\
0.345	0.0149274248030984\\
0.363	0.0175590096572664\\
0.378	0.0204219495511153\\
0.39	0.0233333779701073\\
0.401	0.0266731829987497\\
0.41	0.0300493259257419\\
0.418	0.0336955035915212\\
0.425	0.037532715107097\\
0.432	0.0421504480496828\\
0.438	0.0469051249052097\\
0.444	0.0526047214849905\\
0.449	0.0582692823846123\\
0.454	0.0649854878420182\\
0.459	0.0730271003941703\\
0.463	0.0806542976195357\\
0.467	0.089598220957686\\
0.471	0.100160580065086\\
0.475	0.112723200364695\\
0.479	0.127769803821145\\
0.483	0.14591271148613\\
0.487	0.167924574760292\\
0.491	0.19477430672963\\
0.495	0.227664635060288\\
0.499	0.268065723217844\\
0.502	0.304341736051429\\
0.506	0.362282159314558\\
0.51	0.433190347019308\\
0.514	0.51936493883116\\
0.518	0.622952573252615\\
0.522	0.74553878148161\\
0.531	1.03978749078122\\
0.535	1.14422029677775\\
0.539	1.23124367025929\\
0.543	1.30174211637949\\
0.546	1.34463917897506\\
0.549	1.37986372700928\\
0.552	1.40831299721869\\
0.555	1.4309208403051\\
0.558	1.44860588639224\\
0.561	1.46223076844898\\
0.564	1.47257431824798\\
0.567	1.48031623880133\\
0.57	1.48603229977885\\
0.573	1.49019745946113\\
0.576	1.49319426904721\\
0.579	1.4953242344657\\
0.583	1.497210065743\\
0.587	1.49836438044176\\
0.593	1.49928884374763\\
0.602	1.49980963697067\\
0.621	1.49999084408589\\
0.857	1.5\\
1	1.5\\
};
\addlegendentry{NT with \eqref{eq:slopesv2}}

\addplot [color=black, line width=1.5pt]
  table[row sep=crcr]{%
-1	0.5\\
0	0.5\\
0.00100000000000011	1.5\\
1	1.5\\
};
\addlegendentry{Initial Condition}

\end{axis}

\begin{axis}[%
width=0.951\fwidth,
height=0.2\fwidth,
at={(0\fwidth,0\fwidth)},
scale only axis,
xmin=-1,
xmax=1,
ymin=-1.5,
ymax=1.5,
xlabel={$x$},
ylabel={$u$},
ylabel style={rotate=-90},
axis background/.style={fill=white}
]
\addplot [color=mycolor1, dashdotted, line width=2.5pt, forget plot]
  table[row sep=crcr]{%
-1	-1\\
-0.683	-0.999684664054286\\
-0.659	-0.998947193888717\\
-0.643	-0.997849990415902\\
-0.63	-0.996352863418759\\
-0.618	-0.994292526575926\\
-0.608	-0.99194257283992\\
-0.598	-0.988904007876815\\
-0.589	-0.985498914174102\\
-0.58	-0.9813940557645\\
-0.571	-0.976539981485984\\
-0.562	-0.970902021682961\\
-0.553	-0.964461211821582\\
-0.544	-0.957214012611681\\
-0.535	-0.949171018090652\\
-0.525	-0.939328984679336\\
-0.515	-0.928577755030517\\
-0.504	-0.915767731969058\\
-0.492	-0.9007125062351\\
-0.479	-0.88325722533075\\
-0.465	-0.863285217539709\\
-0.45	-0.840720139056935\\
-0.433	-0.813914905668507\\
-0.414	-0.782680572805966\\
-0.392	-0.745169376284296\\
-0.367	-0.701183326358325\\
-0.338	-0.648798962332928\\
-0.303	-0.584168720437975\\
-0.26	-0.503310554174343\\
-0.206	-0.400290840616741\\
-0.134	-0.261415623801785\\
-0.0230000000000001	-0.0457151382554781\\
0.16	0.309882162090838\\
0.231	0.446361260902684\\
0.284	0.546871860405395\\
0.325	0.623274790690868\\
0.358	0.683457360856756\\
0.386	0.733199054824567\\
0.41	0.774506926552016\\
0.43	0.807696006171605\\
0.448	0.836346841316401\\
0.464	0.860621505044871\\
0.478	0.88076095055153\\
0.491	0.898390482701329\\
0.503	0.913622513798344\\
0.514	0.926607222729253\\
0.525	0.938569542423811\\
0.535	0.948493254946285\\
0.545	0.957463029519573\\
0.554	0.964693644520937\\
0.563	0.971115386577829\\
0.572	0.976732296661255\\
0.581	0.981563996378341\\
0.59	0.985645921973689\\
0.599	0.989028334684462\\
0.609	0.992042935208477\\
0.619	0.994371094869084\\
0.63	0.996266549681361\\
0.643	0.99779544441694\\
0.658	0.99886662845508\\
0.678	0.999576742588853\\
0.709	0.999926534881662\\
0.801	0.999999919629523\\
1	0.999999999999907\\
};
\addplot [color=mycolor2, dashed, line width=2.5pt, forget plot]
  table[row sep=crcr]{%
-1	-1\\
-0.548	-0.999841778586102\\
-0.539	-0.999462218724894\\
-0.533	-0.998873863996098\\
-0.528	-0.998015317007177\\
-0.524	-0.996979066854752\\
-0.52	-0.995537803358785\\
-0.517	-0.994138561011664\\
-0.514	-0.992429533022847\\
-0.511	-0.990383928310579\\
-0.508	-0.987983329106347\\
-0.505	-0.985219208246359\\
-0.502	-0.982093493530313\\
-0.498	-0.977385264889969\\
-0.494	-0.972107464117734\\
-0.49	-0.966329329374053\\
-0.485	-0.958521173808237\\
-0.479	-0.948486938665747\\
-0.472	-0.936135959631959\\
-0.462	-0.917745383118817\\
-0.448	-0.891203978783791\\
-0.428	-0.852520335471995\\
-0.397	-0.791761713464537\\
-0.349	-0.696871923625202\\
-0.272	-0.543843655188873\\
-0.145	-0.290635888251311\\
0.0580000000000001	0.114906925997193\\
0.234	0.466156750099471\\
0.34	0.677036049846957\\
0.398	0.791758602042782\\
0.432	0.858353942497871\\
0.453	0.898843402663404\\
0.467	0.92517077278323\\
0.476	0.941485842559398\\
0.483	0.953567822092577\\
0.489	0.963261360807397\\
0.494	0.970690257611543\\
0.498	0.976099530515961\\
0.502	0.980955693318389\\
0.506	0.985202315417523\\
0.509	0.987967162407955\\
0.512	0.990368806932528\\
0.515	0.992415727451889\\
0.518	0.994126267887041\\
0.521	0.995527130933567\\
0.525	0.996970569109782\\
0.529	0.998008843382547\\
0.533	0.998729398566106\\
0.538	0.999304130567705\\
0.545	0.999721487301693\\
0.555	0.999934454318146\\
0.579	0.999998877835311\\
1	1\\
};
\addplot [color=mycolor3, dotted, line width=2.5pt, forget plot]
  table[row sep=crcr]{%
-1	-1.00489475285522\\
-0.555	-1.00458362447431\\
-0.544	-1.00385434111753\\
-0.537	-1.00282164637409\\
-0.531	-1.00133808634021\\
-0.526	-0.999520304707471\\
-0.521	-0.997044696965876\\
-0.516	-0.993810953004788\\
-0.511	-0.989756727140497\\
-0.506	-0.984864329981429\\
-0.5	-0.977926012684039\\
-0.494	-0.969932054896108\\
-0.487	-0.959486323120941\\
-0.478	-0.944711723253447\\
-0.467	-0.925294742089875\\
-0.451	-0.895591902425382\\
-0.426	-0.847641763659599\\
-0.384	-0.765498685030774\\
-0.311	-0.621112179549375\\
-0.179	-0.358408478089592\\
0.0509999999999999	0.100922004962048\\
0.265	0.527730897299025\\
0.376	0.747849741743595\\
0.43	0.853647226278343\\
0.458	0.907241870298404\\
0.475	0.938469121913919\\
0.486	0.957469433892133\\
0.495	0.97179299423901\\
0.502	0.981848677292843\\
0.508	0.989509603665079\\
0.514	0.996136122424248\\
0.519	1.00078999941295\\
0.524	1.00462216149942\\
0.529	1.00764836266427\\
0.534	1.00993772611441\\
0.54	1.01186609590469\\
0.546	1.01310385291793\\
0.554	1.01402672542927\\
0.565	1.01454759443173\\
0.587	1.01474482813375\\
0.808	1.01475625165365\\
1	1.01475625165365\\
};
\addplot [color=mycolor4, dashdotted, line width=2.5pt, forget plot]
  table[row sep=crcr]{%
-1	-1.00489475285522\\
-0.555	-1.00458374409253\\
-0.544	-1.00385464879231\\
-0.537	-1.00282214982933\\
-0.531	-1.00133879887815\\
-0.526	-0.999521202920811\\
-0.521	-0.997045763890327\\
-0.516	-0.993812133808223\\
-0.511	-0.98975793928976\\
-0.506	-0.984865491616126\\
-0.5	-0.977927033084435\\
-0.494	-0.969932888654961\\
-0.487	-0.959486934203258\\
-0.478	-0.944712095874209\\
-0.467	-0.925294921156552\\
-0.451	-0.895591939072862\\
-0.426	-0.847641718677148\\
-0.384	-0.765498605045102\\
-0.311	-0.621112099118276\\
-0.179	-0.358408427585284\\
0.0509999999999999	0.100921988398019\\
0.265	0.527730756742141\\
0.376	0.747849445982345\\
0.43	0.853646863570513\\
0.458	0.907241636277309\\
0.475	0.938469283082334\\
0.486	0.957470162463269\\
0.495	0.971794414442049\\
0.502	0.981850714207826\\
0.508	0.989512121643258\\
0.514	0.996138956144479\\
0.519	1.0007928892295\\
0.524	1.00462488794543\\
0.529	1.00765074951732\\
0.534	1.00993968746336\\
0.54	1.01186753531109\\
0.546	1.01310483424373\\
0.554	1.01402725565123\\
0.565	1.01454778524169\\
0.587	1.01474484288634\\
0.808	1.01475625165365\\
1	1.01475625165365\\
};
\addplot [color=black, line width=1.5pt, forget plot]
  table[row sep=crcr]{%
-1	-1\\
0	-1\\
0.00100000000000011	1\\
1	1\\
};
\end{axis}

\begin{axis}[%
width=1.227\fwidth,
height=0.738\fwidth,
at={(-0.16\fwidth,-0.081\fwidth)},
scale only axis,
xmin=0,
xmax=1,
ymin=0,
ymax=1,
axis line style={draw=none},
ticks=none,
axis x line*=bottom,
axis y line*=left
]
\end{axis}
\end{tikzpicture}%

%% file: NLGarzNew.tex
%
%
\definecolor{mycolor1}{rgb}{0.06600,0.44300,0.74500}%
\definecolor{mycolor2}{rgb}{0.86600,0.32900,0.00000}%
\definecolor{mycolor3}{rgb}{0.92900,0.69400,0.12500}%
\begin{tikzpicture}

\begin{axis}[%
width=0.951\fwidth,
height=0.2\fwidth,
at={(0\fwidth,0.25\fwidth)},
scale only axis,
xmin=-1,
xmax=1,
ymin=0,
ymax=0.06,
ylabel={$\rho$},
ylabel style={rotate=-90},
axis background/.style={fill=white},
legend style={legend cell align=left, align=left, draw=white!15!black},
legend pos=outer north east
]
\addplot [color=mycolor1, dashdotted, line width=2.5pt]
  table[row sep=crcr]{%
-1.005	0.0499612575759676\\
-0.935	0.0499128647124694\\
-0.88	0.0498418832119853\\
-0.835	0.0497494787750496\\
-0.8	0.0496477631173957\\
-0.765	0.0495121762920492\\
-0.735	0.049362825094345\\
-0.705	0.0491769245911744\\
-0.68	0.0489897446363443\\
-0.655	0.0487693879817115\\
-0.63	0.0485122109718685\\
-0.61	0.048277538077143\\
-0.59	0.0480152112928565\\
-0.57	0.0477235289106355\\
-0.55	0.0474009006295559\\
-0.53	0.0470458753550351\\
-0.51	0.0466571675941085\\
-0.49	0.0462336819284648\\
-0.47	0.0457745351647816\\
-0.45	0.0452790758936175\\
-0.43	0.0447469013232356\\
-0.41	0.0441778713835954\\
-0.39	0.0435721202100454\\
-0.37	0.0429300652096052\\
-0.35	0.0422524139810747\\
-0.33	0.0415401694017539\\
-0.31	0.0407946332086466\\
-0.29	0.0400174083927176\\
-0.27	0.0392104006943763\\
-0.25	0.0383758194409674\\
-0.225	0.0372976753222718\\
-0.2	0.0361859732783054\\
-0.17	0.0348165811768792\\
-0.13	0.0329512773162273\\
-0.0600000000000001	0.0296755503618498\\
-0.03	0.0283104014784858\\
-0.00499999999999989	0.0272102450851626\\
0.0149999999999999	0.0263630408470916\\
0.0349999999999999	0.025551868127351\\
0.05	0.0249710215478296\\
0.0649999999999999	0.0244168839669305\\
0.0800000000000001	0.0238923147690862\\
0.095	0.0234002240298332\\
0.11	0.02294356693977\\
0.125	0.0225253378470931\\
0.14	0.0221485637408163\\
0.155	0.0218162968649762\\
0.17	0.0215316059804749\\
0.18	0.021369761584956\\
0.19	0.0212313383205645\\
0.2	0.0211172420781582\\
0.21	0.0210283722923275\\
0.22	0.020965618699025\\
0.23	0.0209298574445831\\
0.24	0.0209219464385568\\
0.25	0.0209427198414738\\
0.26	0.0209929815822469\\
0.27	0.0210734978100249\\
0.28	0.0211849882027635\\
0.29	0.0213281160809098\\
0.3	0.0215034773100604\\
0.31	0.0217115880216672\\
0.32	0.0219528712356358\\
0.33	0.022227642532086\\
0.34	0.0225360949899014\\
0.35	0.022878283684344\\
0.36	0.023254110111397\\
0.37	0.0236633069781997\\
0.38	0.0241054238618632\\
0.39	0.0245798142876168\\
0.4	0.0250856248061382\\
0.41	0.0256217866540398\\
0.42	0.0261870105568194\\
0.43	0.0267797851776359\\
0.44	0.0273983796275512\\
0.45	0.0280408503352045\\
0.465	0.0290445812543878\\
0.48	0.030089164474647\\
0.495	0.031165826686447\\
0.515	0.032635072985485\\
0.565	0.0363343913812713\\
0.58	0.037413197449141\\
0.595	0.0384626991055135\\
0.61	0.0394757826517649\\
0.625	0.040446358098039\\
0.635	0.0410672771152638\\
0.645	0.0416658499989726\\
0.655	0.0422410871966805\\
0.665	0.0427922126394653\\
0.675	0.0433186563719294\\
0.685	0.0438200449472976\\
0.695	0.0442961900149748\\
0.705	0.0447470755149391\\
0.715	0.045172843868744\\
0.725	0.0455737815238435\\
0.735	0.0459503041695171\\
0.745	0.04630294190144\\
0.755	0.0466323245702176\\
0.765	0.0469391675086601\\
0.78	0.0473589080385304\\
0.795	0.0477326097160176\\
0.81	0.0480633691861838\\
0.825	0.0483544063071852\\
0.84	0.0486089915404655\\
0.855	0.0488303833690751\\
0.87	0.0490217755334768\\
0.885	0.049186253668033\\
0.905	0.0493687543212187\\
0.925	0.0495151201913813\\
0.945	0.0496312451711594\\
0.97	0.0497418093906075\\
1	0.0498351873455203\\
1.005	0.0498474144909564\\
};
\addlegendentry{LxF}

\addplot [color=mycolor2, dashed, line width=2.5pt]
  table[row sep=crcr]{%
-1.5	0.05\\
-0.87	0.0499915190494553\\
-0.795	0.0499729536506346\\
-0.745	0.0499440364784967\\
-0.705	0.0499026505846398\\
-0.675	0.0498550468632932\\
-0.65	0.0498003418265527\\
-0.625	0.0497279372134223\\
-0.605	0.0496542681133469\\
-0.585	0.0495637803451903\\
-0.565	0.0494535817490023\\
-0.55	0.0493561206163604\\
-0.535	0.0492444160510146\\
-0.52	0.0491170203889573\\
-0.505	0.0489724517032617\\
-0.49	0.0488092117367656\\
-0.475	0.0486258056339379\\
-0.46	0.0484207629313995\\
-0.445	0.0481926591662873\\
-0.43	0.0479401373921691\\
-0.415	0.0476619288607303\\
-0.405	0.0474616075948771\\
-0.395	0.0472490371582241\\
-0.385	0.0470239299785016\\
-0.375	0.0467860266362423\\
-0.365	0.0465350972103473\\
-0.355	0.0462709423072427\\
-0.345	0.0459933937641632\\
-0.335	0.0457023150235765\\
-0.325	0.0453976011817832\\
-0.315	0.0450791787201421\\
-0.305	0.0447470049321212\\
-0.295	0.0444010670638069\\
-0.285	0.0440413811899749\\
-0.275	0.0436679908526025\\
-0.265	0.0432809654933139\\
-0.255	0.0428803987142536\\
-0.245	0.0424664064008449\\
-0.235	0.0420391247324863\\
-0.225	0.041598708093014\\
-0.215	0.0411453268751036\\
-0.205	0.0406791651600316\\
-0.195	0.0402004182580011\\
-0.185	0.0397092901238296\\
-0.175	0.0392059907132825\\
-0.16	0.0384286874952899\\
-0.145	0.0376252006265465\\
-0.13	0.0367962368815848\\
-0.115	0.0359424611394743\\
-0.0999999999999999	0.0350644478236037\\
-0.085	0.0341626035776121\\
-0.0700000000000001	0.0332370842615404\\
-0.0549999999999999	0.0322877485887065\\
-0.04	0.031314132775508\\
-0.0249999999999999	0.0303153121227531\\
-0.0149999999999999	0.0296345549126882\\
-0.00499999999999989	0.0289407140428823\\
0.00499999999999989	0.0282320273625278\\
0.0149999999999999	0.0275061462428778\\
0.0249999999999999	0.0267604262888215\\
0.0349999999999999	0.0259920788478332\\
0.0449999999999999	0.0251981983529623\\
0.0549999999999999	0.0243757455370415\\
0.0649999999999999	0.0235215676309841\\
0.0700000000000001	0.0230815940521096\\
0.075	0.0226325266706899\\
0.0800000000000001	0.0221740222377569\\
0.085	0.0217057893523542\\
0.105	0.0197930120189538\\
0.11	0.019333665284671\\
0.115	0.0188852891169229\\
0.12	0.0184482210854624\\
0.125	0.0180228023575439\\
0.13	0.0176093592992301\\
0.135	0.017208188207902\\
0.14	0.0168195449575892\\
0.145	0.0164436394228786\\
0.15	0.0160806336569086\\
0.155	0.0157306423765167\\
0.16	0.0153937342085999\\
0.165	0.0150699323929053\\
0.17	0.0147592142278148\\
0.175	0.0144615094011209\\
0.18	0.0141766982471259\\
0.185	0.0139046115972399\\
0.19	0.0136450339661409\\
0.195	0.0133977112727519\\
0.2	0.0131623630929651\\
0.205	0.0129386974984516\\
0.21	0.012726424809677\\
0.215	0.0125252663375042\\
0.22	0.0123349564022703\\
0.225	0.0121552386751311\\
0.23	0.0119858575365657\\
0.235	0.0118265438687672\\
0.24	0.0116769947378181\\
0.245	0.0115368469807045\\
0.255	0.01128282348332\\
0.265	0.0110593927028959\\
0.275	0.0108597075511552\\
0.29	0.010587727245793\\
0.33	0.00988384736846748\\
0.35	0.00949058265668068\\
0.355	0.00940706130578484\\
0.36	0.00933530308411235\\
0.365	0.00927706178317078\\
0.37	0.00923326927544754\\
0.375	0.00920384871199631\\
0.38	0.00918756358070327\\
0.39	0.0091830395415462\\
0.395	0.00919409344633326\\
0.4	0.00922440588680584\\
0.405	0.00928422012542929\\
0.41	0.00938415228450684\\
0.415	0.00953490762259546\\
0.42	0.00974712099464781\\
0.425	0.0100312894644143\\
0.43	0.0103977587802839\\
0.435	0.0108567106911119\\
0.44	0.0114181056626901\\
0.445	0.0120915664832495\\
0.45	0.0128862083984307\\
0.455	0.0138104217785073\\
0.46	0.0148716113984404\\
0.465	0.0160759098081784\\
0.47	0.0174279096913459\\
0.475	0.0189304857667285\\
0.48	0.0205847825748251\\
0.485	0.0223904211579993\\
0.49	0.0243459981046423\\
0.495	0.0264504382557094\\
0.5	0.0287075356218769\\
0.505	0.0310019124151892\\
0.51	0.0331608004189052\\
0.515	0.0351422751591057\\
0.52	0.0369439165348888\\
0.525	0.0385686369807192\\
0.53	0.0400236757347725\\
0.535	0.041319171694721\\
0.54	0.0424669256998402\\
0.545	0.0434794783479453\\
0.55	0.0443694616143064\\
0.555	0.045149161641737\\
0.56	0.0458302396463812\\
0.565	0.0464235687999983\\
0.57	0.0469391539598136\\
0.575	0.0473861086500074\\
0.58	0.0477726700177457\\
0.585	0.0481062376505186\\
0.59	0.0483934262154833\\
0.595	0.0486401249891455\\
0.6	0.0488515596465695\\
0.605	0.0490323533374848\\
0.61	0.0491865852410516\\
0.615	0.0493178455871062\\
0.62	0.0494292866599735\\
0.625	0.0495236696392802\\
0.63	0.0496034073391391\\
0.635	0.0496706030255438\\
0.64	0.0497270855525724\\
0.65	0.0498140416577746\\
0.66	0.0498745469670443\\
0.67	0.0499162070424231\\
0.685	0.0499551089624624\\
0.705	0.0499811355494331\\
0.735	0.0499952256338159\\
0.8	0.0499998184380102\\
1.5	0.05\\
};
\addlegendentry{2nd order LxF}

\addplot [color=mycolor3, dotted, line width=2.5pt]
  table[row sep=crcr]{%
-1.005	0.04999899978514\\
-0.83	0.0499824133007891\\
-0.76	0.0499507651746942\\
-0.71	0.0499018505977911\\
-0.67	0.0498343954072815\\
-0.64	0.049759077495122\\
-0.61	0.0496548645714057\\
-0.585	0.0495398937838112\\
-0.56	0.0493934108587366\\
-0.54	0.0492494087696198\\
-0.52	0.0490779416275942\\
-0.5	0.0488755346734473\\
-0.48	0.0486386712528999\\
-0.465	0.0484363162345494\\
-0.45	0.0482111972524841\\
-0.435	0.0479619513651699\\
-0.42	0.0476872968637827\\
-0.405	0.0473860521426501\\
-0.39	0.0470571525590491\\
-0.375	0.0466996647008342\\
-0.36	0.0463127975935644\\
-0.345	0.0458959105111121\\
-0.33	0.0454485171780992\\
-0.315	0.0449702863231296\\
-0.3	0.0444610386846496\\
-0.285	0.0439207406641562\\
-0.27	0.0433494948866986\\
-0.255	0.0427475280347982\\
-0.24	0.0421151765335945\\
-0.225	0.0414528692444278\\
-0.21	0.0407611075623422\\
-0.195	0.040040442972981\\
-0.18	0.039291452047884\\
-0.165	0.0385147080502062\\
-0.15	0.0377107455150918\\
-0.135	0.0368800102709215\\
-0.12	0.0360227872891232\\
-0.105	0.0351391113564361\\
-0.0900000000000001	0.0342286844181521\\
-0.075	0.033290802725493\\
-0.0600000000000001	0.0323242247999647\\
-0.0449999999999999	0.0313268089588721\\
-0.03	0.0302946356002181\\
-0.02	0.0295839772802129\\
-0.01	0.0288523588280416\\
0	0.0280971102477943\\
0.01	0.0273155653554229\\
0.02	0.0265051049150735\\
0.03	0.0256631815366648\\
0.04	0.0247873981335835\\
0.05	0.023875713967366\\
0.0600000000000001	0.0229268651192154\\
0.075	0.0214767457125771\\
0.085	0.0205525341998116\\
0.095	0.0196694329065119\\
0.105	0.0188314197855031\\
0.115	0.0180428442848188\\
0.12	0.0176683930140069\\
0.125	0.0173077688135947\\
0.13	0.0169613384211758\\
0.135	0.0166293829247928\\
0.14	0.0163120858651351\\
0.145	0.0160095239495575\\
0.15	0.0157216606915886\\
0.155	0.0154483432747763\\
0.16	0.0151893030039776\\
0.17	0.0147124331011814\\
0.18	0.0142868865336165\\
0.19	0.0139074239492252\\
0.2	0.0135684110565233\\
0.21	0.0132644734327625\\
0.22	0.0129909015166885\\
0.23	0.0127438393072612\\
0.24	0.0125204435220818\\
0.25	0.0123189139169859\\
0.26	0.0121380441642454\\
0.275	0.0119016998137818\\
0.295	0.0116305912769896\\
0.325	0.0112304959592011\\
0.335	0.0110718424432585\\
0.345	0.0109064667942147\\
0.355	0.0107757421795354\\
0.365	0.0106907166246561\\
0.375	0.0106530889215197\\
0.385	0.0106484668424554\\
0.395	0.0106741961179599\\
0.4	0.0107155810168684\\
0.405	0.0107875021256851\\
0.41	0.0108989242987441\\
0.415	0.0110588210451463\\
0.42	0.0112760775083995\\
0.425	0.0115594455821078\\
0.43	0.0119175178156981\\
0.435	0.0123586878658146\\
0.44	0.0128910808879581\\
0.445	0.0135224568005818\\
0.45	0.0142600948194567\\
0.455	0.0151106597926987\\
0.46	0.016080047952044\\
0.465	0.0171732223169805\\
0.47	0.0183940667245892\\
0.475	0.0197452983443007\\
0.48	0.0212284824933879\\
0.485	0.0228442064093217\\
0.49	0.0245925097851649\\
0.495	0.0264737346616213\\
0.5	0.0284899015844906\\
0.505	0.0305556113462624\\
0.51	0.0324986963493807\\
0.515	0.0342990933122143\\
0.52	0.0359556697475185\\
0.525	0.0374707380604991\\
0.53	0.0388491917489922\\
0.535	0.0400977407699508\\
0.54	0.0412242577668094\\
0.545	0.0422372564923301\\
0.55	0.0431454957717479\\
0.555	0.0439576896172855\\
0.56	0.044682302922437\\
0.565	0.0453274144100233\\
0.57	0.0459006312301364\\
0.575	0.0464090422563579\\
0.58	0.0468591996032544\\
0.585	0.0472571201011593\\
0.59	0.0476083003613261\\
0.595	0.0479177406339912\\
0.6	0.0481899739234672\\
0.605	0.0484290978143382\\
0.61	0.0486388072232826\\
0.615	0.0488224268637423\\
0.62	0.0489829426335737\\
0.625	0.049123031441882\\
0.63	0.0492450892078524\\
0.635	0.0493512569138268\\
0.645	0.0495233540101636\\
0.655	0.0496522200139622\\
0.665	0.0497480237337034\\
0.675	0.0498187274325717\\
0.69	0.0498908633762307\\
0.705	0.0499353452588953\\
0.725	0.0499686273898408\\
0.755	0.049989945949517\\
0.815	0.0499991434593392\\
1.005	0.0499999999225678\\
};
\addlegendentry{NT}

\addplot [color=black, line width=1.5pt]
  table[row sep=crcr]{%
-1.005	0.05\\
1.005	0.05\\
};
\addlegendentry{Initial Condition}

\end{axis}

\begin{axis}[%
width=0.951\fwidth,
height=0.2\fwidth,
at={(0\fwidth,0\fwidth)},
scale only axis,
xmin=-1,
xmax=1,
ymin=0.3,
ymax=0.8,
xlabel={$x$},
ylabel={$w$},
ylabel style={rotate=-90},
axis background/.style={fill=white}
]
\addplot [color=mycolor1, dashdotted, line width=2.5pt, forget plot]
  table[row sep=crcr]{%
-1.005	0.350000000788932\\
-0.465	0.350135410451607\\
-0.385	0.350427593849171\\
-0.33	0.350875379230414\\
-0.285	0.351510526720241\\
-0.25	0.352255580883168\\
-0.22	0.353132572627493\\
-0.19	0.354293289430864\\
-0.165	0.355529637187289\\
-0.14	0.357063153321747\\
-0.12	0.358541898508594\\
-0.0999999999999999	0.360279812267009\\
-0.0800000000000001	0.362312990241817\\
-0.0600000000000001	0.364681019224322\\
-0.0449999999999999	0.366702601688088\\
-0.03	0.368956228643258\\
-0.0149999999999999	0.371462483545092\\
0	0.374242973665608\\
0.0149999999999999	0.377320236451596\\
0.03	0.38071760406961\\
0.0449999999999999	0.384459017772867\\
0.0600000000000001	0.388568783178992\\
0.075	0.393071257415177\\
0.0900000000000001	0.397990459582413\\
0.0999999999999999	0.401512953541546\\
0.11	0.405237565536575\\
0.12	0.409170505143294\\
0.13	0.413317536908265\\
0.14	0.417683881584002\\
0.15	0.422274110926866\\
0.16	0.427092037189247\\
0.17	0.432140598824256\\
0.18	0.437421744332375\\
0.19	0.442936316599397\\
0.2	0.448683940482889\\
0.21	0.45466291677495\\
0.22	0.460870125973016\\
0.23	0.46730094549629\\
0.245	0.477352460250043\\
0.26	0.487865064798632\\
0.275	0.498802789329154\\
0.29	0.510121635708641\\
0.305	0.52177013994235\\
0.325	0.537713201479689\\
0.35	0.558088223236343\\
0.405	0.603163784311338\\
0.425	0.619143606079103\\
0.44	0.630844007672553\\
0.455	0.642244364857108\\
0.47	0.653301145289736\\
0.485	0.663977538321018\\
0.5	0.67424337582377\\
0.515	0.684074852479399\\
0.53	0.693454101423244\\
0.545	0.702368679325355\\
0.555	0.70804965551391\\
0.565	0.713519494992439\\
0.575	0.71877781563868\\
0.585	0.723824944752087\\
0.595	0.728661854883028\\
0.605	0.733290104383586\\
0.615	0.73771178288453\\
0.625	0.741929461672722\\
0.635	0.745946148761166\\
0.645	0.749765248307589\\
0.66	0.755131743976225\\
0.675	0.760076258719934\\
0.69	0.764614266780155\\
0.705	0.768762467688162\\
0.72	0.772538600045318\\
0.735	0.775961264868146\\
0.75	0.779049753339581\\
0.765	0.781823875474341\\
0.78	0.784303787850068\\
0.795	0.7865098200723\\
0.81	0.788462300954489\\
0.825	0.790181386452063\\
0.845	0.792144621039365\\
0.865	0.793773619273282\\
0.885	0.79511171057097\\
0.905	0.796199508906063\\
0.93	0.797264057051039\\
0.955	0.79806144637349\\
0.985	0.798744998763214\\
1.005	0.799073070391109\\
};
\addplot [color=mycolor2, dashed, line width=2.5pt, forget plot]
  table[row sep=crcr]{%
-1.5	0.35\\
-0.0449999999999999	0.349901645014581\\
-0.01	0.349691531043085\\
0.0149999999999999	0.349330360538108\\
0.03	0.348950421066447\\
0.0449999999999999	0.348372772152066\\
0.0549999999999999	0.347832411077634\\
0.0649999999999999	0.347124125408931\\
0.075	0.346196956395464\\
0.085	0.344982401126425\\
0.0900000000000001	0.344324793985281\\
0.095	0.343823133091996\\
0.105	0.343190460156515\\
0.115	0.342823852210734\\
0.125	0.342770074797239\\
0.135	0.34307273744363\\
0.14	0.343369543920679\\
0.145	0.343768650829029\\
0.15	0.344273484324303\\
0.155	0.344887127621908\\
0.16	0.345612413576621\\
0.165	0.346451962704569\\
0.17	0.347408121484244\\
0.175	0.348482787326063\\
0.18	0.349677155706561\\
0.185	0.350991471252551\\
0.19	0.352424882780655\\
0.195	0.353975479214636\\
0.2	0.355640530744953\\
0.205	0.35741690845223\\
0.215	0.361292450835432\\
0.225	0.365590225972418\\
0.23	0.36790003377302\\
0.235	0.370320967223654\\
0.24	0.37285612313974\\
0.245	0.375507590781192\\
0.25	0.378275654466473\\
0.26	0.384149690940633\\
0.27	0.390418479730006\\
0.28	0.396960564854611\\
0.295	0.407040667883552\\
0.3	0.410553935578626\\
0.31	0.417961946649368\\
0.32	0.425520574124394\\
0.325	0.429404695935556\\
0.33	0.433452994830559\\
0.335	0.437705659692501\\
0.34	0.442089026846654\\
0.345	0.446272363190064\\
0.35	0.449884901231349\\
0.355	0.452805558879698\\
0.36	0.455246932627269\\
0.365	0.457435579520604\\
0.37	0.459273306587371\\
0.375	0.46065389284839\\
0.385	0.462554596443529\\
0.39	0.46406489924974\\
0.395	0.466539680253594\\
0.4	0.470064493473948\\
0.405	0.474700272859063\\
0.41	0.480476687365932\\
0.415	0.4873895141386\\
0.42	0.49539891592071\\
0.425	0.50442869124978\\
0.43	0.514368325875873\\
0.435	0.525079139958204\\
0.44	0.536404074365121\\
0.445	0.54817910305575\\
0.455	0.572450202074376\\
0.465	0.596785987417482\\
0.475	0.620393600307932\\
0.48	0.631764033338067\\
0.485	0.642793929571833\\
0.49	0.653464012831187\\
0.495	0.663763834515319\\
0.5	0.673643431712488\\
0.505	0.685876745578734\\
0.51	0.701076324229312\\
0.515	0.715220034057889\\
0.52	0.727049300144574\\
0.525	0.737064456070809\\
0.53	0.745618730900941\\
0.535	0.752970765272893\\
0.54	0.759318315520191\\
0.545	0.764816965744592\\
0.55	0.769591475251742\\
0.555	0.773743477017954\\
0.56	0.777356984419453\\
0.565	0.780502392020132\\
0.57	0.783239402572356\\
0.575	0.785619186424375\\
0.58	0.787685991664376\\
0.585	0.789478359266286\\
0.59	0.791030052217684\\
0.595	0.792370776123795\\
0.6	0.793526746834099\\
0.605	0.794521145172141\\
0.61	0.795374487820997\\
0.615	0.7961049354751\\
0.625	0.797259537000104\\
0.635	0.798092181137464\\
0.645	0.798685948619964\\
0.655	0.799104558826665\\
0.67	0.799506280741271\\
0.69	0.7997848828556\\
0.72	0.799942744890697\\
0.785	0.799997599516509\\
1.5	0.8\\
};
\addplot [color=mycolor3, dotted, line width=2.5pt, forget plot]
  table[row sep=crcr]{%
-1.005	0.35\\
-0.0900000000000001	0.349874192617793\\
-0.05	0.349597930659838\\
-0.0249999999999999	0.349198703506203\\
-0.00499999999999989	0.348635457641675\\
0.01	0.347987772177923\\
0.0249999999999999	0.347058824626649\\
0.0349999999999999	0.346229679998439\\
0.0449999999999999	0.345184201703355\\
0.0549999999999999	0.343870193567505\\
0.0600000000000001	0.343134318645262\\
0.0649999999999999	0.342627748151929\\
0.075	0.341967949741325\\
0.085	0.341518827620276\\
0.095	0.341323632407516\\
0.105	0.341430801757251\\
0.115	0.341892856899319\\
0.125	0.342764491051219\\
0.135	0.344099332362833\\
0.145	0.345945874259874\\
0.15	0.347073669962437\\
0.155	0.348342963557201\\
0.16	0.349756410742986\\
0.165	0.351315525361904\\
0.17	0.353020519209954\\
0.18	0.356863209983235\\
0.19	0.361262259531094\\
0.2	0.366174672978592\\
0.21	0.371536321591448\\
0.22	0.377262754580603\\
0.235	0.38629417493486\\
0.255	0.398393026579569\\
0.265	0.404062307772837\\
0.28	0.412376993764244\\
0.285	0.415356434226247\\
0.29	0.418565562261933\\
0.295	0.422075282879778\\
0.3	0.425942415217506\\
0.31	0.434276066115878\\
0.315	0.438159789799698\\
0.325	0.445481793710271\\
0.335	0.452802215449546\\
0.34	0.456553171564529\\
0.345	0.459978779292621\\
0.35	0.463004355640887\\
0.355	0.465594660903454\\
0.36	0.467709116555425\\
0.365	0.469337974025708\\
0.37	0.470556838244619\\
0.38	0.472626413003965\\
0.385	0.474227588824741\\
0.39	0.476638433904175\\
0.395	0.479906824638247\\
0.4	0.484069905359968\\
0.405	0.489150229585698\\
0.41	0.495151984564827\\
0.415	0.502058797248073\\
0.42	0.509832269388325\\
0.425	0.518411012095426\\
0.43	0.527711092055967\\
0.435	0.537628998150547\\
0.44	0.548047175191962\\
0.45	0.569885078381186\\
0.47	0.614341990571561\\
0.48	0.635557469484133\\
0.485	0.645718825292145\\
0.49	0.655541486580874\\
0.495	0.665001165558724\\
0.5	0.674068256652072\\
0.505	0.684718031617092\\
0.51	0.698919145374643\\
0.515	0.711617369556898\\
0.52	0.722473996456875\\
0.525	0.731830201969907\\
0.53	0.739950984059049\\
0.535	0.747040999979461\\
0.54	0.753260234809159\\
0.545	0.758735846324903\\
0.55	0.763570525712066\\
0.555	0.767848415296955\\
0.56	0.771639383673054\\
0.565	0.775002168392507\\
0.57	0.777986714723833\\
0.575	0.780635933438901\\
0.58	0.78298703428415\\
0.585	0.785072546928708\\
0.59	0.786921109908022\\
0.595	0.788558086076208\\
0.6	0.790006047477814\\
0.605	0.791285161373859\\
0.615	0.79340729920183\\
0.625	0.7950482192101\\
0.635	0.796308125175686\\
0.645	0.797268211319744\\
0.655	0.79799407193533\\
0.67	0.798756003531245\\
0.685	0.799241855769559\\
0.705	0.79961868127583\\
0.735	0.799871495548915\\
0.79	0.799985250023387\\
1.005	0.799999999533031\\
};
\addplot [color=black, line width=1.5pt, forget plot]
  table[row sep=crcr]{%
-1.005	0.35\\
0	0.35\\
0.00499999999999989	0.8\\
1.005	0.8\\
};
\end{axis}

\begin{axis}[%
width=1.227\fwidth,
height=0.738\fwidth,
at={(-0.16\fwidth,-0.081\fwidth)},
scale only axis,
xmin=0,
xmax=1,
ymin=0,
ymax=1,
axis line style={draw=none},
ticks=none,
axis x line*=bottom,
axis y line*=left
]
\end{axis}
\end{tikzpicture}%